\documentclass[10pt]{amsart}

\usepackage[foot]{amsaddr}
\usepackage{amssymb}
\usepackage{amsthm}
\usepackage{amsmath}
\usepackage{mathrsfs}
\usepackage[usenames]{color}
\usepackage{enumitem}
\usepackage[normalem]{ulem}
\usepackage[backref=page]{hyperref}
\usepackage[capitalise]{cleveref}
\usepackage{thmtools, thm-restate}
\usepackage{tikz-cd}
\usepackage{float}
\usetikzlibrary{matrix, fit, arrows}

\hypersetup{
    colorlinks=true,
    linkcolor=blue,
    citecolor=blue
  }

\usepackage{todonotes} 
\usepackage{tcolorbox}

\makeatletter
\newcommand{\labelitem}[1]{%
\item[#1]\protected@edef\@currentlabel{#1}%
}
\makeatother

\theoremstyle{plain}
\newtheorem{theorem}{Theorem}[section]
\newtheorem{proposition}[theorem]{Proposition}
\newtheorem{lemma}[theorem]{Lemma}
\newtheorem{claim}{Claim}[theorem]
\newtheorem{corollary}[theorem]{Corollary}
\newtheorem{fact}[theorem]{Fact}

\newtheorem*{question*}{Question}

\newtheorem*{theorem*}{Theorem}

\theoremstyle{definition}
\newtheorem{example}[theorem]{Example}
\newtheorem{definition}[theorem]{Definition}

\theoremstyle{remark}
\newtheorem{remark}[theorem]{Remark}

\DeclareMathOperator{\Aut}{Aut}

\DeclareMathOperator{\diam}{diam}

\DeclareMathOperator{\SL}{SL}
\DeclareMathOperator{\GL}{GL}
\DeclareMathOperator{\PO}{PO}
\DeclareMathOperator{\SO}{SO}

\DeclareMathOperator{\PSL}{PSL}
\DeclareMathOperator{\PGL}{PGL}

\DeclareMathOperator{\End}{End}

\DeclareMathOperator{\Span}{Span}

\DeclareMathOperator{\Gr}{Gr}

\DeclareMathOperator{\image}{Im}

\DeclareMathOperator{\hil}{d_{\Omega}}

\DeclareMathOperator{\dproj}{d_{\mathbb{P}}}
\DeclareMathOperator{\domains}{\mathcal{C}}
\DeclareMathOperator{\pdomains}{\mathcal{C}_*}

\DeclareMathOperator{\orbitcl}{\mathscr{O}}

\DeclareMathOperator{\Cc}{\mathcal{C}}

\DeclareMathOperator{\Gc}{\mathcal{G}}

\DeclareMathOperator{\Kc}{\mathcal{K}}

\DeclareMathOperator{\Pc}{\mathcal{P}}

\DeclareMathOperator{\Hb}{\mathbb{H}}

\DeclareMathOperator{\Nb}{\mathbb{N}}
\DeclareMathOperator{\Pb}{\mathbb{P}}
\DeclareMathOperator{\Rb}{\mathbb{R}}

\DeclareMathOperator{\Zb}{\mathbb{Z}}

\DeclareMathOperator{\bdry}{\partial \Omega}

\DeclareMathOperator{\dist}{d}

\DeclareMathOperator{\Haus}{Haus}

\newcommand{\abs}[1]{\left|#1\right|}

\newcommand{\norm}[1]{\left\|#1\right\|}
\newcommand{\wt}[1]{\widetilde{#1}}

\newcommand{\dee}{\ensuremath{\partial}}
\newcommand{\conorm}[1]{\mathbf{m}(#1)}
\newcommand{\seq}[1]{\{#1\}}

\newcommand{\gtriple}[1][]{\mathcal{T}_{#1}}

\newcommand{\minus}{-}

\newcommand{\op}{\operatorname}

\newcommand{\set}[1]{\left\{ #1 \right\}}

\newcommand{\st}{:}

\newcommand{\identity}{\mathrm{id}}

\newcommand{\eps}{\ensuremath{\varepsilon}}

\newcommand{\PGLdR}{\mathrm{PGL}(d, \mathbb{R})}
\newcommand{\GLdR}{\mathrm{GL}(d, \mathbb{R})}
\newcommand{\SLdR}{\mathrm{SL}(d, \mathbb{R})}

\newcommand{\Pd}{\mathbb{P}^*}
\newcommand{\lift}[1]{\widetilde{#1}}
\newcommand{\Pspan}[1]{\mathrm{span}_{\mathbb{P}}\{#1\}}

\usepackage[color, final]{showkeys} 		
\definecolor{refkey}{rgb}{0.9,0,0.1}
\definecolor{labelkey}{rgb}{0.9,0,0.1}

\begin{document}

\title[Morse Properties in Convex Projective Geomerty]{Morse Properties in Convex Projective Geometry}
\author{Mitul Islam}
\email{mitul.islam@mis.mpg.de}
\author{Theodore Weisman}
\email{weisman@mis.mpg.de}

\address{Max Planck Institute for Mathematics in the Sciences,
  Inselstraße 22, 04103 Leipzig}

\date{\today}
\keywords{}
\subjclass[2010]{}

\begin{abstract}
  We study properties of ``hyperbolic directions'' in groups acting
  cocompactly on properly convex domains in real projective space,
  from three different perspectives simultaneously: the (coarse)
  metric geometry of the Hilbert metric, the projective geometry of
  the boundary of the domain, and the singular value gaps of
  projective automorphisms. We describe the relationship between
  different definitions of ``Morse'' and ``regular'' quasi-geodesics
  arising in these three different contexts. This generalizes several
  results of Benoist and Guichard to the non-Gromov hyperbolic setting.
\end{abstract}

\maketitle

\tableofcontents

\section{Introduction}

Group actions on $\Hb^n$ have classically played a pivotal role in the
study of discrete subgroups of Lie groups, geometric topology, and
geometric group theory. Hyperbolic geometry provides a
strong link between these fields, since hyperbolic manifolds (whose
holonomy representations have discrete images lying in
$\PO(n,1) \simeq \mathrm{Isom}(\Hb^n)$) give some of the most
important examples of geometric structures on manifolds, and the
properties of hyperbolic space are effectively coarsified via the
theory of Gromov-hyperbolic groups.

It is thus natural to try and find a substitute for hyperbolic
geometry which extends this connection beyond the negative curvature
setting. In particular, one would like to find a model geometry which
facilitates the study of discrete subgroups of higher-rank Lie groups,
such as $\SLdR$ for $d \ge 3$. One reasonable possibility to
consider is the non-positively curved Riemannian symmetric space
$\SLdR/\SO(d)$, but actions on the symmetric space are often
unsatisfyingly rigid. For instance, when $d \geq 3$, any discrete
Zariski-dense subgroup of $\SLdR$ which acts cocompactly on a
convex subset of $\SLdR/\SO(d)$ is a uniform lattice in
$\SLdR$ \cite{Q2005, KL2006}.

Convex projective geometry fills this gap by providing examples of
natural spaces---\emph{properly convex domains}---for discrete
subgroups of $\SLdR$ to act on. A properly convex domain $\Omega$ is
an open subset of real projective space $\Pb(\Rb^d)$, which is bounded
in some affine chart. Such a domain can be equipped with its
\emph{Hilbert metric} $\hil$, and the group
$\Aut(\Omega) \subset \PGLdR$ of projective transformations preserving
$\Omega$ acts by isometries of this metric. When $\Omega$ is a round
ball in $\Pb(\Rb^{d+1})$, then $\Aut(\Omega) \simeq \PO(d, 1)$ and the
space $(\Omega, \hil)$ is precisely the projective Beltrami-Klein
model of the real hyperbolic space $\Hb^d$. By virtue of this example,
convex projective geometry can be viewed as a far-reaching
generalization of real hyperbolic geometry.  This viewpoint has been
of much interest lately, and consequently convex projective geometry
has developed close connections with higher Teichm\"{u}ller theory
(see e.g. \cite{G1990, CG1993,GW2008,AW2018,B2005, BK2023}) and the theory of Anosov
representations \cite{DGK2017convex, Z2017}.

Outside of coarse negative curvature, however, convex projective
geometry can also be used to model examples which have a mixture of
``negatively curved'' and ``flat'' behavior. This allows for the study
of discrete subgroups of Lie groups which have markedly
\emph{different} behavior from those in rank one. For instance,
consider a closed 3-manifold $M$ with a geometric decomposition along
a nonempty collection of tori, into pieces whose interiors admit
finite-volume complete hyperbolic structures. Benoist \cite{B2006}
constructed examples of such 3-manifolds $M$ which are diffeomorphic
to a quotient $\Omega/\Gamma$, where $\Omega$ is a properly convex
domain and $\Gamma$ is a discrete subgroup of $\Aut(\Omega)$. In this
case, $\Gamma \simeq \pi_1M$, is relatively hyperbolic relative to
2-flats and the domain $\Omega$ is quasi-isometric to $\pi_1(M)$.

In this above example, the projective structure on the manifold
$\Omega / \Gamma$ is not ``non-positively curved,'' in the sense that
$\Omega$ equipped with its Hilbert metric $\hil$ is not a CAT(0)
metric space. In fact, a classical theorem states that a Hilbert
geometry $(\Omega,\hil)$ is CAT(0) if any only if $\Omega$ is
equivalent to the projective model of hyperbolic space
\cite{KS1958}. Despite this, convex projective domains share some
striking similarities with CAT(0) geometry.  Lately, there has been
much activity in understanding these similarities \cite{IZ2019,
  Islam2019, W2020, MB2020, IZ2021, PLB2021}. A key upshot of these
recent developments is the realization that irreducible properly
convex domains with a cocompact action essentially come in two
flavors: rank one and higher rank \cite{Islam2019,Z2020}. The
higher-rank domains are special, and have a complete
classification. On the other hand, many different examples of rank one
domains have been constructed \cite{B2006b, B2006, K2007, CLM2016,
  BDL2018, BV2023}, and their classification seems difficult. These domains
are characterized by the presence of abundant ``negatively curved
behavior'' (see \Cref{sec:rank_one}).

\subsubsection*{Summary of results}

Motivated by this, we initiate in this paper a study of ``hyperbolic''
geodesic directions in $(\Omega,\hil)$. We are mainly interested in
the case where there is a discrete subgroup
$\Gamma \subset \Aut(\Omega)$ acting cocompactly on $\Omega$; in this
situation, following Benz\'ecri \cite{benzecri} and Benoist
\cite{B2008}, we say that $\Gamma$ \emph{divides} $\Omega$ and that
$\Omega$ is \emph{divisible}. We consider \emph{projective} geodesic
rays in a divisible domain $\Omega$, i.e. geodesic rays
$c:[0, \infty) \to \Omega$ whose image is a projective line segment,
and sequences $\{\gamma_n\}$ in $\Gamma$ whose orbits give a coarse
approximation of $c$. In this paper we understand ``hyperbolicity'' of
these geodesic rays from three different perspectives:
\begin{enumerate}[label=(\Roman*), wide, labelwidth=!,labelindent=0pt]
\item \label{item:ggt_hyperbolicity} \textbf{The coarse metric geometry of the space
    $(\Omega, \hil)$.} In this context, there are two notions of
  ``hyperbolic geodesic'' which are relevant for this paper:
  \emph{Morse geodesics}, which are geodesics which satisfy the same
  ``Morse'' or ``quasi-geodesic stability'' property as geodesics in
  hyperbolic spaces, and \emph{contracting geodesics}, whose
  nearest-point projection maps have a similar ``contracting''
  property as hyperbolic geodesics. In CAT(0) spaces, Morse and
  contracting geodesics coincide; we prove that the same
  is true for Hilbert geometries
  (\Cref{thm:hilbert_morse_equals_contracting}).
  
\item \label{item:la_hyperbolicity}\textbf{The linear algebraic
    behavior of the sequence $\{\gamma_n\}$ in $\SLdR$.} Here, our
  understanding of ``hyperbolicity'' of a geodesic comes from results
  of Benoist \cite{B2004}, Bochi-Potrie-Sambarino \cite{BPS}, and
  Kapovich-Leeb-Porti \cite{KLP2017}, implying that a discrete group
  $\Gamma \subseteq \Aut(\Omega)$ acting cocompactly on $\Omega$ is
  Gromov-hyperbolic if and only if the singular values of $\Gamma$
  satisfy a \emph{uniform exponential gap} condition along all
  geodesics in $\Gamma$. Thus we understand ``hyperbolic directions''
  as geodesics in $\Gamma$ whose singular value gaps satisfy a similar
  exponential growth condition. This perspective is closely tied to
  the notion of a \emph{$k$-Morse} quasi-geodesic in the Riemannian
  symmetric space $\SLdR / \SO(d)$, introduced by Kapovich-Leeb-Porti
  \cite{KLP2018}.
  
\item \label{item:pg_hyperbolicity} \textbf{The projective geometry of
    the boundary of the domain $\Omega$.} Our motivation for this
  perspective comes from results of Benoist \cite{B2004} and
  \cite{guichard05}, which imply that, if $\Gamma$ is a discrete
  hyperbolic group dividing a domain $\Omega$, then the boundary of
  $\Omega$ is a $C^{\alpha}$ hypersurface in projective space for some
  $\alpha > 1$. From this perspective, the ``hyperbolicity'' of a
  geodesic ray $c$ in a general divisible domain $\Omega$ is captured
  by the local regularity of the hypersurface $\partial \Omega$ at the
  endpoint of $c$.
\end{enumerate}

Our main aim in this paper is to establish relationships between
geodesics satisfying various versions of these notions of
``hyperbolicity.'' Many of these relationships (in the case of a
projective geodesic $c$ in a convex divisible domain with
\emph{exposed boundary}, see \Cref{defn:exposed_boundary}) are
summarized in \Cref{fig:summary_diagram} below.

\begin{figure}[h!]
  \begin{center}
    \begin{tikzpicture}
    \tikzstyle{cond} = [
    rectangle, align=center, draw
    ];
    \matrix[column sep=2.2em, row sep=1em]
    {
      & & \node[align=center] (suf_1) {strongly unif.\\$1$-regular}; &
      \node[cond] (uf_1) {uniformly\\$1$-regular};&
      \node[cond] (B) {$\beta$-convex};
      \\
      \node[cond] (C) {contracting}; &
      \node[cond] (M) {Morse};&
      \node (+) {$+$};\\
      & \node[cond] (K) {$1$-Morse\\(KLP)}; & \node[align=center] (suf_2) {strongly
        unif.\\$(d-1)$-regular};&
      \node[cond] (uf_2) {uniformly\\$(d-1)$-regular};&
      \node[cond] (A) {$C^\alpha$-regular};\\
    };
    \node[draw, fit = (suf_1) (+) (suf_2)] (suf) {};
    \draw[implies-implies, double distance=2pt] (M) -- (C)
    node[midway, above]
    {\ref{thm:hilbert_morse_equals_contracting}}; 
    \draw[implies-implies, double distance=2pt] (M) -- (suf)
    node[midway, above,
    align=center]{\ref{thm:morse_implies_bdry_reg_and_unif_reg}\\
      \ref{thm:1_morse_ggt_morse}};
    \draw[-implies, double distance=2pt] (suf_1) -- (uf_1);
    \draw[-implies, double distance=2pt] (suf_2) -- (uf_2);
    \draw[implies-implies, double distance=2pt] (uf_1) -- (B)
    node[midway, below] {\ref{thm:regularity_equiv_regularity}};
    \draw[implies-implies, double distance=2pt] (uf_2) -- (A)
    node[midway, above] {\ref{thm:regularity_equiv_regularity}};
    \draw[implies-implies, double distance=2pt] (K) -- (suf);
  \end{tikzpicture}
\end{center}
\caption{Relationships between various notions of ``hyperbolicity''
  for a projective geodesic in a convex divisible domain with exposed
  boundary.}
\label{fig:summary_diagram}
\end{figure}

Before giving precise theorem statements in the next section, we
briefly outline the main results expressed by this
diagram. \Cref{thm:morse_implies_bdry_reg_and_unif_reg} and
\Cref{thm:1_morse_ggt_morse} relate perspectives
\ref{item:ggt_hyperbolicity} and \ref{item:la_hyperbolicity}
above. These theorems show that, if $c$ is a projective geodesic in a
divisible domain tracked by a sequence $\{\gamma_n\}$, then Morseness
of $c$ (in the sense of \ref{item:ggt_hyperbolicity}) is characterized
by the behavior of singular value gaps of the sequence
$\{\gamma_n\}$. In particular, this implies that for projective
geodesics, the Kapovich-Leeb-Porti notion of ``$1$-Morseness'' for
quasi-geodesics in symmetric spaces coincides with the coarse metric
notion of Morseness in \ref{item:ggt_hyperbolicity}
(\Cref{cor:morse_equals_morse}).

\Cref{thm:regularity_equiv_regularity} in the diagram directly relates
perspectives \ref{item:la_hyperbolicity} and
\ref{item:pg_hyperbolicity}. The theorem concerns projective geodesics
$c$ whose endpoint $c(\infty)$ in $\partial \Omega$ satisfies a
certain regularity property; roughly, this property asserts that if
$\partial \Omega$ is locally the graph of a convex function $f(x)$
near $c(\infty)$, then $f$ is sandwiched between $C_1 \norm{x}^\alpha$
and $C_2\norm{x}^{\beta}$ for some $\alpha>1$ and $\beta<\infty$. We
prove that this property is characterized by the behavior of the
singular values of the sequence $\{\gamma_n\}$, and give a formula for
the optimal constants $\alpha$ and $\beta$ in terms of these singular
values. Via the results alluded to in the previous paragraph, this
also relates perspectives \ref{item:ggt_hyperbolicity} and
\ref{item:pg_hyperbolicity}, and shows that every Morse quasi-geodesic
in the sense of \ref{item:ggt_hyperbolicity} satisfies the regularity
property mentioned above. However, the converse to this statement
turns out to be false (see
\Cref{thm:not_strongly_uniform}). Effectively, this theorem says that
the reverse of the implications ``strong uniform regularity''
$\implies$ ``uniform regularity'' in \Cref{fig:summary_diagram} do not
always hold.

\subsection*{Statement of the main results}

We now provide a more detailed and precise account of the main results
in the paper.

\subsection{$M$-Morse geodesics and uniform regularity}
\label{sec:M_Morse_singular_values_intro}
\begin{definition}
\label{defn:morse_gauge}
Let $(X, d)$ be a proper geodesic metric space and
$M:\Rb_{\ge 0} \times \Rb_{\ge 0} \to \Rb_{\ge 0}$ be any function. A
geodesic (segment, ray, line) $\ell$ is \emph{$M$-Morse} if any
$(\lambda, a)$-quasi-geodesic $\ell'$ with endpoints $x, y \in \ell$
lies in the $M(\lambda, a)$-neighborhood of $\ell$. The function $M$ is called a \emph{Morse gauge} for the geodesic
$\ell$. 
\end{definition}

At times, we will skip explicit mention of the Morse gauge and only say that $\ell$ is \emph{Morse}, instead of $\ell$ is $M$-Morse. Also, if $\ell$ is $M$-Morse for some Morse gauge $M$, then it is also $M'$-Morse for an \emph{increasing} Morse gauge $M'$, but this will not be relevant for us.

Geodesic rays in $\Hb^2$ are all $M_0$-Morse for a fixed Morse gauge
$M_0$. On the other hand, flat spaces like $\Rb^2$ and higher rank
CAT(0) spaces like $\Hb^2 \times \Hb^2$ or $\SL(3, \Rb)/\SO(3)$
do not contain any Morse geodesics.  The results in
\cite{Islam2019} and \cite {Z2020} indicate that a ``generic'' divisible domain $\Omega$ has
many Morse geodesics that project to closed geodesics in the quotient
$\Omega/\Gamma$ (see \cref{sec:rank_one}). It is also straightforward to check (see
\Cref{sec:morse_contracting}) that any $M$-Morse geodesic ray in a
convex projective domain $\Omega$ is uniformly close to a projective
geodesic ray. 

We would like to understand the Morseness of a projective geodesic ray
by studying the sequence of automorphisms that approximates the ray
via an orbit map. To make this precise, we use the following
terminology throughout this paper.

\begin{definition}[Tracking sequences]
  \label{defn:tracking}
  Let $\Omega$ be a properly convex domain, $c:[0, \infty) \to \Omega$
  be a projective geodesic ray, $x_0 \in \Omega$, and $R>0$.  We say
  that a sequence $\set{\gamma_n}_{n \in \Nb}$ in $\Aut(\Omega)$
  $R$-\emph{tracks $c$ with respect to $x_0$} if
  $\hil(x_0,\gamma_n^{-1}c(n))<R $ for every $n \in \Nb$.
\end{definition}
\begin{remark}
  While discussing tracking sequences, unless necessary we will omit
  the specific constant $R$ and the basepoint $x_0$, and simply say
  that $\{\gamma_n\}$ \emph{tracks} $c$. Note that if $\Omega/\Gamma$ is
  compact, then every geodesic is $R$-tracked by some sequence in
  $\Gamma$ for $R=\diam(\Omega/\Gamma)$. Also, by definition,
  $\{\gamma_n\}$ tracks $c$ along a sequence of equally-spaced points
  $\{c(n)\}$. One can consider other kinds of sequences, but we do not
  pursue this here.
\end{remark}

Now, for any $g \in \GLdR$, let $\sigma_1(g) \geq  \ldots \geq \sigma_d(g)$ be
the singular values of $g$, and for any $1 \le i < j \le d$, let
$\mu_{i,j}(g) := \log \frac{\sigma_i(g)}{\sigma_j(g)}$. Note that
$\mu_{i,j}$ descends to a well-defined map on $\PGLdR$.

\begin{definition}
  \label{defn:strong_unif_reg}
  For $1 \le k < d$, we say that a sequence $\{g_n\}$ in $\GLdR$ is
  \emph{uniformly $k$-regular} if it is divergent (it leaves
  every compact set in $\GLdR$) and
  \[
    \liminf_{n \to \infty} \frac{\mu_{k, k+1}(g_n)}{\mu_{1,d}(g_n)} >
    0.
  \]
  We say that the sequence $\{g_n\}$ is \emph{strongly uniformly
    $k$-regular} if it is divergent and there are constants $C, N > 0$
  such that for all $n,m \in \Nb$ with $m> N$, we have
  \[
    \frac{\mu_{k, k+1}(g_n^{-1}g_{n+m})}{\mu_{1,d}(g_n^{-1}g_{n+m})} > C.
  \]
\end{definition}

\begin{remark}\
\label{rem:comparison_with_KLP_stuff}
\begin{enumerate}[label=(\alph*)]
    \item It is immediate that a strongly uniformly $k$-regular sequence is
  also uniformly $k$-regular. In general, the converse does not hold;
  the construction in \Cref{sec:morse_counterexample} of this paper
  provides a counterexample. 
\item Note that our definition of ``uniform regularity''
    agrees with the definition given by Kapovich-Leeb-Porti in
    \cite{KLP2017,KL2018}, but \emph{not} with the definition of
    ``coarse uniform regularity'' in \cite{KLP2018}. In the special case
    where the sequence $\{g_n\}$ determines a quasi-isometric
    embedding $\Nb \to \PGL(d, \Rb)$, then our definition of
    \emph{strong} uniform regularity coincides with the notion of
    ``coarse uniform regularity'' defined in \cite{KLP2018}; in
    general our definition is weaker.
\end{enumerate}
\end{remark}

We prove the following:
\begin{theorem}[\cref{sec:sv_morse}]
\label{thm:morse_implies_bdry_reg_and_unif_reg}
  Suppose $\Omega$ is a properly convex domain,
  $c:[0,\infty) \to \Omega$ is a geodesic ray, and $\{\gamma_n\}$
  $R$-tracks $c$ with respect to $x_0 \in \Omega$. If $c$ is
  $M$-Morse, then $\{\gamma_n\}$ is strongly uniformly $k$-regular for
  both $k=1$ and $k=d-1$.

  Moreover, the constants $C, N$ in the definition of strong uniform
  regularity depend only on $M$, $R$, and the basepoint
  $x_0 \in \Omega$.
\end{theorem}

This theorem implies in particular that Morse geodesics give rise to
uniformly regular sequences. We express this via the corollary below.
\begin{corollary}
  \label{prop:morse_implies_weakly_uniformly_regular}
  Suppose that $\Omega$ is a properly convex domain, $c$ is a
  projective geodesic ray, and $\{\gamma_n\}$ $R$-tracks $c$ with
  respect to $x_0 \in \Omega$. For any Morse gauge $M$, there exists
  $\xi = \xi(M, R, x_0) > 1$ so that
  \[
    \liminf_{n \to \infty}
    \frac{\mu_{1,d}(\gamma_n)}{\mu_{1,d-1}(\gamma_n)} > 1+ \frac{1}{\xi-1} \quad \text{ and } \quad 
    \limsup_{n \to \infty}
    \frac{\mu_{1,d}(\gamma_n)}{\mu_{1,2}(\gamma_n)}<\xi.
  \]
\end{corollary}

If we impose additional assumptions on the domain $\Omega$, then a
partial converse to \Cref{thm:morse_implies_bdry_reg_and_unif_reg}
also holds. Specifically, we have the following:
\begin{restatable}[\cref{sec:sv_morse}]{theorem}{KLPMorseGGTMorse}
  \label{thm:1_morse_ggt_morse}
  Let $\Omega$ be a convex divisible domain with exposed boundary and
  let $c$ be a projective geodesic ray in $\Omega$. Suppose
  $\{\gamma_n\}$ $R$-tracks $c$ with respect to $x_0 \in \Omega$. If
  $\{\gamma_n\}$ is strongly uniformly k-regular for $k=1$ and
  $k=d-1$, then $c$ is $M$-Morse for some Morse gauge $M$.

  Moreover, $M$ can be chosen to depend only on $x_0$, $R$, and the
  constants in the definition of strong uniform $k$-regularity.
\end{restatable}

We provide the precise definition of a convex projective domain with
\emph{exposed boundary} in \Cref{defn:exposed_boundary}. The
additional assumptions on $\Omega$ in \Cref{thm:1_morse_ggt_morse} are
necessary, as is the assumption that $c$ is a \emph{projective}
geodesic; see \Cref{ex:nondivisible_nonmorse} and
\Cref{ex:triangle_nonmorse}.

Together, \Cref{thm:morse_implies_bdry_reg_and_unif_reg} and
\Cref{thm:1_morse_ggt_morse} show that, when $\Omega$ is a convex
divisible domain with exposed boundary, it is possible to completely
characterize (projective) Morse geodesics in terms of the singular
values of tracking sequences. We also show that the weaker uniform
regularity condition in
\Cref{prop:morse_implies_weakly_uniformly_regular} does \emph{not}
imply Morseness:
\begin{theorem}[\cref{sec:morse_counterexample}]
  \label{thm:not_strongly_uniform}
  There exists a convex divisible domain $\Omega$ with exposed
  boundary, a projective geodesic ray $c$, and a sequence
  $\{\gamma_n\}$ tracking $c$ so that $\{\gamma_n\}$ is both uniformly
  $1$-regular and uniformly $(d-1)$-regular, but \emph{not} strongly
  uniformly $1$-regular. In particular, due to
  \Cref{thm:morse_implies_bdry_reg_and_unif_reg}, $c$ is \emph{not}
  $M$-Morse for any Morse gauge $M$.
\end{theorem}

\subsubsection{$k$-Morseness in symmetric spaces}

In \cite{KLP2018}, Kapovich-Leeb-Porti developed a notion of
``Morseness'' for quasi-geodesics in certain symmetric spaces. If $X$
is the Riemannian symmetric space $\PGLdR / \PO(d)$, then a
quasi-geodesic ray $q:[0, \infty) \to X$ is \emph{$k$-Morse in the
  sense of Kapovich-Leeb-Porti} if it satisfies a ``higher-rank Morse
property'' with respect to the Grassmannian of $k$-planes $\Gr(k, d)$,
viewed as a space of simplices in the visual boundary of $X$. This
property says that $q$ lies in a bounded neighborhood of a Euclidean
Weyl sector asymptotic to a $k$-plane in $\mathrm{Gr}(k, d)$.

In the same paper, Kapovich-Leeb-Porti proved a \emph{higher-rank Morse
  lemma}, characterizing Morse quasi-geodesics in terms of their
uniform regularity (cf. \ref{defn:strong_unif_reg} and \ref{rem:comparison_with_KLP_stuff}). Combining this result with
\Cref{thm:morse_implies_bdry_reg_and_unif_reg} and
\Cref{thm:1_morse_ggt_morse}, we obtain the following:
\begin{corollary}
  \label{cor:morse_equals_morse}
  Let $\Omega$ be a convex divisible domain with exposed boundary, let
  $c:[0, \infty) \to \Omega$ be a projective geodesic, and let
  $\{\gamma_n\}$ be a sequence in $\Aut(\Omega)$ which tracks
  $c$. Then $c$ is a Morse geodesic if and only if the mapping
  $\Nb \to \PGLdR / \PO(d)$ given by $n \mapsto \gamma_n\PO(d)$ is a
  $1$-Morse quasi-geodesic in the sense of Kapovich-Leeb-Porti.
\end{corollary}

\begin{remark} If $\{g_n\}$ is a sequence in $\PGL(d, \Rb)$
  determining a $1$-Morse quasi-geodesic, then the sequence
  $\{\mu_{1,2}(g_n)\}$ satisfies the \emph{CLI condition} of
  Gu\'eritaud-Guichard-Kassel-Wienhard \cite{ggkw2017anosov} (as a
  consequence of the above-mentioned higher rank Morse lemma). So, the
  corollary above shows that Morse geodesics in $\Omega$ determine CLI
  sequences. One can also prove this directly (without invoking the
  higher-rank Morse lemma) by applying \Cref{lem:geodesic_additivity}
  in the present paper.
\end{remark}

\subsection{Uniform regularity and boundary regularity}
\label{sec:morse_implies_bdry_regular}

We now relate the singular value gap conditions appearing in
\Cref{thm:morse_implies_bdry_reg_and_unif_reg} and
\Cref{prop:morse_implies_weakly_uniformly_regular} to the smoothness
or regularity of the boundary $\partial \Omega$ at the endpoint of a
projective geodesic. The boundary $\partial \Omega$ is a convex
hypersurface in $\Pb(\Rb^d)$, meaning it is locally the graph of a
convex function $f:\Rb^{d-2} \to \Rb$. In general, this hypersurface could be nowhere $C^1$. For instance, Benoist constructed closed 3-manifolds $\Omega/\Gamma$ for which $\partial \Omega$ has a dense set of non-differentiable points (\cite{B2006}; also see \cref{sec:description_of_benoist_domains}). But even when this occurs, we can still make sense of local regularity using
convexity.

We say that a point $z \in \bdry$ is a \emph{$C^1$ point} if there is
a unique supporting hyperplane of $\Omega$ at $z$, i.e. a hyperplane
containing $z$, but not intersecting $\Omega$. Note that if the function $f$ defining $\bdry$ near $z$ is $C^1$, then $z$ is indeed a $C^1$ point; however the converse is not always true.  At a $C^1$ point $z$,
we further have a local notion of $\alpha$-H\"older regularity.  We
say that $z$ is a \emph{$C^\alpha$ point} if there exist Euclidean
coordinates on an affine chart in $\Pb(\Rb^d)$ such that $\bdry$ is
the graph of a convex function $f$, $(z,f(z))$ is the origin, and
$f(x) \le C_1||x||^\alpha$ for some $C_1 > 0$ and all $x$ sufficiently
close to $z$.  Dually, we say that $z$ is a \emph{$\beta$-convex}
point if $f(x) \ge C_2||x||^\beta$ for some $C_2 > 0$ and all $x$
sufficiently close to $z$.

\begin{restatable}{definition}{defnbdryregularity}
  \label{defn:alpha_beta_point}
  Let $\Omega$ be a properly convex domain and $x \in \bdry$ be a $C^1$ point. Set
  \begin{align*}
    \alpha(x, \Omega) := \sup\{\alpha > 1 : \dee \Omega \textrm{ is
    }C^\alpha\textrm{ at }x\}
    \end{align*}
  and 
  \begin{align*}
    \beta(x, \Omega) := \inf\{\beta < \infty : \dee \Omega \textrm{ is
    }\beta\textrm{-convex at }x\}.
  \end{align*}
  If $\dee \Omega$ is not $C^\alpha$ at $x$ for any $\alpha > 1$, we
  define $\alpha(x, \Omega) = 1$. Similarly if $\dee \Omega$ is not
  $\beta$-convex at $x$ for any $\beta < \infty$, we define
  $\beta(x, \Omega) = \infty$.
\end{restatable}

We show that for a divisible domain $\Omega$, the functions $\alpha(x,\Omega)$ and $\beta(x,\Omega)$ defined above are determined by singular values of tracking sequences.  To state our result, we require $x$ to be an exposed boundary point; see \cref{defn:exposed_boundary} and \cref{fig:defn_h}.

\begin{restatable}[\cref{sec:boundary_regularity_sv}]{theorem}{boundaryRegularity}
  \label{thm:regularity_equiv_regularity}
  Let $\Omega$ be a properly convex domain, let $\{\gamma_n\}$ track a
  projective geodesic ray $c:[0, \infty) \to \Omega$, and suppose that
  $c(\infty) = x$ is an exposed $C^1$ extreme point in $\dee
  \Omega$. Define
  \begin{align*}
    \alpha _0:= \liminf_{n \to \infty}
    \frac{\mu_{1,d}(\gamma_n)}{\mu_{1,d-1}(\gamma_n)} \quad \text{ and } \quad 
    \beta_0 := \limsup_{n \to \infty}
    \frac{\mu_{1,d}(\gamma_n)}{\mu_{1,2}(\gamma_n)}.
  \end{align*}
  Then $\alpha_0 = \alpha(x, \Omega)$ and $\beta_0 = \beta(x,
  \Omega)$.

  In particular, $c(\infty)$ is a $C^\alpha$ point for some
  $\alpha > 1$ if and only if $\{\gamma_n\}$ is uniformly
  $(d-1)$-regular, and $c(\infty)$ is $\beta$-convex for
  $\beta < \infty$ if and only $\{\gamma_n\}$ is uniformly
  $1$-regular.
\end{restatable}

An immediate consequence of
\Cref{prop:morse_implies_weakly_uniformly_regular} and
\cref{thm:regularity_equiv_regularity} is the following, which is our
link between perspectives \ref{item:ggt_hyperbolicity} and
\ref{item:pg_hyperbolicity} in this paper:
\begin{corollary}
  \label{cor:morse_regular_convex}
  Suppose $\Omega$ is a properly convex domain and $\{\gamma_n\}$
  $R$-tracks a projective geodesic ray $c$ with respect to
  $x_0 \in \Omega$. If $c$ is $M$-Morse, then $c(\infty)$ is
  $C^\alpha$ and $\beta$-convex for some $\alpha>1, \beta<\infty$, depending only on $M$, $R$, and $x_0$. Moreover 
  \begin{align*}
    \alpha(c(\infty),\Omega) = \liminf_{n \to \infty}
    \frac{\mu_{1,d}(\gamma_n)}{\mu_{1,d-1}(\gamma_n)} > 1 ~\text{ and }~
    \beta(c(\infty),\Omega)= \limsup_{n \to \infty}
    \frac{\mu_{1,d}(\gamma_n)}{\mu_{1,2}(\gamma_n)}< \infty.
  \end{align*}
\end{corollary}

By applying \Cref{thm:not_strongly_uniform} and
\Cref{thm:regularity_equiv_regularity}, we can also see that the
converse to \Cref{cor:morse_regular_convex} does not hold:
\begin{corollary}
  There exists a convex divisible domain $\Omega$ with exposed
  boundary and a projective geodesic ray $c$ so that $c$ is \emph{not}
  $M$-Morse for any Morse gauge $M$, but $c(\infty)$ is both a
  $C^\alpha$ point for some $\alpha > 1$ and a $\beta$-convex point
  for some $\beta < \infty$.
\end{corollary}

\subsection{$D$-contracting geodesics and Morse local-to-global
  spaces}
\label{sec:contracting_morse_equivalence}

We now mention a few additional results regarding notions of
``coarsely negatively curved'' geodesic directions in $\Omega$. Recall
that geodesics in hyperbolic metric spaces always satisfy a
\emph{contracting} property, which motivates the following definition:
\begin{definition}
  \label{defn:D_contracting}
  Let $(X, d)$ be a proper metric space, $\ell$ a geodesic (ray,
  segment, line), and let $\pi_{\ell}:X \to 2^\ell$ denote the
  closest-point projection on $\ell$, i.e.
  \[
    \pi_{\ell}(x) = \{y \in \ell : d(x, y) = d(x, \ell)\}.
  \]
  Then $\ell$ is \emph{$D$-contracting} for $D > 0$ if, for any metric
  ball $B_r(x)$ disjoint from $\ell$,
  \[
    \diam(\pi_{\ell}(B_r(x))) \le D.
  \]
  If $\ell$ is $D$-contracting for some $D > 0$, then we simply say
  that $\ell$ is \emph{contracting.}
\end{definition}

A result of Charney-Sultan \cite{CS2015} implies that, if $X$ is a
CAT(0) metric space, then contracting geodesics are exactly the same
as Morse geodesics. We prove:
\begin{theorem}[\cref{sec:morse_contracting}]
  \label{thm:hilbert_morse_equals_contracting}
  Let $\Omega$ be a properly convex domain, and let $c$ be a geodesic
  in $\Omega$. Then $c$ is Morse if and only if $c$ is contracting.
\end{theorem}

It is well-known that in general metric spaces, every
  $D$-contracting geodesic is $M$-Morse for some Morse gauge $M$
  depending only on $D$. So, our main contribution in
  \Cref{thm:hilbert_morse_equals_contracting} is proving the converse,
  i.e. that Morse geodesics are always contracting. Our proof for this
  direction relies on specific features of the projective geometry of
  $\Omega$.

  When we prove this direction, we do \emph{not} in general obtain
  uniform control over the contraction constant $D$ in terms of the
  Morse gauge $M$. However, we do have uniform control if we
  additionally assume that $\Omega$ has a cocompact action by
  automorphisms; see \Cref{cor:morse_contracting_uniform}.
  
  As an application of this uniform control, we obtain that divisible convex domains have the so-called \emph{Morse local-to-global property} \cite{RST22}.  As the name suggests, this property means
the following. Suppose $c$ is a path in a metric space $X$, such that
any sufficiently long finite sub-segment of $c$ is an $M$-Morse
quasi-geodesic. Then the entire path $c$ is an $M'$-Morse
quasi-geodesic, for some Morse gauge $M'$. This property avoids certain pathologies which can arise when considering arbitrary Morse geodesics in general metric spaces, and it is known to have a number of nice consequences; see e.g. \cite{RST22,HSZ24,CRSZ22,MS24}, or the survey in the appendix of \cite{DSZ25}. We show the following.
\begin{theorem}
  \label{thm:morse_local_to_global}
  If $\Omega$ is any convex divisible domain, the metric space
  $(\Omega, \hil)$ is Morse local-to-global.
\end{theorem}
This result is an immediate consequence of \Cref{cor:morse_contracting_uniform} and a result of Sisto-Zalloum \cite[Prop. 4.7]{SZ22}. Indeed, \cref{cor:morse_contracting_uniform} and \cref{prop:contracting_implies_morse} imply that in convex divisible domains, Morse quasi-geodesics are \emph{strongly contracting quantitatively} in the sense of \cite{SZ22}. Then \cite[Prop 4.7]{SZ22} implies the Morse local-to-global property; we are grateful to Alex Sisto for pointing this out to us.

\subsection{Comparison to previous results}

\subsubsection{The Gromov-hyperbolic case}
Several of the results in this paper are inspired by previous work of
Benoist and Guichard on convex divisible domains. In particular, we are
motivated by the following theorem:
\begin{theorem}[\cite{B2004}]
  \label{thm:hyperbolic_divisibles}
  Let $\Gamma$ be a group dividing a properly convex domain
  $\Omega \subset \Pb(\Rb^d)$. Then the following are equivalent:
  \begin{enumerate}[label=(\alph*)]
  \item $\Gamma$ is Gromov-hyperbolic;
  \item\label{item:anosov} The inclusion
    $\Gamma \hookrightarrow \PGL(d, \Rb)$ is a \emph{$1$-Anosov
      representation}.
  \item The domain $\Omega$ is strictly convex, i.e. its boundary
    $\dee \Omega$ contains no nontrivial projective segments;
  \item The boundary $\dee \Omega$ is a $C^1$ hypersurface.
  \end{enumerate}
\end{theorem}

We can interpret this theorem as giving a link between our
perspectives \ref{item:ggt_hyperbolicity}, \ref{item:la_hyperbolicity}
\ref{item:pg_hyperbolicity} in the hyperbolic setting. In particular,
if $\Gamma$ is a Gromov-hyperbolic group, then every geodesic in
$\Gamma$ is $M_0$-Morse for some uniform Morse gauge $M_0$, so every
geodesic direction in $\Gamma$ is ``hyperbolic'' in the sense of our
perspective \ref{item:ggt_hyperbolicity}. Part \ref{item:anosov} of
the theorem connects to perspective \ref{item:la_hyperbolicity} via
work of Kapovich-Leeb-Porti \cite{KLP2017} and Bochi-Potrie-Sambarino
\cite{BPS}, who poved that a representation $\Gamma \to \PGL(d, \Rb)$
is $1$-Anosov if and only if it is a quasi-isometric embedding and
every geodesic is strongly uniformly $1$-regular. From this viewpoint,
our \Cref{thm:morse_implies_bdry_reg_and_unif_reg},
\Cref{thm:1_morse_ggt_morse}, and \Cref{cor:morse_regular_convex} give
a generalization of \Cref{thm:hyperbolic_divisibles} to the situation
where $\Gamma$ is not necessarily a Gromov-hyperbolic
group. Effectively, we prove that the equivalences in
\Cref{thm:hyperbolic_divisibles} still hold locally, ``along
hyperbolic directions.''

\begin{remark}
  There are explicit constructions for divisible domains in
  $\Pb(\Rb^d)$ that are \emph{not} strictly convex (so none of the
  conditions in \Cref{thm:hyperbolic_divisibles} hold), but still
  contain Morse geodesics (so that our main results apply); for examples, see
  \cite{B2006, CLM2016, BDL2018} when $4 \le d \le 7$,
  and \cite{BV2023} for any $d \ge 4$.
\end{remark}

In \cite{guichard05}, Guichard also investigated the relationship
between regularity of the boundary of a strictly convex divisible
domain $\Omega$, and the linear algebraic properties of the dividing
group $\Gamma$. In particular, Guichard showed that, assuming $\Gamma$
is a hyperbolic group, the \emph{global} H\"older regularity of
$\bdry$ can be computed in terms of the asymptotic behavior of the
eigenvalues of sequences in $\Gamma$; this provides another link
between our perspectives \ref{item:la_hyperbolicity} and
\ref{item:pg_hyperbolicity}, again in the case where $\Gamma$ is
assumed to be a hyperbolic group. We also remark that Crampon \cite{C2009} has related work, which shows that for \emph{strictly} convex
divisible domains, the regularity of the boundary $\Omega$ is related
to the Lyapunov exponent of the geodesic flow. Our Theorem
\ref{thm:regularity_equiv_regularity} can be thought of as a localized
version of Guichard's result, which applies to geodesic directions in
any (not necessarily strictly convex) divisible domain.

\subsubsection{Closed geodesics in rank-one convex projective
  manifolds}
\label{sec:rank_one}
In \cite{Islam2019}, the first author introduced a notion of rank one
properly convex domains, a family that encompasses the Gromov
hyperbolic ones. An infinite order element $\gamma \in \Aut(\Omega)$
is called a \emph{rank one automorphism} if $\gamma$ acts by a
translation along a projective geodesic
$\ell_{\gamma} \subset \Omega$, and $\ell_{\gamma}$ is not contained
in any half triangle (see \cref{defn:half_triangle}). The results in
\cite{Islam2019} show that the axis $\ell_\gamma$ of a rank one
automorphism is always a Morse geodesic, and also characterize
rank-one automorphisms in terms of their eigenvalues. By combining
this with results in the present paper, we obtain the following more
complete description of the relationship between rank one
automorphisms and Morseness.
\begin{proposition}
 \label{prop:rank_one_equiv_morse}
 Suppose an infinite order element $\gamma \in \Aut(\Omega)$ acts by a
 translation along a projective geodesic
 $\ell_{\gamma} \subset \Omega$. Then the following are equivalent:
    \begin{enumerate}
        \item $\gamma$ is a rank one automorphism. 
        \item The geodesic $\ell_{\gamma}$ is Morse (equivalently,
          $\ell_\gamma$ is contracting).
    \end{enumerate}
    If, in addition, there is a discrete group $\Gamma \subseteq \Aut(\Omega)$
    such that $\Gamma$ divides $\Omega$ and $\gamma \in \Gamma$, then
    either of the above conditions is equivalent to:
\begin{enumerate}[resume]
\item $\gamma$ is biproximal, i.e. the matrix representing $\gamma$
  has unique eigenvalues of maximum and minimum modulus.
\end{enumerate}
\end{proposition}

\begin{proof}\ (1)$\implies$(2) is \cite[Proposition
  1.12]{Islam2019}. (2)$\implies$(1) follows from the results in
  \cref{sec:morse_contracting} of this paper: since $\ell_\gamma$ is
  Morse, \cref{cor:morse_not_in_segment} implies that the endpoints of
  $\ell_\gamma$ cannot lie in the closure of a non-trivial projective
  line segment in $\bdry$. Thus $\ell_\gamma$ is not contained in any
  half triangle.
    
      Finally, (1)$\iff$(3) is \cite[Proposition 6.8]{Islam2019}.
\end{proof}

If $\Gamma \subseteq \Aut(\Omega)$ is a discrete subgroup that
contains a rank one automorphism, then we say that the pair
$(\Omega,\Gamma)$ is \emph{rank one}. \cref{prop:rank_one_equiv_morse}
says that $(\Omega,\Gamma)$ is rank one if and only if $\Omega$ has a
Morse geodesic that projects to a closed geodesic in $\Omega/\Gamma$;
it follows from rank-rigidity (see below) that this occurs if and only
if $\Omega$ contains any Morse geodesic ray. A main result of
\cite{Islam2019} is that, when $(\Omega,\Gamma)$ is rank one and
$\Gamma$ divides $\Omega$, then $\Gamma$ has many rank one
automorphisms and in fact $\Gamma$ is an acylindrically hyperbolic
group. That is to say that $\Gamma$ has a `nice' action on some
(possibly non-proper) Gromov hyperbolic metric space, although
$(\Omega,\hil)$ itself may not be Gromov hyperbolic.

On the other hand, in \cite{Z2020}, Zimmer proved a rank rigidity
theorem for properly convex domains. This result implies that if
$\Gamma$ does not contain any rank one automorphisms, then either
$\Omega$ is reducible (meaning a cone over $\Omega$ splits as a
product of cones) or else $\Omega$ and $\Gamma$ are very restricted:
$\Omega$ must be a projective model for the Riemannian symmetric space
$G/K$ for a simple Lie group $G$ with rank at least two, and $\Gamma$
is isomorphic to a uniform lattice in $G$. Taken together, the results
in \cite{Islam2019} and \cite{Z2020} indicate that a ``generic''
divisible domain contains many projective geodesics that point in
``hyperbolic" directions.

\subsection{Comments on the proofs}

The proof of \cref{thm:morse_implies_bdry_reg_and_unif_reg} (our first
main theorem) relies on two key ingredients. The first is a
``straightness'' lemma (\Cref{lem:geodesic_additivity}) that does not
rely on Morseness at all---it holds for any sequence $\{\gamma_n\}$
tracking a projective geodesic. The lemma says that for any three
elements $\gamma_i, \gamma_j, \gamma_k$ in the tracking sequence, with
$i < j < k$, the gap $\mu_{1,2}(\gamma_i^{-1}\gamma_k)$ is coarsely
bounded below by the sum
$\mu_{1,2}(\gamma_i\gamma_j^{-1}) +
\mu_{1,2}(\gamma_j^{-1}\gamma_k)$. In particular, this implies that
$\mu_{1,2}(\gamma_n)$ is coarsely nondecreasing as a function of $n$,
which is \emph{not} a property satisfied by arbitrary quasi-geodesic
sequences in $\PGLdR$. We also remark that this ``straightness''
property does \emph{not} require any assumption on the regularity of
the sequence $\{\gamma_n\}$, which is critical for a later application
in the proof of \Cref{thm:not_strongly_uniform}.

The second ingredient in the proof of
\Cref{thm:morse_implies_bdry_reg_and_unif_reg} relies crucially on
$M$-Morseness, see \cref{lem:morse_geodesic_uniform_gap}. This lemma
shows that Morseness forces growth in $\mu_{1,2}$ as one shadows a
$M$-Morse geodesic for a sufficiently long time.  The proof of
\cref{thm:morse_implies_bdry_reg_and_unif_reg} then follows by a
telescoping argument splitting up the Morse geodesic into pieces with
sufficiently large $\mu_{1,2}$ growth; see
\cref{prop:morse_geodesic_gap_growth}.

\subsection{Outline of the paper}

The first part of this paper focuses mainly on the relationship
between the coarse metric geometry and projective geometry of a convex
projective domain. We provide some background about convex projective
geometry in \Cref{sec:preliminaries}. In \Cref{sec:morse_contracting},
we give several projective geometric characterizations of Morse
geodesics in Hilbert geometry, prove
\Cref{thm:hilbert_morse_equals_contracting}, and sketch the proof of
\Cref{thm:morse_local_to_global}. This section also introduces the
notion of \emph{conically related} pairs of points in the boundary of
a pair of convex projective domains, which is an important ingredient
in the proof of \Cref{thm:1_morse_ggt_morse}.

The next part of the paper focuses more on the linear algebraic
viewpoint. In \Cref{sec:sv_estimates}, we prove singular value
estimates along sequences $\{\gamma_n\}$ in $\PGL(d, \Rb)$ that tracks
a projective geodesic; in particular, we prove the ``straightness''
Lemma \ref{lem:geodesic_additivity} alluded to previously. Then in
\Cref{sec:sv_morse}, we use results from Section
\ref{sec:sv_estimates} (and also \cref{sec:morse_contracting}) to
prove the relationship between Morse geodesics and strongly uniformly
regular sequences, as described by
\Cref{thm:morse_implies_bdry_reg_and_unif_reg} and
\Cref{thm:1_morse_ggt_morse}.

In \Cref{sec:boundary_regularity_sv} we consider $C^\alpha$ regularity
and $\beta$-convexity of the boundary of a convex projective domain,
and prove \Cref{thm:regularity_equiv_regularity}. Finally, we
construct the counterexample described by
\Cref{thm:not_strongly_uniform} in \Cref{sec:morse_counterexample}.

\subsection{Acknowledgments}

The first author was partially supported by DFG Emmy Noether project 427903332 and DFG  project 338644254 (SPP 2026). 
The second author was partially supported by NSF grant DMS-2202770. The second author thanks Heidelberg University for hospitality where a part of the work was completed. We thank the anonymous referees for their comments and suggestions.
\section{Preliminaries}
\label{sec:preliminaries}

\subsection{Notation}

We standardize some notation for the entire paper. If $(X,d)$ is a
metric space, $A \subseteq X$, and $r>0$, then we denote the (open)
metric $r$-neighborhood of $A$ by
\[
  N^{X}_r(A):=\{ x \in X : d(x,A)<r \}.
\]
If $X$ is clear from context, we will simply write
$N^X_r(A) = N_r(A)$. If $A=\{x\}$, then we will use the notation
$B_r(x)$ to denote the metric $r$-ball $N_r(\{x\})$.

\subsubsection{Projective space}

When $V$ is a real vector space, we let $\Pb(V)$ denote the
projectivization of $V$, i.e. the space of 1-dimensional vector
subspaces of $V$. If $v$ is a nonzero vector in $V$ then $[v]$ denotes
the point in $\Pb(V)$ given by the span of $v$.

If $U \subseteq V$ is a subset, then $\Pb(U)$ denotes the image of
$U \minus \{0\}$ under the projectivization map $v \mapsto [v]$. If
$U$ is a vector subspace of $V$, this notation identifies the
projective space $\Pb(U)$ as a subset of $\Pb(V)$ (a \emph{projective
  subspace}). We will \emph{never} implicitly identify a vector
subspace $W \subseteq V$ with the corresponding projective subspace
$\Pb(W)$. If $P$ is a projective subspace, equal to $\Pb(W)$ for
$W \subseteq V$, we write $W = \lift{P}$.

If $U \subseteq V$ then the \emph{projective span} of $U$ is the
projective subspace $\Pspan{U} := \Pb(\Span(U))$. Similarly if
$P \subset \Pb(V)$, the projective span of $P$ is
$\Pspan{P} := \Pb(\Span(\lift{P}))$, where $\lift{P}$ is a lift of $P$
in $V$.

We let $\Pd(V)$ denote the space of codimension-1 subspaces of $V$. If
$W \in \Pd(V)$, then the projective subspace $\Pb(W) \subset \Pb(V)$
is a \emph{projective hyperplane}.

When $V = \Rb^d$, we have \emph{projective coordinates} on projective
space $\Pb(\Rb^d)$ defined in terms of the standard basis: the
notation $[x_1 : \ldots : x_d]$ denotes the projectivization of the
vector $(x_1, \ldots x_d)$.

\subsection{Properly convex domains}

In this section, we give some reminders about convex projective geometry. 
For a set $X \subset \Pb(\Rb^d)$, we denote by $\overline{X}$  the closure of $X$ in the
subspace topology induced from $\Pb(\Rb^d)$.

\begin{definition}
  A subset $\widetilde{\Omega} \subset \Rb^d$ is a \emph{convex cone} if
  it is convex, nonempty, and closed under multiplication by positive
  scalars. If $\widetilde{\Omega} \subset \Rb^d$ is a convex cone, we say
  that its projectivization $\Omega \subset \Pb(\Rb^d)$ is a
  \emph{properly convex domain} if $\Omega$ is open and
  $\overline{\Omega}$ does not contain any projective line in
  $\Pb(\Rb^d)$ (equivalently, if $\overline{\Omega}$ is a bounded
  convex subset of some affine chart in $\Pb(\Rb^d)$).
\end{definition}

The \emph{boundary} of a properly convex domain 
$\Omega$ is its topological boundary
$\dee \Omega:=\overline{\Omega}-\Omega$. Note that $\overline{\Omega}$
is topologically a closed ball and $\dee \Omega$ is homeomorphic to
the boundary of this ball. A \emph{supporting hyperplane} of a convex projective domain $\Omega$
is a projective hyperplane in $\Pb(\Rb^d)$ which intersects
$\overline{\Omega}$, but not $\Omega$.

If $x,y \in \Pb(\Rb^d)$ is a pair of distinct points, then
$\Pspan{x,y}$ is a projective line that contains both of
them. However, there does not exist a canonical notion of a projective
line segment joining $x$ and $y$ in general. But in the presence of a
properly convex domain $\Omega$ such that
$x, y \in \overline{\Omega}$, we can make a canonical choice.

For $x, y \in \Omega$, we say that the \emph{open projective line
  segment} joining $x$ and $y$ is the unique connected component of
$\Pspan{x,y}-\{x,y\}$ that is contained entirely in
$\overline{\Omega}$. We denote this by $(x,y)$. The \emph{projective
  line segment} joining $x$ and $y$, denoted by $[x,y]$, is the
closure of $(x,y)$ in $\Omega$. We will use the notation
$[x,y):=[x,y]-\{y\}$ and $(x,y]:=[x,y]-\{x\}$. Finally, if $x=y$, we
define $[x,y]=\{x\}$ while $(x,y)=\emptyset.$ Often, we will also
refer to projective line segments as projective geodesics, as we
explain below in \cref{sec:hil_metric}.

\begin{figure}[h]
  \centering
\begingroup%
  \makeatletter%
  \providecommand\color[2][]{%
    \errmessage{(Inkscape) Color is used for the text in Inkscape, but the package 'color.sty' is not loaded}%
    \renewcommand\color[2][]{}%
  }%
  \providecommand\transparent[1]{%
    \errmessage{(Inkscape) Transparency is used (non-zero) for the text in Inkscape, but the package 'transparent.sty' is not loaded}%
    \renewcommand\transparent[1]{}%
  }%
  \providecommand\rotatebox[2]{#2}%
  \newcommand*\fsize{\dimexpr\f@size pt\relax}%
  \newcommand*\lineheight[1]{\fontsize{\fsize}{#1\fsize}\selectfont}%
  \ifx\svgwidth\undefined%
    \setlength{\unitlength}{134.50390673bp}%
    \ifx\svgscale\undefined%
      \relax%
    \else%
      \setlength{\unitlength}{\unitlength * \real{\svgscale}}%
    \fi%
  \else%
    \setlength{\unitlength}{\svgwidth}%
  \fi%
  \global\let\svgwidth\undefined%
  \global\let\svgscale\undefined%
  \makeatother%
  \begin{picture}(1,0.97812751)%
    \lineheight{1}%
    \setlength\tabcolsep{0pt}%
    \put(0,0){\includegraphics[width=\unitlength,page=1]{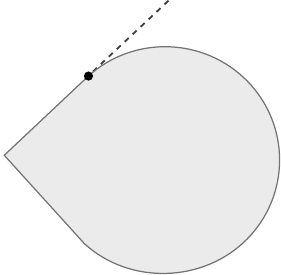}}%
    \put(0.29795715,0.62123421){\color[rgb]{0,0,0}\makebox(0,0)[lt]{\lineheight{1.25}\smash{\begin{tabular}[t]{l}$x$\end{tabular}}}}%
    \put(0.4111134,0.87508906){\color[rgb]{0,0,0}\makebox(0,0)[lt]{\lineheight{1.25}\smash{\begin{tabular}[t]{l}$H$\end{tabular}}}}%
    \put(0,0){\includegraphics[width=\unitlength,page=2]{exposed_face_correct.pdf}}%
  \end{picture}%
\endgroup%

  \caption{The point $x$ is a $C^1$ extreme point and $H$ is the
    unique supporting hyperplane at $x$. Here $F_{\Omega}(x)=\{x\}$
    is a face, but not an exposed face.}
  \label{fig:exposed_face}
\end{figure}

A \emph{face} of $\Omega$ is an equivalence class in $\dee \Omega$ of
the relation $\sim$, where $x \sim y$ if there is an \emph{open}
projective segment in $\dee \Omega$ containing $x$ and $y$.
\begin{definition}[Exposed boundary]
  \label{defn:exposed_boundary}
  We say that a face $F \subset \dee \Omega$ is \emph{exposed} if
  there is a supporting hyperplane $H$ of $\Omega$ whose intersection
  with $\dee \Omega$ is precisely $\overline{F}$; see \cref{fig:exposed_face}. A point $x \in \dee \Omega$ is \emph{exposed} if it lies in an
  exposed face. We say that $\Omega$ has \emph{exposed boundary} if
  every point in $\dee \Omega$ is exposed.
\end{definition}

Note that every known example of a convex divisible domain has exposed
boundary. However, it is still unknown whether this property holds for
every convex divisible domain.

\subsection{The Hilbert metric}
\label{sec:hil_metric}

If $\Omega$ is properly convex, the \emph{automorphism group}
$\Aut(\Omega) \subset \PGLdR$ is the group of projective transformations preserving
$\Omega$. One can always define an $\Aut(\Omega)$-invariant metric
$\hil$ on $\Omega$, called the \emph{Hilbert metric}, as follows. Fix a projective cross-ratio $[\cdot, \cdot ; \cdot, \cdot]$ on $\Rb\Pb^1$. We choose the cross-ratio  given by 
\begin{align*}
    [a, b ; x, y] = \frac{|a - y| \cdot |b - x|}{|a - x| \cdot
      |b - y|},
  \end{align*}
  where $|u - v|$ is the distance measured using any Euclidean metric on
  an affine chart of $\Rb\Pb^1$ containing $u,v$; the choice of chart and metric
  does not matter.

\begin{definition}
  Let $\Omega$ be a properly convex domain. For any $x,y \in \Omega$, let $a,b \in \partial \Omega$ be such that $x,y\in (a,b)$ and  the four points are arranged in the order $a,x,y,b$ along the projective line segment $[a,b]$. The distance between $x, y \in \Omega$ in the \emph{Hilbert metric} is defined as 
  \[
    \hil(x,y) = \frac{1}{2}\log[a, b; x, y].
  \]
\end{definition}

The pair $(\Omega, \hil)$ is always a proper geodesic metric space, on
which $\Aut(\Omega)$ acts by isometries. This ensures that the action
of $\Aut(\Omega)$ on $\Omega$ is always proper. When $\Omega$ is an
ellipsoid, then this metric space is isometric to $(d-1)$-dimensional
hyperbolic space; thus Hilbert geometry is a strict generalization of
hyperbolic geometry.

A projective line segment $[x,y]$ that lies in $\Omega$ (instead of lying entirely in $\dee \Omega$) is a geodesic for the metric $\hil$. Hence, we will call $[x,y]$ the  \emph{projective geodesic segment} joining $x$ and $y$. In the same vein, if $x,y \in \overline{\Omega}$, we call $(x,y)$ a \emph{projective geodesic}. Note that a projective geodesic may be infinite or bi-infinite and, wherever necessary, we will emphasize this by using specific terminology. If $x \in \Omega$ and $y \in \bdry$, then we will call $[x,y)$ (also $(x,y)$) a \emph{projective geodesic ray}. If $x,y \in \bdry$ but $(x,y) \subset \Omega$, we will call $(x,y)$ a \emph{bi-infinite projective geodesic} (or a \emph{projective geodesic line}).  

A projective geodesic segment, however, is often not the only geodesic
joining points $x, y \in \Omega$. One can easily verify the following:
\begin{fact}[Characterizing geodesics]
  \label{fact:nonunique_geodesics}
  For pairwise distinct points $w_1, w_2, w_3 \in \Omega$, we have
  \[
    \hil(w_1, w_2) = \hil(w_1, w_3) + \hil(w_3, w_2)
  \]
  if and only if there are segments $[y, y']$ and $[z, z']$ in
  $\dee \Omega$ such that $y, w_1, w_3, z$ and $y', w_3, w_2, z'$ are
  aligned in that order.
\end{fact}
\Cref{fact:nonunique_geodesics} implies that if $\Omega$ is a \emph{strictly} convex domain (i.e. if
there are no nontrivial projective segments in $\dee \Omega$), then
every geodesic in $\Omega$ is projective.

\subsection{Finer metric properties of $\hil$}

As mentioned in the introduction to this paper, the metric space
$(\Omega, \hil)$ is typically not a CAT(0) space, and in fact the
Hilbert metric often fails to satisfy some of the strong convexity
properties enjoyed by general CAT(0) metrics. However, the Hilbert
metric does satisfy a weak convexity property called the \emph{maximum
  principle.}
\begin{lemma}[{Maximum principle; see \cite[Corollary 1.9]{CLT2015}}]
  \label{lem:maximum_principle} If $C$ is a closed convex set in a
  properly convex domain $\Omega$, then for every compact subset
  $K \subset \Omega$, the function $K \to \Rb_{\ge 0}$ given by
  \[
    x \mapsto \hil(x, C)
  \]
  attains its maximum at an extreme point of $K$.
\end{lemma}

It is also true that when $C$ is a convex subset of a convex
projective domain $\Omega$, the nearest-point projection map
$\Omega \to C$ is not always well-defined. One can still define a
\emph{set-valued} nearest-point projection map $\pi_C:\Omega \to 2^C$,
but this map is not necessarily continuous with respect to Hausdorff
distance on $2^C$. However, using the maximum principle, one can see
that the nearest-point projection map onto a projective geodesic in
$\Omega$ always maps convex sets to connected sets:
\begin{lemma}
  \label{lem:projection_connected}
  Let $\ell$ be a projective geodesic in a properly convex domain
  $\Omega$, and let $\pi_\ell:\Omega \to 2^\ell$ denote the set-valued
  nearest-point projection map, i.e. the map
  \[
    \pi_\ell(x) = \{y \in \ell : \hil(x, y) = \hil(x, \ell)\}.
  \]
  If $A \subset \Omega$ is convex, then $\pi_\ell(A)$ is connected.
\end{lemma}
\begin{proof}
  Fix points $x', y' \in \ell$, so that
  $x' \in \pi_\ell(x), y' \in \pi_\ell(y)$ for $x, y \in A$, and let
  $z'$ be a point on the open segment $(x', y') \subset \ell$. We wish
  to show that $z' \in \pi_\ell(A)$. The point $z'$ separates $\ell$
  into two components, so let $\ell_-$ be the closure of the component
  containing $x'$, and let $\ell_+$ be the closure of the component
  containing $y'$.

  Since $A$ is convex, it contains the projective geodesic
  $[x,y]$. Consider the continuous function $f:[x, y] \to \Rb$ given
  by
  \[
    f(u) = \hil(u, \ell_-) - \hil(u, \ell_+).
  \]
  For any $u \in [x, y]$, we know that
  $\hil(u, \ell) = \min\{\hil(u, \ell_-), \hil(u, \ell_+)\}$. So,
  since $\hil(x, \ell) = \hil(x, x') \ge \hil(x, \ell_-)$, it
  follows that $f(x)$ is nonpositive. Similarly, $f(y)$ is
  nonnegative, so there is some $z \in [x, y]$ with $f(z) = 0$,
  i.e. $z$ satisfying $\hil(z, \ell_-) = \hil(z, \ell_+)$. Thus
  there are points $z_\pm \in \ell_\pm$ which satisfy
  \[
    \hil(z, z_+) = \hil(z, z_-) = \hil(z, \ell_\pm) = \hil(z, \ell).
  \]
  Then, applying \Cref{lem:maximum_principle} (taking the compact set $K$ to be $[z_-, z_+]$ and the closed convex set $C$ to be the singleton $\{z\}$) implies that every point $w$ in $[z_-, z_+]$ satisfies $\hil(z, w) \le \hil(z, \ell)$. Therefore, each such $w$ satisfies $\hil(z, w) = \hil(z, \ell)$, i.e.
  $w \in \pi_\ell(z)$. As $z \in A$ and the geodesic $[z_-, z_+]$
  contains the previously chosen point $z'$, this proves the claim.
\end{proof}

\subsection{Space of properly convex domains}

 Suppose $(X,d)$ is a
metric space. This induces a notion of Hausdorff distance between closed subsets $A,B \subset X$ defined by:
\begin{align*}
    d^{\Haus}(A,B)=\max \left\{ \sup_{a\in A}d(a,B), \sup_{b \in B} d(A,b)\right\}.
\end{align*}

Fixing a metric on the projective space $\Pb(V)$, compatible with the standard topology on $\Pb(V)$, defines a notion of Hausdorff distance between subsets of $\Pb(V)$ (or more precisely, their closures). 
\begin{definition}
  \label{defn:domain_space}
  Let $V$ be a real vector space. We denote by $\domains(V)$ the space
  of properly convex projective domains in $\Pb(V)$. The topology on $\domains(V)$
  is the topology induced by Hausdorff distance, with respect to any metrization of $\Pb(V)$.
\end{definition}
Note that the topology on $\domains(V)$ is independent of the metrization on $\Pb(V)$.

\subsection{The Benz\'ecri cocompactness theorem}

\begin{definition}
  \label{defn:domains_and_pdomains}
  Let $\pdomains(V)$ denote the space of \emph{pointed} domains in
  $\Pb(V)$, i.e. the space
  \[
    \pdomains(V) := \set{(\Omega, x) \in \domains(V) \times \Pb(V) : x
      \in \Omega}.
  \]
\end{definition}
The topology on $\pdomains(V)$ is the product topology that it
inherits from $\domains(V) \times \Pb(V)$.  The group $\PGL(V)$ acts
pointwise on $\domains(V)$, and diagonally on $\pdomains(V)$. We have
the following important result:

\begin{theorem}[Benz\'ecri cocompactness \cite{benzecri}]
  \label{thm:benzecri}
  The action of $\PGL(V)$ on $\pdomains(V)$ is both proper and
  cocompact.
\end{theorem}

\Cref{thm:benzecri} turns out to be very useful when we consider the
case of a \emph{non-cococompact} group action on a properly convex
domain. Although divisible domains $\Omega$ are often technically
easier to work with than general domains, they require the
automorphism group $\Aut(\Omega)$ to be `large'. In this paper, we
will often be interested in studying general properly convex domains,
not necessarily divisible. In such cases, the Benz\'ecri cocompactness
theorem becomes a powerful tool that we can use to import techniques
for divisible domains to the non-divisible case.

Typically, we apply the theorem to a sequence of points $x_n$
in some domain $\Omega$ which leaves every compact subset of $\Omega$,
to find a sequence of ``approximate automorphisms'' taking $x_n$ back
to some fixed basepoint. The properness part of the theorem ensures
that any choice of ``approximate automorphisms'' differ by elements in
a compact set, which we can often use to obtain information about a
given sequence of divergent elements in $\Aut(\Omega)$.

For divisible domains, \cref{thm:benzecri} admits an easy corollary that we record for future use.
\begin{corollary}
\label{cor:benzecri_for_divisible}
Suppose  $\Omega\subset \Pb(V)$ is a properly convex divisible domain. Then the $\PGL(V)$-orbit of $\Omega$ in $\domains(V)$ is closed.
\end{corollary}
\begin{proof}
    Suppose $\seq{g_n}$ is a sequence in $\PGL(V)$ such that $g_n \Omega \to \Omega_{\infty}$ in $\domains(V)$. Pick $x_{\infty} \in \Omega_{\infty}$
    and $\seq{x_n}$ in $\Omega$ such that $g_n(\Omega,x_n) \to (\Omega_{\infty},x_\infty)$ in $\pdomains(V)$. Since $\Omega$ is divisible, we can find a
    sequence $\seq{\gamma_n}$ in $\Aut(\Omega)$ such that, after passing to a subsequence, we have $\gamma_n^{-1}x_n \to x_0 \in \Omega.$
    
    Fix neighborhoods $K_1$ and $K_2$ of $(\Omega,x_0)$ and $(\Omega_{\infty},x_{\infty})$ in $\pdomains(V)$ such that $\overline{K_i}$ is compact for $i=1,2$. Then observe that for $n$ sufficiently large,
    $(\gamma_n^{-1}g_n^{-1} K_2) \cap K_1 \neq \emptyset$. By \cref{thm:benzecri}, $\gamma_n^{-1}g_n^{-1}$ lies in a compact subset of $\PGL(V)$. So, by passing to a subsequence, we can assume that
    $\gamma_n^{-1}g_n^{-1} \to q^{-1} \in \PGL(V)$. Hence $g_n \Omega \to q  \Omega$, which implies $\Omega_\infty=q  \Omega$, i.e. $\Omega_{\infty}$ lies in
    the $\PGL(V)$-orbit of $\Omega$.
\end{proof}

\subsection{Properties of faces in properly convex domains}

Every face $F$ of a properly convex domain is itself a properly convex
domain in its own projective span. Consequently, $F$ can be endowed
with its own Hilbert metric $\mathrm{d}_F$. This Hilbert metric is
related to the Hilbert metric on the larger domain $\Omega$, and gives
a way to characterize faces in terms of metric (rather than
projective) geometry. This is expressed via
\Cref{lem:hilbert_face_metric} below.

We state this lemma in a fairly general form. In particular, we allow
the domain $\Omega$ to vary continuously in the space of all properly
convex domains $\domains(\Rb^d)$ (see \cref{defn:domain_space}).

\begin{lemma}
  \label{lem:hilbert_face_metric}
  Let $\{\Omega_n\}$ be a sequence of properly convex domains in
  $\Pb(\Rb^d)$, converging in $\domains(\Rb^d)$ to a properly convex
  domain $\Omega_\infty$. Suppose that points $x_n, y_n \in \Omega_n$
  converge to $x, y \in \overline{\Omega_\infty}$. If
  \[
    \liminf_{n \to \infty}d_{\Omega_n}(x_n, y_n) < \infty,
  \]
  then $x$ and $y$ lie in the same face $F$ of $\Omega_\infty$, and
  \begin{align*}
    \mathrm{d}_F(x, y) \le \liminf_{n \to \infty} d_{\Omega_n}(x_n, y_n).
  \end{align*}
\end{lemma}
Since this version of the lemma is slightly more general than versions
that typically appear in the literature, we provide a proof.
\begin{proof}[Proof of \cref{lem:hilbert_face_metric}]
  Note that there is nothing to prove if $x=y$, so assume that
  $x \neq y$.  Let
  $[a_n,b_n]:=\overline{\Omega_n} \cap \Pspan{x_n,y_n}$ where the
  labels $a_n, b_n$ are assigned in such a way that the four points
  $a_n,x_n,y_n,b_n$ appear in this order along $\Pspan{x_n,y_n}$. Up
  to passing to a subsequence, we can assume that $a_n \to a$ and
  $b_n \to b$ in $\Pb(\Rb^d)$.  Since $\Omega_n \to \Omega_{\infty}$,
  $a,b \in \overline{\Omega_\infty}$ and $[a_n,b_n] \to [a,b]$. Thus
  $x,y \in [a,b]$ which implies that $a \neq b$, since otherwise $x$
  will be equal to $y$. By the ordering of the labels $a_n,b_n$, we
  know that the points $a,x,y,b$ appear in this order along
  $\Pspan{a,b}$. If either $a=x$ or $b=y$, then the cross-ratio
  $[a_n,b_n;x_n,y_n] \to \infty$ and contradicts
  $\liminf_{n\to \infty}(x_n,y_n)<\infty$. Thus, the four points
  $a,x,y,b \in \overline{\Omega_\infty}$ are pairwise distinct. Hence
  $x,y \in (a,b)$.

  Now observe that $(a,b)$, which is an open projective line segment
  in $\overline{\Omega_{\infty}}$, is either disjoint from
  $\partial \Omega_\infty$ or is entirely contained in it. Since
  $x,y \in (a,b) \cap \partial \Omega_\infty$,
  $(a,b) \subset \partial \Omega_{\infty}$. This implies that $x,y$
  belong to a face $F$ in $\Omega_\infty$. Moreover, by continuity of
  cross-ratios,
\begin{equation*}
d_F(x,y) \leq d_{(a,b)}(x,y) \leq \liminf_{n \to \infty}d_{(a_n,b_n)}(x_n,y_n)=\liminf_{n \to \infty}d_{\Omega_n}(x_n,y_n). \qedhere 
\end{equation*}
\end{proof}
When the domain $\Omega$ is fixed, we can use
\Cref{lem:hilbert_face_metric} together with the maximum principle
(\Cref{lem:maximum_principle}) to obtain a related estimate for the
Hausdorff distance between a pair of projective geodesics. For this
lemma, we follow the convention that $F_{\Omega}(x)=\Omega$ if
$x \in \Omega$, while $F_{\Omega}(x)$ is the face containing $x$ if
$x \in \dee \Omega$.
\begin{lemma}
\label{lem:hil_haus_dist_between_geods}
Suppose $\Omega$ is a properly convex domain, $x_{\pm} \in \overline{\Omega}$,
and $y_\pm \in F_{\Omega}(x_{\pm}).$ If $(x_+,x_-)\subset \Omega$,
then $(y_-,y_+)\subset \Omega$ and
\[
  \hil^{\Haus}((y_+,y_-),(x_+,x_-)) \leq \max \left\{
    d_{F_{\Omega}(x_{\pm})}(x_{\pm},y_{\pm})\right\}. 
\]
\end{lemma}

\subsection{Properly embedded simplices}
\label{sec:properly_embedded_simplices}

For any $k\geq 0$, a \emph{standard projective $k$-simplex} in $\Pb(\Rb^d)$ is $$S_k:=\{ [x_1:x_2:\dots:x_{k+1}:0:\dots:0] ~|~ x_1,\dots,x_{k+1}>0 \}.$$ We say that $S_k$ is the simplex spanned by $[e_1], \dots, [e_{k+1}]$. A \emph{projective $k$-simplex} is any set in $\Pb(\Rb^d)$ that is projectively equivalent to a standard projective $k$-simplex.

\begin{definition}
\label{defn:prop_embedded}
Suppose $\Omega$ is a properly convex domain and $A \subset \Omega$ is a convex subset. Then we say that:
\begin{enumerate}
\item $A$ is a \emph{properly embedded} subset if $A \hookrightarrow
  \Omega$ is a proper map, or equivalently if $\partial A \subset \partial \Omega$.  
\item $A$ is a \emph{properly embedded $k$-simplex} if $A$ is properly
  embedded in $\Omega$ and a projective $k$-simplex.
\end{enumerate}
\end{definition}

Properly embedded simplices are projective analogs of totally geodesic
flats in CAT(0) spaces. Consider, for example, a properly embedded
triangle, or 2-simplex.  Suppose the vertices of such a triangle
$\Delta$ are represented by the standard basis vectors in
$\Rb^3$. Then the group
\[
  \left\{
    \begin{pmatrix}
      2^a\\ & 2^b\\ & & 2^c
    \end{pmatrix} : a, b, c \in \Zb, a + b + c = 0 \right\}
\]
acts properly discontinuously and cocompactly on $\Delta$. So $(\Delta,d_{\Delta})$ equipped with its Hilbert metric is
quasi-isometric to a 2-flat. Hence properly embedded simplices serve
as analogs of isometrically embedded flats in CAT(0) spaces.

\subsection{Singular values and the Cartan projection}
\label{sec:sv_background}

In this section we briefly recall the definitions and basic properties
of the Cartan projection $\GLdR \to \Rb^d$. We will always equip
$\Rb^d$ with its standard Euclidean inner product.

\begin{definition}
  For any $g \in \GLdR$, we let
  $\sigma_1(g) \ge \sigma_2(g) \ge \ldots \ge \sigma_d(g) > 0$ denote
  the \emph{singular values} of $g$, counted with multiplicity. We let
  $\mu:\GLdR \to \Rb^d$ denote the \emph{Cartan projection},
  given by $\mu_i(g) = \log \sigma_i(g)$. The Cartan projection can be
  also be defined via the \emph{Cartan decomposition} of a group
  element $g \in \GLdR$: $\mu(g)$ is the unique vector in
  $\Rb^d$ with nonincreasing entries such that
  \[
    g = k \cdot \exp(\op{diag}(\mu_1(g), \ldots \mu_d(g))) \cdot \ell,
  \]
  for some $k, \ell \in \mathrm{O}(d)$. For $1 \le i \le j \le d$, we let $\mu_{i,j}(g)$ denote the
  nonnegative quantity $\mu_i(g) - \mu_{j}(g)$.
\end{definition}

\begin{remark}
  Although the Cartan projection $\mu$ is only defined on
  $\GLdR$, the values of $\mu_{i,j}$ are well-defined on the
  quotient $\PGL(d, \Rb)$.
\end{remark}

The singular values of any $g \in \GL(V)$ have an interpretation in
terms of the norm and the conorm of $g$. Recall that if
$g \in \GL (V)$, the operator norm is
\[
  ||g|| = \sup_{v \in \Rb^d \minus \set{0}} \frac{||gv||}{||v||},
\]
while the conorm is 
\[\conorm{g} = ||g^{-1}||^{-1}.\]
The largest singular value is given by $\sigma_1(g)=\norm{g}$ while
the smallest singular value is given by $\sigma_d(g)=\conorm{g}$. More
generally, for any $1 \le k \le d$, we let $\mathrm{Gr}(k, d)$ denote
the Grassmannian of $k$-dimensional subspaces of $\Rb^d$. Then one has
the ``minimax'' formula:
\begin{equation}
  \label{eq:sv_minimax}
  \sigma_k(g) = \max_{W \in \mathrm{Gr}(i, d)} \conorm{g|_W}.
\end{equation}

Note that if $g \in \SLdR$, we
have $\prod \sigma_i(g) = 1$ and thus $\sum \mu_i(g) = 0$. Using this,
we see that for any $g \in \SLdR$, we have
\[
  \mu_{1,d}(g) = \log(||g||) + \log(||g^{-1}||).
\]

\begin{lemma}[Additivity of Cartan projection, see {\cite[Fact
    2.18]{ggkw2017anosov}}]
  \label{lem:additive_root_bound}
  There is a constant $K_0 > 0$ so that for any
  $g, h_1, h_2 \in \GLdR$, we have
  \begin{equation}
    \label{eq:cartan_inequality}
    ||\mu(h_1gh_2) - \mu(g)|| \le K_0(||\mu(h_1)|| + ||\mu(h_2)||).
  \end{equation}
  In particular, for any $1 \le i < j \le d$, there is a constant
  $K > 0$ such that
  \begin{equation}
    \label{eq:cartan_inequality_ii}
    |\mu_{i,j}(h_1gh_2) - \mu_{i,j}(g)| \le K(||\mu(h_1)|| +
    ||\mu(h_3)||).
  \end{equation}
\end{lemma}
\begin{remark}
  For an appropriate choice of norm on $\Rb^d$ (which is typically not
  the standard norm), the inequality \eqref{eq:cartan_inequality}
  can be strengthened to
  \[
    ||\mu(h_1gh_2) - \mu(g)|| \le ||\mu(h_1)|| + ||\mu(h_2)||.
  \]
  This immediately implies the version of the inequality we have
  stated above.
\end{remark}

\begin{lemma}
\label{lem:singular_value_ratios}
Suppose $g \in \GLdR$ and there exist $C>0$ and
$1 \leq i \leq j \leq d$ such that
\begin{equation*}
  \abs{\mu_{i,j}(g) - \mu_{1,d}(g)} \leq C.
\end{equation*}
Then:
\begin{enumerate}
\item $\mu_{1,k}(g) \leq C$ for $k \in \set{1, \ldots, i}$,
\item $\mu_{k,d}(g) \leq C$ for $k \in \set{j, \ldots, d}$, and
\item $\mu_{k, k+1}(g) \leq C$ for
  $k \in \set{1, \ldots, i-1} \cup \set{j, \ldots, d-1}$.
\end{enumerate}
\end{lemma}
\begin{proof}
Since the values of $\mu_k(g)$ are non-increasing, $\mu_{i',j'}(g)\geq 0$ for any $1 \leq i'\leq j' \leq d$. But 
  $\mu_{1,d}(g)$ is equal to the sum $\mu_{1,i}(g) + \mu_{i,j}(g) + \mu_{j,d}(g)$. Thus $\mu_{1,i}(g)\leq C$ and $\mu_{j,d}(g) \leq C$. The first two inequalities are then immediate as 
  $\mu_{1,k}(g) \le \mu_{1,i}(g)$ whenever
  $k \in \set{1, \ldots, i - 1}$, and $\mu_{k,d}(g) \le \mu_{j,d}(g)$
  for any $k \in \set{j, \ldots, d}$. The third inequality follows from the first two and
  the fact that $\mu_{k,k+1}(g)$ is bounded by either $\mu_{1,i}(g)$
  or $\mu_{j,d}(g)$ whenever
  $k \in \set{1, \ldots, i-1} \cup \set{j, \ldots, d-1}$.
\end{proof}

Let $\angle$ be the standard angle in $\Rb^d$ induced by the standard Euclidean inner product. Note that $\angle$ also defines a Riemannian metric $\dproj$ on
 $\Pb(\Rb^d)$, by setting $\dproj(u,v) = \angle(u,v)$ for any $u,v \in \Pb(\Rb^d)$. There is an analogous notion of angles between subspaces.

\begin{definition}
  If $U, W$ are two transverse subspaces of $\Rb^d$, we define the
  angle $\angle(U, W)$ by
  \[
    \angle(U, W) = \inf_{\substack{u \in U \minus \set{0}\\ w \in W
        \minus \set{0}}} \angle(u, w).
  \]
\end{definition}

\begin{lemma}
  \label{lem:transverse_decomp_compact}
  For any $\eps > 0$, there exists $C \equiv C(\eps)$ satisfying the
  following. Suppose we have two decompositions
  \begin{align*}
    \Rb^d &= U_1 \oplus \ldots \oplus U_k,\\
    \Rb^d &= W_1 \oplus \ldots \oplus W_k,
  \end{align*}
  such that $\dim(U_i) = \dim(W_i)$ for all $i$, and
  $\angle(U_i, U_j) \ge \eps$ and $\angle(W_i, W_j) \ge \eps$ for all
  $i \ne j$. Then there is some $q \in \GLdR$ such that
  $q(U_i) = W_i$ for all $i$ and $\mu_{1,d}(q) \le C$.
\end{lemma}
\begin{proof}
  By choosing orthogonal bases for each $U_i$ and each $W_i$, we can
  reduce to the case where the subspaces $U_i$ and $W_i$ are all
  one-dimensional. Then, using \Cref{lem:additive_root_bound}, we can
  further reduce to the case where the subspaces $U_i$ give the
  decomposition of $\Rb^d$ into the lines spanned by the standard
  basis vectors $e_1, \ldots, e_d$.

  We can pick unit vectors $w_1, \ldots, w_d$ spanning each $W_i$, and
  consider the matrix $q$ whose columns are $w_1, \ldots, w_d$. Then
  $q$ takes $U_i$ to $W_i$, and lies in the compact subset $K(\eps)$
  of $\GLdR$ consisting of matrices whose columns are unit
  vectors having pairwise angles at least $\eps$. By compactness there
  is a uniform upper bound $C$ on $\mu_{1,d}(K(\eps))$, and the result
  follows. 
\end{proof}

\section{Morse geodesics are contracting}
\label{sec:morse_contracting}

Our main goal in this section is to prove
\cref{thm:hilbert_morse_equals_contracting}, which says that Morse
geodesics (\cref{defn:morse_gauge}) in a convex projective domain
$\Omega$ are equivalent to contracting geodesics
(\cref{defn:D_contracting}). As part of the proof, we also introduce
the framework of \emph{conically related} pairs of points in
boundaries of convex projective domains, and use this to provide a
number of other characterizations of Morse geodesics in
$\Omega$. These ideas will reappear later in \Cref{sec:sv_morse}, when
we use them to study the linear algebraic behavior of Morse geodesics.

Our proof of the equivalence between Morse and contracting geodesics
goes through a \emph{$\delta$-slimness} property for geodesic
triangles, which is reminiscent of a similar property that also
characterizes Morseness in CAT(0) spaces. We define this property
below. Note that our definition only refers to projective geodesic triangles in $(\Omega,\hil)$ (and not \emph{arbitrary} geodesic triangles). 
\begin{definition}
  Let $\ell$ be a projective geodesic in a properly convex domain
  $\Omega$ and $\delta \geq 0$. We say that $\ell$ is
  \emph{projectively $\delta$-slim} if, any projective geodesic
  triangle $[x,y]\cup[y,z]\cup[z,x]$ with $x,y, z \in \Omega$ and
  $[x,y]\subset \ell$ is $\delta$-slim, i.e., for
  $\{a, b, c\} = \{x, y, z\}$, we have
  \[
    [a,c] \subset N_\delta([a,b]) \cup N_\delta([b,c]).
  \] 
\end{definition}
\begin{remark}
\label{rem:proj_delta_slim_defn}
    \cite[Lemma 13.8]{IZ2019} implies that for a projective geodesic
    triangle to be $\delta$-slim, it suffices that one of its edges is
    contained in the $\frac{\delta}{2}$ neighborhood of its other two
    edges. More precisely, $[x,y]\cup[y,z]\cup[z,x] \subset \Omega$ is
    $\delta$-slim if $[x,y] \subset N_{r}([x,z]) \cup N_r([z,y])$ with
    $r:=\frac{\delta}{2}$.
\end{remark}

Our main result in this section is the following:
\begin{proposition}
  \label{prop:morse_contracting_slim}
  Let $\Omega$ be a properly convex domain and let $\ell$ be a
  projective geodesic in $\Omega$. The following are equivalent:
  \begin{enumerate}
  \item $\ell$ is Morse.
    \item $\ell$ is projectively $\delta$-slim.
  \item $\ell$ is contracting.
  \end{enumerate}
\end{proposition}

In \Cref{prop:morse_contracting_slim}, the implication (3) $\implies$
(1) follows from a well-known general result, stated below. The proof
is standard; see e.g. \cite[Lemma 3.3]{Sultan2014}.
\begin{proposition}
  \label{prop:contracting_implies_morse}
  Let $X$ be a proper geodesic metric space and let $D > 0$. There exists a Morse gauge $M$, depending only on $D$ and $X$, so that
  any $D$-contracting geodesic in $X$ is $M$-Morse.
\end{proposition}

The implication (2) $\implies$ (3) in
\Cref{prop:morse_contracting_slim} is also straightforward, and we
provide a quick proof below. Most of the rest of this section is then
devoted to the proof of the implication (1) $\implies$ (2).

\subsection{Projectively $\delta$-slim implies contracting}
This is the implication (2) $\implies$ (3) in
\Cref{prop:morse_contracting_slim}. For this result, we mostly imitate
the proof, due to Charney-Sultan, of an analogous statement for CAT(0)
spaces (see Theorem 2.14 in \cite{CS2015}). It turns out that in most
situations, the Charney-Sultan proof does not use the full strength of
the CAT(0) condition, but only the weaker \emph{maximum principle}
(see \Cref{lem:maximum_principle}).

The Charney-Sultan proof does also appeal to the CAT(0) condition in
ways not covered by the maximum principle. However, it is not
difficult to modify the proof to avoid this, at the cost of increasing
some of the constants appearing in the proof. The first step is the
following lemma.

\begin{lemma}
  \label{lem:thin_triangle_i}
  Let $\Omega$ be a properly convex domain and let
  $\ell\subset \Omega$ be a projective geodesic. Suppose that $\ell$
  is projectively $\delta$-slim. Then, for any $x \in \Omega$,
  $y \in \ell$, and $z \in \pi_\ell(x)$, we have
  $\hil(z, [x,y]) < 4\delta$.
\end{lemma}
\begin{proof}
  If $\hil(y, z) \le 2\delta$ we are done, so assume that
  $\hil(y,z) > 2\delta$, and then choose a point $w \in [y,z]$ so that
  $2\delta < \hil(w, z) < 3\delta$. Then let $u$ be a point on
  $[x, z]$ so that
  \[
    \hil(w, u) = \hil(w, [x, z]).
  \]
  From the triangle inequality we have
  \[
    \hil(x, u) + \hil(u, w) \ge \hil(x, w),
  \]
  and since $z \in \pi_\ell(x)$ and $w \in \ell$ we know that
  \[
    \hil(x, w) \ge \hil(x, z) = \hil(x, u) + \hil(u, z).
  \]
  Putting the previous two lines together we see that
  $\hil(u, z) \le \hil(u, w)$. But then
  \[
    2\delta < \hil(w, z) \le \hil(w, u) + \hil(u, z) \le 2\hil(u, w),
  \]
  which implies that $\hil(w, [x,z]) = \hil(w, u) > \delta$.
  
  Now, as $\ell$ is projectively $\delta$-slim, the projective
  geodesic triangle $[x,y] \cup [y,z] \cup [z,x]$ is
  $\delta$-slim. Since $\hil(w, [x,z]) > \delta$, we have
  $\hil(w, [x,y]) < \delta$. Thus
  \begin{equation*}
    \hil(z, [x,y]) \leq  \hil(z, w) + \hil(w, [x,y]) < 4\delta.
  \end{equation*} \end{proof}

The following completes the proof that (2) $\implies$ (3) in
\Cref{prop:morse_contracting_slim}.
\begin{proposition}
  \label{prop:slim_contracting}
  Let $\Omega$ be a properly convex domain and let $\ell$ be a
  projective geodesic in $\Omega$. If $\ell$ is projectively
  $\delta$-slim, then $\ell$ is $24\delta$-contracting.
\end{proposition}
\begin{proof}
  Let $B = B(x,r)$ be a ball not intersecting $\ell$. Let $y \in B$
  and let $x' \in \pi_\ell(x), y' \in \pi_\ell(y)$.

  By \Cref{lem:thin_triangle_i}, there is a point $u \in [y,x']$ such
  that $d(y', u) < 4\delta$. The maximum principle
  (\Cref{lem:maximum_principle}) implies that
  \[
    \hil(x, u) \le \max\{\hil(x,y), \hil(x, x')\} = \hil(x, x'),
  \]
  so
  \[
    \hil(x, y') \le \hil(x, x') + 4\delta.
  \]
  Then we apply \Cref{lem:thin_triangle_i} again to see that there is
  a point $w \in [y',x]$ so that $\hil(x', w) < 4\delta$. Then
  \begin{align*}
    \hil(x, y') &= \hil(x, w) + \hil(w, y')\\
                &\ge \hil(x, x') - \hil(x', w) + \hil(y', x') -
                  \hil(x', w)\\ 
                &\ge \hil(x, x') + \hil(y', x') - 8\delta.
  \end{align*}
  That is, we have
  \[
    \hil(x, x') + \hil(y', x') - 8\delta \le \hil(x, y') \le \hil(x,
    x') + 4\delta,
  \]
  which implies $\hil(y', x') < 12\delta$. Thus the diameter of the
  nearest-point projection of $B$ onto $\ell$ is at most $24\delta$.
\end{proof}

Having proved that (2) $\implies$ (3) $\implies$ (1) in
\Cref{prop:morse_contracting_slim}, we now turn to the implication (1)
$\implies$ (2). Our proof of this implication relies much more heavily
on the convex projective geometry of the domain $\Omega$. In
particular, we develop a notion of \emph{conically related} pairs of
points in the boundary of certain pairs of properly convex domains,
and show that Morseness is preserved between conically related
points. This allows us to develop several other more technical
characterizations of Morse geodesics in a convex projective domain
$\Omega$, which we ultimately use to establish the desired implication
in \Cref{prop:morse_contracting_slim}.

We state all of our different equivalences below in
\Cref{prop:morse_hilbert_equivalences}. First, however, we need a few
more definitions.

\subsection{Half triangles}

Half-triangles in convex projective domains extend the analogy between
properly embedded triangles and flats in CAT(0) spaces (see
\Cref{sec:properly_embedded_simplices}) to isometrically embedded
\emph{half-flats} in CAT(0) spaces.
\begin{definition}
\label{defn:half_triangle}
  Let $\Omega$ be a properly convex domain. A \emph{half-triangle} in
  $\Omega$ is a convex subset given by ($\Pb(\Span\{x,y,z\}) \cap \Omega$) where $x,y,z \in \partial \Omega$ are such that $[x,y]\cup[y,z] \subset \partial \Omega$ and $(x,z) \subset \Omega$. 
  
  A half triangle is uniquely determined by such a triple $x,y,z \in \partial \Omega$ satisfying $[x,y]\cup[y,z] \subset \partial \Omega$ and $(x,z) \subset \Omega$. Moreover, given such a triple $x,y,z \in \partial \Omega$, we say that $(x,z)$ is \emph{contained in a half triangle in $\Omega$}.  
\end{definition}

Note that, as a subspace of $\Omega$ with its restricted Hilbert
metric, a half-triangle is \emph{not} necessarily isometric to a
half-space (i.e. a subspace bounded by a geodesic) in a properly
embedded triangle. Nevertheless half-triangles still bear some
resemblance to half-flats, since segments in the boundary of a
properly convex domain correspond roughly to ``flat directions'' (see
e.g. \Cref{lem:int_segments_triangles}).

\subsection{Conically related points}

The idea behind our next definition (that of \emph{conically related
  points}) is to consider what a properly convex domain $\Omega$
``looks like'' from the perspective of a sequence of points traveling
along a projective geodesic ray $c$ towards the ideal endpoint
$c(\infty)$ in $\partial \Omega$. If $\Omega$ has a cocompact action
by projective automorphisms, we can consider a sequence of points
$\{x_n\}$ limiting to an ideal endpoint $z$ of $c$, and a sequence of
group elements $\{\gamma_n\}$ in $\Aut(\Omega)$ taking $x_n$ back to
some fixed compact subset of $\Omega$. The projective geometry of the
accumulation points of the sequence $\{\gamma_n z\}$ in $\dee \Omega$
should inform the metric geometry of the geodesic $c$.

More generally, when $\Aut(\Omega$) does \emph{not} act cocompactly on
$\Omega$, we can use the Benz\'ecri cocompactness theorem
(\Cref{thm:benzecri}) to find elements $g_n$ in $\PGL(V)$ which
``translate'' points in the sequence $\{x_n\}$ into a fixed compact
subset of some limiting domain $\Omega_\infty$. Again, we can
understand the metric geometry of the geodesic $c$ by looking at
accumulation points of $\{g_nz\}$ in $\dee \Omega_\infty$.

In \cite{B2003}, Benoist used essentially this approach to investigate
the global hyperbolicity of arbitrary convex projective domains. The definition below gives one way to formalize this idea. (For another,
see e.g. \cite[Section 5]{W2020}).
\begin{definition}
\label{defn:conically_related_points}
  Let $\Omega_1, \Omega_2$ be properly convex domains, let
  $z_1 \in \dee \Omega_1$, and let $z_2^+, z_2^-$ be points in
  $\dee \Omega_2$ such that $(z_2^+, z_2^-) \subset \Omega_2$. Suppose that:
    \begin{enumerate}
    \item there is a sequence of points $\{x_n\}$ in the projective
      geodesic ray $[x, z_1) \subset \Omega_1$ such that $x_n$
      converges to $z_1$, and
    \item there is a divergent sequence of group elements $\{g_n\}$ in
      $\PGL(V)$ (i.e. a sequence $\{g_n\}$ which leaves every compact
      set in $\PGL(V)$) such that $g_n(z_1, x_n)$ converges to
      $(z_2^+,z_2^-) \subset \Omega_2$ and $g_n \Omega_1$ converges to
      $\Omega_2$.
    \end{enumerate}
    Then we say that $(z_1, \Omega_1)$ is \emph{forward conically
      related to} $(z_2^+, \Omega_2)$ by the sequence $\{g_n\}$, and
    \emph{backward conically related to} $(z_2^-, \Omega_2)$ by the
    sequence $\{g_n\}$.
  
  If the domains $\Omega_1, \Omega_2$ are understood from context, we
  will sometimes just say that $z_1$ is (forward or backward)
  conically related to $z_2^+$ or $z_2^-$.
\end{definition}

\begin{figure}[h]
    \centering
    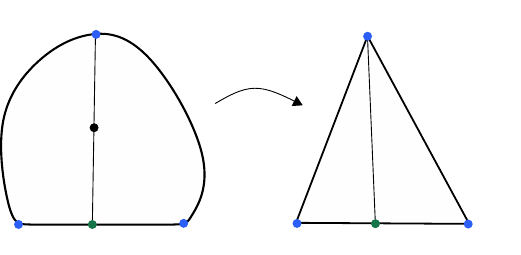
    \caption{Forward and backward conically related points (\cref{eg:fb_conical_pts})}
    \label{fig:conically_rel}
\end{figure}
\begin{example}[See \Cref{fig:conically_rel}]
\label{eg:fb_conical_pts}
    Let $\{e_1,e_2,e_3\}$ be the standard basis of $\Rb^3$ and $m:=(e_2+e_3)$.
    Consider a properly convex domain $\Omega_1$ such that $[e_1], [e_2], [e_3] \in \partial \Omega_1$; the projective line segment  $([e_2],[e_3]) \subset \partial \Omega_1$; and $([e_1],[m]) \subset \Omega_1$. Let $\ell_1$ be the projective geodesic ray in $\Omega_1$ starting at $[e_1+m]$ with endpoint $[m]$. For each $n\geq 1$, let $g_n$ be a diagonal matrix ${\rm diag}(n^2, 1/n, 1/n)$ in $\PGL(\Rb^3)$. Then $g_n\Omega_1$ converges to the projective 2-simplex $\Delta:=\{ [x:y:z] \in \Pb(\Rb^3) | x,y,z>0\}$ while $g_n \ell_1$ converges to the bi-infinite projective geodesic $\ell_0:=\{ [t:s:s] | t,s>0\}$ in $\Delta$. Thus $([m],\Omega_1)$ is forward conically related to $([m],\Delta)$ and backward conically related to $([e_1],\Delta).$ 
\end{example}

Observe that if $(z_1, \Omega_1)$ is (forward or backward) conically
related to $(z_2, \Omega_2)$, it follows immediately that for any
$g_1, g_2 \in \PGL(V)$, $(g_1z_1, g_1\Omega_1)$ is also (forward or
backward) conically related to $(g_2z_2, g_2\Omega_2)$. That is:
\begin{proposition}
  \label{prop:conical_relation_projective_equiv}
  The relation ``$(z_1, \Omega_1)$ is conically related to
  $(z_2, \Omega_2)$'' is well-defined when we regard $(z_i, \Omega_i)$
  as elements in the quotient set
  \[
    \{(x, \Omega) \in \Pb(V) \times \domains(V) : x \in \dee \Omega\} /
    \PGL(V),
  \]
  where $\PGL(V)$ acts diagonally on $\Pb(V) \times \domains(V)$.
\end{proposition}

Later, we will prove a number of other straightforward but useful
properties of conically related points. In particular, we will show
that Morseness is preserved between geodesics with conically related
endpoints (see \Cref{lem:morse_conically_related}).

\subsection{Characterizations of Morseness}

We can now state our full characterization of Morse projective
geodesics, giving an expanded version of
\Cref{prop:morse_contracting_slim}. For convenient reading, we have given the various characterizations suggestive labels instead of numbers: M for ``Morse'', HT for ``half triangle'', V for ``visibility'',  and C for ``contracting.''

\begin{proposition}
  \label{prop:morse_hilbert_equivalences} Suppose $\ell$ is a
  projective geodesic in a properly convex domain $\Omega$. Then the
  following are equivalent:
  \begin{enumerate}
    \labelitem{($M$)} \label{item:nondiv_morse} The geodesic $\ell$ is
    Morse for some Morse gauge $M$.

	\labelitem{($HT$)} \label{item:halftri_forward} For every endpoint
    $z_1$ of $\ell$ in $\dee \Omega$, if $z_1$ is forward conically
    related to a point $z_2 \in \dee \Omega_2$, then $z_2$ does not
    lie in the boundary of a half-triangle in $\Omega_2$.

    \labelitem{($HT-$)} \label{item:halftri_backward} For every
    endpoint $z_1$ of $\ell$ in $\dee \Omega$, if $z_1$ is backward
    conically related to a point $z_2 \in \dee \Omega_2$, then $z_2$
    does not lie in the boundary of a half-triangle in $\Omega_2$.

    \labelitem{($V$)} \label{item:segclose_forward} For every endpoint
    $z_1$ of $\ell$ in $\dee \Omega$, if $z_1$ is forward conically
    related to a point $z_2 \in \dee \Omega_2$, then
    $(z_2, w) \subset \Omega_2$ for every
    $w \in \partial \Omega_2 \minus \{z_2\}$.

    \labelitem{($V-$)} \label{item:segclose_backward} For every
    endpoint $z_1$ of $\ell$ in $\dee \Omega$, if $z_1$ is backward
    conically related to a point $z_2 \in \dee \Omega_2$, then
    $(z_2, w) \subset \Omega_2$ for every
    $w \in \partial \Omega_2 \minus \{z_2\}$.

    \labelitem{($\delta$)} \label{item:delta_slim} The geodesic $\ell$
    is projectively $\delta$-slim for some $\delta > 0$.

    \labelitem{($C$)} \label{item:contracting} The geodesic $\ell$ is
    $D$-contracting for some $D$.
  \end{enumerate} 
\end{proposition}

\begin{remark}
  \label{rem:compact_segment_morse_trivial}
  We allow the projective geodesic $\ell$ in the statement of
  \Cref{prop:morse_hilbert_equivalences} to have zero, one, or two
  endpoints in the boundary of the properly convex domain $\Omega$. In
  the case where $\ell$ has zero ideal endpoints (meaning it is a
  compact segment in $\Omega$), then the conditions
  \ref{item:halftri_forward}, \ref{item:halftri_backward},
  \ref{item:segclose_forward}, and \ref{item:segclose_backward} are
  vacuous. In this case, conditions \ref{item:nondiv_morse} and
  \ref{item:contracting} hold trivially since $\ell$ has finite
  diameter, and condition \ref{item:delta_slim} follows from
  \Cref{lem:hil_haus_dist_between_geods}.
\end{remark}

The proof of \Cref{prop:morse_hilbert_equivalences} follows the scheme
given in \Cref{fig:morse_hilbert_diagram} below. Each implication is
labeled with the number of the intermediate result(s) that provide
its proof.
\begin{figure}[h!]
  \begin{center}
    \begin{tikzcd}[arrows=Rightarrow]
      &  &  & (V) \arrow[r] & (HT) \arrow[dr, "\ref{lem:halftri_projective_delta_slim}" description] &  \\
      (\delta) \arrow{r}{\ref{prop:slim_contracting}} & (C) \arrow{r}{\ref{prop:contracting_implies_morse}} & (M) \arrow[ur,"\ref{cor:morse_not_in_segment} + \ref{lem:morse_conically_related}" description] \arrow[dr,"\ref{cor:morse_not_in_segment} + \ref{lem:morse_conically_related}" description] &  & & (\delta)\\
      &  &  & (V-) \arrow[r] & (HT-) \arrow[ur,"\ref{lem:halftri_projective_delta_slim}" description] & 
    \end{tikzcd}
  \end{center}
  \caption{Proof outline for \Cref{prop:morse_hilbert_equivalences}}
  \label{fig:morse_hilbert_diagram}
\end{figure}

Note that we have already shown the implications \ref{item:delta_slim}
$\implies$ \ref{item:contracting} $\implies$
\ref{item:nondiv_morse}. There are no labels on the implications
\ref{item:segclose_forward} $\implies$ \ref{item:halftri_forward} and
\ref{item:segclose_backward} $\implies$ \ref{item:halftri_backward} as
they are immediate. Indeed, if $z_2 \in \partial \Omega_2$ is in the
boundary of a half-triangle in $\Omega_2$, then there exists
$w \in \partial \Omega_2-\{z_2\}$ such that
$[z_2,w] \subset \partial \Omega_2$.

\begin{remark}
\label{rem:morseness_basept_indep}
    If $z$ is a point in $\bdry$, and $x, x'$ are points in $\Omega$, then \Cref{lem:hil_haus_dist_between_geods} implies that the two projective geodesic rays $[x, z)$ and $[x', z)$ are within bounded Hausdorff distance of each other. It follows easily that if $[x, z)$ is $M$-Morse for some Morse gauge $M$, then $[x', z)$ is $M'$-Morse for a (possibly different) Morse gauge $M'$. For this reason, we will sometimes call a point $z \in \bdry$ ``Morse'' if some (any) projective geodesic ray $[x, z)$ is Morse.  
\end{remark}

\subsection{Projective geodesics in triangles and half-triangles}

The first step towards proving the remaining implications in
\Cref{prop:morse_hilbert_equivalences} is to observe that Morse
geodesics cannot have endpoints lying in the boundary of triangles or
half-triangles. This should be unsurprising if we accept that
triangles and half-triangles are analogs of flats and half-flats.

\begin{lemma}
  \label{lem:morse_not_in_tri}
  Suppose that $y \in \dee \Omega$ lies in the boundary of a properly
  embedded triangle in $\Omega$. Then for any $x \in \Omega$, the
  projective geodesic $[x, y)$ is not Morse.
\end{lemma}
\begin{proof}
  Since Morseness does not depend on the choice of basepoint, we can
  assume that $x$ lies in the interior of the properly embedded
  triangle $\Delta$ whose boundary contains $y$. Then the projective
  geodesic $[x, y)$ is also contained in $\Delta$. But $\Delta$ is
  quasi-isometric to a 2-flat, and 2-flats contain no Morse
  quasi-geodesics, so $[x, y)$ cannot be Morse.
\end{proof}

\begin{lemma}
  \label{lem:morse_not_in_halftri}
  Let $x, y, z$ be the vertices of a half-triangle in $\dee \Omega$
  with $(y, z) \subset \Omega$, and suppose that $[x, y]$ is a maximal
  segment in $\dee \Omega$. Then for any $w \in \Omega$, the
  projective geodesic $[w, y)$ is not Morse.
\end{lemma}

The proof of \Cref{lem:morse_not_in_halftri} is somewhat more
complicated than the proof of \Cref{lem:morse_not_in_tri}, because
half-triangles in a properly convex domain are not necessarily
quasi-isometric to half-flats. Our proof instead takes advantage of
the following result of Cordes:
\begin{lemma}[{\cite[Key Lemma]{Cordes17}}]
  \label{lem:cordes_lemma}
  Let $X$ be a geodesic metric space. For any Morse gauge $M$, there
  exists a constant $\delta_M$ so that, if $\alpha:[0, \infty) \to X$
  is an $M$-Morse geodesic ray, and $\beta:[0, \infty) \to X$ is a
  geodesic ray such that $\beta(0) = \alpha(0)$ and the images of
  $\alpha, \beta$ have finite Hausdorff distance, then for all
  $t \in [0, \infty)$ we have $d_X(\alpha(t), \beta(t)) < \delta_M$.
\end{lemma}

\begin{proof}[Proof of \Cref{lem:morse_not_in_halftri}]
  Let $x, y, z, w$ be as in the statement of the lemma. As in \Cref{rem:morseness_basept_indep}, since
  Morseness of a projective geodesic ray is independent of the choice of initial point for the ray, we may assume that $w$ actually
  lies in the convex hull of $x, y, z$. Consider the projective
  geodesic $[w, y)$, and fix a point $u \in (y, x)$ so that the
  projective line spanned by $u, w$ has its other ideal endpoint in
  the interval $(z, x)$. Let $c:[0, \infty) \to \Omega$ be a
  unit-speed parameterization of the geodesic ray $[w, y)$, and let
  $s:[0, \infty) \to \Omega$ be a unit-speed parameterization of
  $[w, u)$.

  For each $n \in \Nb$, let $r_n:[0, \infty) \to \Omega$ be a
  unit-speed parameterization of the projective geodesic $[s(n),
  y)$. Consider the sequence of ``broken geodesics''
  $c_n:[0, \infty) \to \Omega$ given by
  \[
    c_n(t) =
    \begin{cases}
      s(t), & t \le n\\
      r_n(t - n), &t > n
    \end{cases}.
  \]
  \Cref{fact:nonunique_geodesics} implies that each $c_n$ is actually
  a geodesic in $\Omega$, with endpoint $y$ (see
  \Cref{fig:nonmorse_geodesics}). Moreover, by
  \Cref{lem:hil_haus_dist_between_geods}, the Hausdorff distance
  between $c_n([0, \infty))$ and $(w,y)$ is bounded by
  $\hil(c_n(n),(w,y))$.

  Now suppose that $(w,y)$ is a $M$-Morse geodesic for some $M$. Then
  \Cref{lem:cordes_lemma} tells us that $\hil(c_n(n), c(n))$ is
  bounded above by a constant that depends only on $M$. As
  $n \to \infty$, the sequence $c_n(n)$ approaches $u$, and $c(n)$
  approaches $y$. Then \Cref{lem:hilbert_face_metric} implies that
  $u\in F_{\Omega}(y)$, which contradicts the maximality of the line
  segment $[x, y] \in \dee \Omega$. Thus $(w, y)$ cannot be Morse.\end{proof}
  
  \begin{figure}
    \centering
    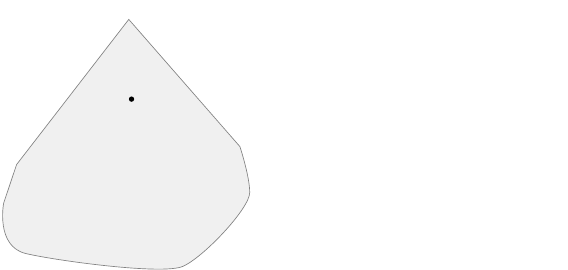
    \caption{Left: the sequence of ``broken geodesics'' $c_n$. Right:
      verifying that each $c_n$ is actually a geodesic, using
      \Cref{fact:nonunique_geodesics}.}
    \label{fig:nonmorse_geodesics}
  \end{figure}

\subsection{Properties of conically related points}

The previous two lemmas show that, for a projective geodesic $\ell$ in
$\Omega$, having an endpoint in a triangle or half-triangle is an
obstruction to Morseness for $\ell$. For the proof of
\Cref{prop:morse_hilbert_equivalences}, we want to show that having an
endpoint which is \emph{conically related} to an endpoint in a
triangle or half-triangle also obstructs Morseness; this will follow
from \Cref{lem:morse_conically_related} below. Before we state and
prove this lemma, however, we make a brief digression to develop the
theory of conically related points a little further.

First, observe that ``$(z_1, \Omega_1)$ is conically related to
$(z_2, \Omega_2)$'' is \emph{not} an equivalence relation, since it is
not in general symmetric. In addition, the relation is not even
necessarily reflexive, since we require the sequence of group elements
$g_n$ appearing in the definition to be divergent. However, the
relation does satisfy the following:
\begin{lemma}
\label{lem:conical_related_reflexive}
    Suppose $\Omega$ is a properly convex domain and $x \in \dee
    \Omega$. Then there exists a properly convex domain $\Omega'$ and
    $x'_{\pm}\in \partial \Omega'$ such that $(x,\Omega)$ is forward
    conically related to $(x'_+,\Omega')$ and backward conically
    related to $(x'_-,\Omega')$.
\end{lemma}
\begin{proof}
  Fix a basepoint $x_0 \in \Omega$ and pick a sequence $\{p_n\}$ in
  $[x_0,x)$ such that $p_n \to x$. Pick another sequence $\{q_n\}$
  such that $q_n \in [p_n,x)$ and $\hil(p_n,q_n)=n$. By
  \cref{thm:benzecri}, there exists a sequence $\{g_n\}$ in $\PGL(V)$
  such that $g_n(\Omega,q_n)$ converges to
  $(\Omega_{\infty},q_\infty) \in \Cc_{*}(V)$. Up to passing to a
  subsequence, we can assume that $g_np_n \to p_{\infty}$ and
  $g_nx \to x_\infty$.

  Since $x \in \partial \Omega$, we have
  $x_\infty \in \partial \Omega_\infty$. We also know that
  $p_\infty \in \partial \Omega_\infty$, because
  $\hil(q_n,p_n)\to \infty$ and $q_\infty \in \Omega_\infty$. But
  $(p_\infty,x_\infty) \subset \Omega_{\infty}$ as
  $q_\infty \in (p_\infty,x_\infty) \cap \Omega_\infty$. Thus
  $g_n(p_n,x) \to (p_\infty,x_\infty)$. Hence
  $g_n(x_0,x) \to (p_\infty,x_\infty)$.

  As the sequence $q_n \in (x_0,x)$ converges to
  $q_\infty \in \Omega_{\infty}$ under $g_n$, we see that $(x,\Omega)$
  is forward (resp. backward) conically related to
  $(x_\infty,\Omega_\infty)$ (resp. $(p_\infty,\Omega_\infty)$).
\end{proof}

\begin{remark}
  It turns out that the conical relation is also transitive, in the
  sense that, if $(z_1, \Omega_1)$ is forward conically to
  $(z_2, \Omega_2)$, and $(z_2, \Omega_2)$ is forward conically
  related to $(z_3, \Omega_3)$, then $(z_1, \Omega_1)$ is forward
  conically related to $(z_3, \Omega_3)$. The proof of this fact is
  straightforward; we omit it as we have no need for it in this paper.
\end{remark}

\subsubsection{Conically related points along $k$-sections}

It is often useful to consider the behavior of projective
automorphisms on a lower-dimensional ``projective slice'' of a convex
projective domain. Adopting the terminology of Benoist and Benz\'ecri, we refer to such slices as \emph{$k$-sections.}
\begin{definition}[{$k$-section \cite[Definition 2.5]{B2003}}]
  \label{defn:k_section}
  Let $\Omega$ be a properly convex domain in $\Pb(V)$. A
  \emph{$k$-section} $\omega$ of $\Omega$ is a nonempty intersection
  $\Pb(W) \cap \Omega$, where $\Pb(W)$ is a projective subspace of
  dimension $k$.
\end{definition}

Fix a $k$-dimensional projective space $\Pb(W_0)$ of $\Pb(V)$. Then
the space of $k$-dimensional projective subspaces of $\Pb(V)$ is the 
$\PGL(V)$ orbit of $\Pb(W_0)$. Thus, any $k$-section in $\Omega$ can be
canonically identified with a projective equivalence class of properly
convex domains in $W_0$. So, owing to
\Cref{prop:conical_relation_projective_equiv}, if $\Omega_1$ and
$\Omega_2$ are properly convex domains in $\Pb(V)$ and
$x_i \in \dee \omega_i$ for $k$-sections $\omega_i$ of $\Omega_i$
($i=1,2$), it makes sense to say that $(x_1, \omega_1)$ is (forward or
backward) conically related to $(x_2, \omega_2)$, as elements in
$W_0 \times \domains(W_0)$. The lemma below essentially follows from
\cite[Lemma 2.8]{B2003}:
\begin{lemma}
  \label{lem:conically_related_sections}
  Let $\Omega_1, \Omega_2$ be properly convex domains in $\Pb(V)$, and
  fix $1 \leq k < d$. Then $(x_1, \Omega_1)$ is (forward or backward)
  conically related to $(x_2, \Omega_2)$ if and only if there are
  $k$-sections $\omega_1, \omega_2$ so that $x_i \in \dee \omega_i$ for
  $i = 1,2$, and $(x_1, \omega_1)$ is (forward or backward) conically
  related to $(x_2, \omega_2)$.
\end{lemma}

\subsubsection{Uniqueness for conically related points}

In general, a pair $(x_1, \Omega_1)$ may be conically related to many
different pairs $(x_2, \Omega_2)$, even up to projective
equivalence. However, as a basic application of \Cref{thm:benzecri},
we can recover some uniqueness given additional information about the
sequence $\{g_n\}$ realizing the conical relation.

\begin{definition}
\label{defn:conically_related_along_a_sequence}
If $(x_1, \Omega_1)$ is (forward or backward) conically related to
$(x_2, \Omega_2)$ by some $g_n \in \PGL(V)$, and there is some
sequence $\{p_n\}$ in $\Omega_1$ so that $g_n(\Omega_1, p_n)$
converges in $\pdomains(V)$, then we say that $(x_1, \Omega_1)$ is
conically related to $(x_2, \Omega_2)$ \emph{along the sequence
  $\{p_n\}$}.
\end{definition}

\begin{lemma}
  \label{lem:changing_conically_related_sequence}
  Let $x_1$ be a point in the boundary of a properly convex domain
  $\Omega_1$, and suppose that $(x_1, \Omega_1)$ is forward
  (resp. backward) conically related to both $(x_2, \Omega_2)$ and
  $(x_2', \Omega_2')$ along the same sequence $\{p_n\}$ in
  $\Omega$. Then there is some $h \in \PGL(V)$ such that
  $(hx_2, h\Omega) = (x_2', \Omega_2')$.
\end{lemma}
\begin{proof}
  Consider sequences $\{g_n\}, \{g_n'\}$ in $\PGL(V)$ so that
  $(x_1, \Omega_1)$ is conically related to $(x_2, \Omega_2)$ by
  $g_n$, $(x_1, \Omega_1)$ is conically related to $(x_2', \Omega_2')$
  by $\{g_n'\}$, and the sequences $g_n(\Omega_1, p_n)$ and
  $g_n'(\Omega_1, p_n)$ both converge in the space $\pdomains(V)$ of
  pointed domains in $\Pb(V)$.

  This means that we can find compact subsets $\Kc, \Kc'$ in
  $\pdomains(V)$ so that the intersection
  $g_n'g_n^{-1}\Kc \cap \Kc' \ne \emptyset$. From \Cref{thm:benzecri},
  it then follows that $g_n' = k_ng_n$ for $k_n$ in a fixed compact
  subset of $\PGL(V)$. Any subsequence of $k_n$ has a further
  subsequence which converges to some $h \in \PGL(V)$; it follows that
  $h$ takes the limit of $g_n(x_1, \Omega_1)$ to the limit of
  $g_n'(x_1, \Omega_1)$, i.e. $h(x_2, \Omega_2) = (x_2', \Omega_2')$.
\end{proof}

\subsection{Points conically related to Morse points}

We now return to our main task of proving
\Cref{prop:morse_hilbert_equivalences}. The next lemma is a key tool
we need for several of the equivalences in the proposition. It says
that Morseness is preserved (in one direction) along a conical
relation.

\begin{lemma}
  \label{lem:morse_conically_related}
  Let $y_1 \in \dee \Omega_1$, and suppose that the projective
  geodesic $[x_1, y_1)$ is Morse for some (any) $x_1 \in \Omega_1$.
  If $y_1$ is forward or backward conically related to
  $y_2 \in \dee \Omega_2$, then for some (any) $x_2 \in \Omega_2$, the
  projective geodesic $[x_2, y_2)$ is Morse.
\end{lemma}
\begin{proof}
  We first remark that the choice of $x_1$ and $x_2$ in the statement
  of the lemma is not significant, since the Morseness of a geodesic
  ray is independent of its basepoint. So, fix any $x_1$ in $\Omega_1$ and $y_1 \in \partial \Omega_1$. We
  will prove the contrapositive of the desired statement, and show
  that if $y_1 \in \dee \Omega_1$ is forward or backward conically
  related to $y_2 \in \dee \Omega_2$, and $[x_2,y_2)$ is non-Morse for
  some $x_2 \in \Omega_2$, then $[x_1,y_1)$ is also non-Morse.
  
  Let $(z_1,y_1)$ be the bi-infinite projective geodesic in $\Omega_1$
  that contains $[x_1,y_1)$. As $y_1$ is conically related to $y_2$,
  there is a sequence $\{g_n\}$ in $\PGL(V)$ so that
  $g_n\Omega_1 \to \Omega_2$ and $y_2$ is the limit of either $g_ny_1$
  or $g_nz_1$ (depending on whether $y_1$ is forward or backward
  conically related to $y_2$). By definition of the conical relation,
  there exists $(z_2,y_2) \subset \Omega_2$ such that
  $g_n(x_1,y_1) \to (z_2,y_2)$. Fix a point $x_2 \in (z_2,y_2)$.
  
  Assume that the projective geodesic ray $[x_2, y_2)$ is not
  Morse. This means that there exist quasi-geodesic constants
  $K \ge 1, C \ge 0$ such that for every $m \in \Nb$, there is a
  $(K,C)$-quasi-geodesic $q_m:[0, T_m] \to \Omega_2$ with endpoints in
  $[x_2, y_2)$ such that the image of $q_m$ does not lie in the
  $m$-neighborhood of $[x_2, y_2)$.

  We now claim that there exist constants $K',C'$ such that: for any
  $m\in \Nb$, there exists a $(K', C')$-quasi-geodesic
  $q_m':[0, T_m] \to \Omega_1$ with endpoints on $[x_1, y_1)$, but not
  contained in the $(m-1)$-neighborhood of $[x_1,y_1)$ in the metric
  $d_{\Omega_1}$. This claim essentially follows from the fact that
  the convergence of $g_n\Omega_1$ to $\Omega_2$ in $\Cc(V)$ is
  uniform on compact subsets of $\Pb(V)$ that intersect $\Omega_2$.

  Fix any $m \in \Nb$, and pick a compact convex set
  $D_m \subset \Omega_2$ large enough to contain the $m$-neighborhood
  of the set $q_m([0, T_m])$. Then for sufficiently large $n$
  (depending on $m$), the subset $D_m$ is contained in
  $g_n\Omega_1$. Moreover, we have
  \[
    d_{g_n \Omega_1}|_{D_m\times D_m} \to d_{\Omega_2}|_{D_m\times
      D_m} 
  \]
  uniformly as $n \to \infty$.

  As $q_m(0),q_m(T_m) \in (z_2,y_2)$, the projective geodesic
  $(z_2,y_2)$ intersects $D_m$ in a finite length projective geodesic
  segment. As $n$ tends to infinity, we have
  $g_n(x_1,y_1) \cap D_m \to (z_2,y_2) \cap D_m$. Hence, for sufficiently large $n$, the endpoints
  $q_m(0),q_m(T_m)$ lie at a distance at most 1 from $g_n[x_1, y_1)$,
  with respect to the Hilbert metric on $\Omega_2$. So, for each
  sufficiently large $n$, we can define a map
  $q_{m,n}:[0, T_m] \to \Omega_2$, agreeing with $q_m$ on the open
  interval $(0, T_m)$, and whose endpoints lie on the ray
  $g_n[x_1, y_1)$ at a distance at most 1 from the endpoints of
  $q_m([0,T_m])$. The image of each $q_{m,n}$ lies in the set $D_m$. Since
  $q_m$ is a $(K, C)$-quasi-geodesic with respect to the Hilbert
  metric on $\Omega_2$, $q_{m,n}$ must be a
  $(K, C + 1)$-quasi-geodesic with respect to the same metric.

  Now, we know that the Hilbert on $g_n\Omega_1$ converges to the
  Hilbert distance on $\Omega_2$ uniformly on $D_m$. So, if we fix
  $K' = K + 1$ and $C' = C + 2$, then for $n$ large enough, the map
  $q_{m,n}:[0, T_m) \to \Omega_2$ must also be a
  $(K', C')$-quasi-geodesic with respect to the Hilbert metric on
  $g_n\Omega_1$.

  By construction of $q_m$, we also know that there is some
  $t_m \in (0, T_m)$ so that the $(m-1)$-ball $B_2$ about $q_m(t_m)$
  (with respect to the Hilbert metric on $\Omega_2$) does not
  intersect the geodesic $(z_2, y_2)$. Letting $B_{1,n}$ be the
  $(m-1)$-ball about $q_m(t_m)$ with respect to the Hilbert metric on
  $g_n\Omega_1$, the uniform convergence of Hilbert metrics on $D_m$
  implies that $B_{1,n} \subset D_m$ for large enough $n$ and that
  $B_{1,n} \to B_2$ as $n \to \infty$. Then, as
  $g_n(x_1, y_1) \cap D_m$ converges to $(z_2, y_2) \cap D_m$, for
  large enough $n$ we see that $B_{1,n}$ cannot intersect the
  projective geodesic $g_n(x_1, y_1)$.

  This implies that, for all sufficiently large $n$, the
  quasi-geodesic $q_{m,n}$ is not contained in the
  $(m-1)$-neighborhood of $g_n(x_1, y_1)$ with respect to the Hilbert
  metric on $g_n\Omega_1$. But then $g_n^{-1}q_{m,n}$ is a
  $(K', C')$-quasi-geodesic with endpoints on $[x_1, y_1)$, whose
  image does not lie in the $(m-1)$-neighborhood of $[x_1, y_1)$ with
  respect to the Hilbert metric on $\Omega_1$. Since $m$ was
  arbitrary, and $K', C'$ are independent of $m$, this proves that
  $[x_1, y_1)$ cannot be Morse.
\end{proof}

Combining the above lemma with \Cref{lem:morse_not_in_tri} and
\Cref{lem:morse_not_in_halftri}, we obtain a direct proof of the
implications \ref{item:nondiv_morse} $\implies$
\ref{item:halftri_forward} and \ref{item:nondiv_morse} $\implies$
\ref{item:halftri_backward} in
\Cref{prop:morse_hilbert_equivalences}. However, we need to do some
more work to prove the implications \ref{item:nondiv_morse} $\implies$
\ref{item:segclose_forward} and \ref{item:nondiv_morse} $\implies$
\ref{item:segclose_backward}. Note that these implications are not necessary for the proof of \Cref{prop:contracting_implies_morse}. We will use them at the end of the section to prove another characterization of Morseness (\Cref{cor:morse_point_is_C1_extreme}) which will be needed in \Cref{sec:sv_morse}.

\subsection{Conically related points in triangles and half-triangles}

The lemma below is well-known to experts, and a similar proof already
appears in \cite{benzecri}. This result expresses the idea that, in
any domain $\Omega$, segments (or non-$C^1$ points) in the boundary
correspond to ``flat directions:'' as we follow a projective geodesic
towards a segment or corner in $\dee \Omega$, the domain ``looks more
like'' a domain containing a properly embedded triangle, with the
original segment or corner in its boundary. We give a statement which
uses the language of conically related points, and provide a proof for
convenience.
\begin{lemma}
  \label{lem:int_segments_triangles}
  Suppose that $x_1 \in \dee \Omega_1$ is forward conically related to
  $x_2^+ \in \dee \Omega_2$, and backward conically related to
  $x_2^- \in \dee \Omega_2$.
  \begin{enumerate}
  \item If $x_1$ lies in the interior of a nontrivial segment in
    $\dee \Omega_1$, then there is a properly embedded triangle
    $\Delta$ in $\Omega_2$ so that $x_2^+$ lies in the interior of an
    edge of $\Delta$, and $x_2^-$ is a vertex of $\Delta$.
  \item If $x_1$ is not a $C^1$ point, then there is a properly
    embedded triangle $\Delta$ in $\Omega_2$ so that $x_2^+$ is a
    vertex of $\Delta$ and $x_2^-$ is on the interior of an edge of
    $\Delta$.
  \end{enumerate}
\end{lemma}
\begin{proof}
  Via \Cref{lem:conically_related_sections}, it suffices to consider
  the case where $\Omega_1, \Omega_2$ are 2-dimensional. The
  The definition of conically related points implies that there exists a
  projective geodesic ray $[a,x_1) \subset \Omega_1$, a sequence in
  $[a,x_1)$, and a sequence $\seq{g_n}$ in $\PGLdR$ such that
  $g_n[a,x_1) \to (x_2^-,x_2^+)\subset \Omega_2$. Let $x_1^-$ be a
  point in $\partial \Omega_1$ so that $[a,x_1) \subset (x_1^-, x_1)$,
  and hence, $(x_1^-,x_1^+)\subset \Omega_1$. Then $g_n(x_1^-, x_1)$
  converges to $(x_2^-, x_2^+)$. Let $\{p_n\}$ be a sequence in
  $[a, x_1)$ so that $g_np_n$ converges to some point in the interior
  of $(x_2^-,x_2^+)$.

  (1)  See \Cref{fig:conically_rel} for an illustration of the main idea behind this proof.
  By assumption, there exists a maximal nontrivial projective line
  segment $[b,c] \subset \partial \Omega_1$ with $x_1 \in
  (b,c)$. Consider a sequence of projective transformations $h_n$,
  defined (with respect to the projective basis $\{b, c, x_1^-\}$) by
  \[
    h_n :=
    \begin{pmatrix}
      \lambda_n^{-1}\\ & \lambda_n^{-1} \\ & & \lambda_n^2
    \end{pmatrix},
  \]
  where $\lambda_n$ is chosen so that $h_np_n$ converges to a point in
  the interior of the line segment $x_1, x_1^-$. Then $h_n\Omega_1$
  converges to the triangle with vertices at $b, c, x_1^-$. The result
  then follows from \Cref{lem:changing_conically_related_sequence}.

  (2) This case is similar. Fix a supporting line $L_-$ for $\Omega_1$
  at the point $x_1^-$. Since $x_1$ is not a $C^1$-point, we can
  choose two distinct supporting hyperplanes of $\Omega_1$ at $x_1$
  that we label $H_b$ and $H_c$. Let $b=H_b \cap L_-$ and
  $c=H_c \cap L_-$. Here we consider the sequence of projective
  transformations (defined with respect to the projective basis
  $\{x_1, b, c\}$) by
  \[
    h_n :=
    \begin{pmatrix}
      \lambda_n^{-2} \\ & \lambda_n \\ & & \lambda_n
    \end{pmatrix},
  \]
  where $\lambda_n$ is again chosen so that $h_np_n$ converges to a
  point in the interior of $(x_1, x_1^-)$. This time, since $\Omega_1$
  is not $C^1$ at $x_1$, the sequence of domains converges to a
  triangle with a vertex at $x_1$, and an edge containing $x_1^-$ in
  its interior and again we are done by
  \Cref{lem:changing_conically_related_sequence}.
\end{proof}

The next lemma does not appear to be well-known. It says that, just as
a point $z$ in the interior of a boundary segment in a properly convex
domain $\Omega$ can be thought of as a ``flat direction,'' a point $z$
in the \emph{closure} of a segment can be thought of as a
``half-flat'' direction: as we approach $z$ along a projective
geodesic, the domain ``looks more'' like it contains a properly
embedded half-triangle.

\begin{lemma}
  \label{lem:boundary_segment_halftri}
  Suppose that $x_1 \in \dee \Omega$ is forward conically related to
  $x_2^+ \in \dee \Omega_2$ and backward conically related to
  $x_2^- \in \dee \Omega_2$. If $x_1$ lies in the closure of a
  nontrivial segment in $\dee \Omega$, then both $x_2^+$ and $x_2^-$
  lie in the boundary of a half-triangle in $\Omega_2$.
\end{lemma}
\begin{proof}
  After applying \Cref{lem:conically_related_sections} we may assume
  that $\Omega$ and $\Omega_2$ are both two-dimensional, and using
  \Cref{lem:int_segments_triangles}, we can further reduce to the case
  where $x_1$ is the endpoint of a maximal nontrivial segment in
  $\dee \Omega$. Let $z$ be the other endpoint of this segment, and
  let $L_+$ be the projective span of $x_1$ and $z$, so that $L$ is a
  supporting line of $\Omega$ at $x_1$.

  Fix a sequence $\seq{g_n}$ realizing the conical relations between
  $x_1$ and $x_2^\pm$, so that $g_nx_1 \to x_2^+$ and for some
  $x_1^- \in \dee \Omega$, we have $(x_1^-, x_1) \subset \Omega$ and
  $g_nx_1^- \to x_2^-$. Let $L_-$ be a supporting line of $\Omega$ at
  $x_1^-$, let $x_0 = L_- \cap L_+$, and let $p_n \in (x_1, x_1^-)$ be
  a sequence converging to $x_1$ so that $g_np_n$ converges to a point
  $p_0 \in \Omega_2$.

  We fix lifts $v_1, v_0, v_1^-$ for $x_1, x_0, x_1^-$ respectively,
  so that $\{v_1, v_0, v_1^-\}$ is a basis for $\Rb^3$ and the
  projectivization of the convex hull of $v_1, v_0, v_1^-$ lies in
  $\Omega$. We consider a sequence of linear maps $\{h_n\}$, defined
  with respect to this basis by
  \[
    h_n :=
    \begin{pmatrix}
       \lambda_n^{-1} \\ & \lambda_n^{-1} \\ & & \lambda_n^2
    \end{pmatrix}.
  \]
  Here $\lambda_n > 0$ is chosen so that $h_np_n$ converges to a point
  $p_0' \in (x_1, x_1^-)$. The sequence of domains $h_n\Omega$
  converges to a triangle with vertices $x_1, x_1^-, z$ (so this
  triangle does \emph{not} contain $p_0'$ in its interior). See \Cref{fig:halftri_conically_related}.

        \begin{figure}[h]
      \centering
      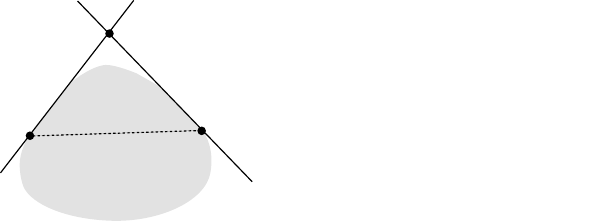
      \caption{Showing that $x_1$ is conically related to a point in the boundary of a half-triangle in \Cref{lem:boundary_segment_halftri} (Step 1). The domains $h_n\Omega$ converge to a triangle.}
      \label{fig:halftri_conically_related}
  \end{figure}

  Our chosen basis $\{v_1, v_0, v_1^-\}$ also determines projective
  coordinates $[x : y : z]$ for projective space
  $\Pb(\Rb^3)$. Consider the affine chart
  \[
    \{[x : y : 1 - x] \st x, y \in \Rb\} \simeq \Rb^2,
  \]
  which has affine coordinates given by $(x,y)$. In these coordinates,
  the projective line spanned by $x_1, x_1^-$ corresponds to the
  horizontal axis $y = 0$, so without loss of generality the triangle
  limited to by $h_n\Omega$ is a bounded convex subset of the upper
  half-plane. Therefore, since $h_np_n$ lies in the interior of
  $h_n\Omega$, the intersection of $h_n\Omega$ with the lower-half
  plane is nonempty, and contained in an open subset of the form
  $I \times (0, -\eps_n)$, where $I$ is a fixed interval and $\eps_n$
  is a positive constant tending to zero.

  We can then compose $h_n$ with a projective transformation $f_n$
  given by a ``vertical rescaling'' (i.e. a transformation which
  preserves vertical lines, and acts on them by homotheties centered
  at zero) so that the intersection of $f_nh_n\Omega$ with the lower
  half-plane converges to a bounded nonempty convex set; explicitly,
  in our chosen projective basis, each $f_n$ has the form
  \[
    \begin{pmatrix}
      \eta_n^{-1}\\ & \eta_n^2 \\ & & \eta_n^{-1}
    \end{pmatrix}.
  \]
  Since $\eps_n \to 0$, the vertical scaling factor of each $f_n$ must
  tend to infinity, which means that the intersection of
  $f_nh_n\Omega$ with the upper half-plane limits to a subset of the
  form $I \times (0, \infty)$.  But this subset is projectively
  equivalent to a half-triangle in the limit of $f_nh_n\Omega$. See \Cref{fig:halftri_conically_related_2}.
  \begin{figure}
      \centering
      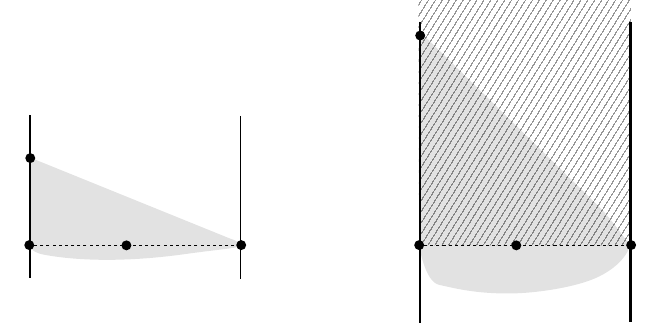
      \caption{Step 2 in \Cref{lem:boundary_segment_halftri}. The point $x_0$ is at infinity, so the diagonally shaded region is projectively equivalent to a half-triangle.}
      \label{fig:halftri_conically_related_2}
  \end{figure}

  Altogether, we have seen that the sequence of pointed domains
  $f_nh_n(\Omega, p_n)$ converge to a pointed domain
  $(\Omega_2', p_0')$, so that $f_nh_nx_1 = x_1$ and
  $x_1^- = f_nh_nx_1^-$ both lie in the boundary of a half-triangle in
  $\Omega_2'$. We can then apply
  \Cref{lem:changing_conically_related_sequence} to complete the proof.
\end{proof}

\subsection{Projective geodesics with endpoints in segments}

We can combine our previous results regarding Morse geodesics and
conically related points to prove some more facts about the endpoints of
Morse geodesics. 
\begin{corollary}
    \label{cor:morse_point_is_C1_extreme}
    Suppose that $y \in \bdry$ is the endpoint of a $M$-Morse geodesic
    ray. Then $y$ is a $C^1$ extreme point of $\bdry$.
\end{corollary}
\begin{proof}
  Fix a projective geodesic ray $[y_0,y)$ that is $M$-Morse. By
  \cref{lem:conical_related_reflexive}, $(y,\Omega)$ is forward
  conically related to $(y',\Omega')$. Now suppose that $y$ is not an
  extreme point, meaning that $y$ is contained in the interior of a
  projective line segment in $\bdry$. So
  \cref{lem:int_segments_triangles} part (1) implies that there is a
  properly embedded triangle $\Delta' \subset \Omega'$ whose boundary
  contains $y'$. Let $p' \in \Delta'.$ By
  \cref{lem:morse_conically_related}, $[p',y')$ is a Morse geodesic
  ray. This contradicts \cref{lem:morse_not_in_tri}, so $y$ is an
  extreme point. That $y$ is a $C^1$ point follows from similar
  reasoning, applying part (2) of \cref{lem:int_segments_triangles}
  instead of part (1).
\end{proof}

We can use this result to obtain:
\begin{corollary}
  \label{cor:morse_not_in_segment}
  Suppose that $y$ lies in the closure of a nontrivial segment in
  $\partial \Omega$. Then for any $x \in \Omega$, the projective
  geodesic $[x, y)$ is not Morse.
\end{corollary}
\begin{proof}
  Suppose, for a contradiction, that $[x,y)$ is $M$-Morse for some
  $x \in \Omega$. In this case \Cref{cor:morse_point_is_C1_extreme}
  implies that $y$ is a $C^1$ extreme point. So, we may assume that
  $y$ is the endpoint of a nontrivial segment in $\bdry$.

  By \Cref{lem:conical_related_reflexive}, $(y,\Omega)$ is forward
  conically related to $(y_+,\Omega_{\infty})$ and backward conically
  related to $(y_-,\Omega_{\infty})$ for some properly convex domain
  $\Omega_\infty$. Fix a point $p \in \Omega_\infty$. By
  \Cref{lem:morse_conically_related}, $[p,y_\infty)$ is also Morse,
  and by \Cref{lem:boundary_segment_halftri}, $y_\infty$ lies in the
  boundary of a half-triangle in $\partial \Omega_\infty$.

  Now, \Cref{cor:morse_point_is_C1_extreme} again implies that
  $y_\infty$ cannot lie in the interior of a segment in
  $\partial \Omega_\infty$. But in this case
  \Cref{lem:morse_not_in_halftri} implies that $[p, y_\infty)$ cannot
  be Morse and we get a contradiction.
\end{proof}

Combining the above with \Cref{lem:morse_conically_related}
immediately yields the implications \ref{item:nondiv_morse} $\implies$
\ref{item:segclose_forward} and \ref{item:nondiv_morse} $\implies$
\ref{item:segclose_backward} in
\Cref{prop:morse_hilbert_equivalences}.

\subsection{$\delta$-slimness}

To finish the proof of \Cref{prop:morse_hilbert_equivalences}, we need
to prove the final two implications \ref{item:halftri_forward}
$\implies$ \ref{item:delta_slim} and \ref{item:halftri_backward}
$\implies$ \ref{item:delta_slim} (again see
\Cref{fig:morse_hilbert_diagram}). These both follow from the lemma
below.
\begin{lemma}
  \label{lem:halftri_projective_delta_slim}
  Let $\ell$ be a projective geodesic in $\Omega$. If $\ell$ is not
  projectively $\delta$-slim for any $\delta > 0$, then there is an
  endpoint $y$ of $\ell$ in $\dee \Omega$ and points $z_+, z_-$ lying
  in the boundary of a half-triangle in some domain $\Omega_2$ so that
  $y$ is forward (resp. backward) conically related to $z_+$
  (resp. $z_-$).
\end{lemma}
\begin{proof}
  The argument is essentially identical to the proof of Proposition
  2.5 in \cite{B2004}; we reproduce it here for convenience. Fix a
  projective geodesic $\ell$ which is not projectively $\delta$-slim
  for any $\delta$. We choose a sequence of triples
  $\{(a_n, b_n, c_n)\}$ in $\Omega$, with $a_n, b_n \in \ell$, such
  that the projective geodesic triangle with vertices $a_n, b_n, c_n$
  is not $2n$-slim. Then, by \cref{rem:proj_delta_slim_defn}, the
  segment $[a_n, b_n]$ cannot be contained in the union of metric
  $n$-neighborhoods
  \[
    N_n([a_n, c_n]) \cup N_n([b_n, c_n]).
  \]
   Since the projective geodesic segment $[a_n, b_n]$ is connected, there is a point
  $x_n \in [a_n, b_n]$ so that $\hil(x_n, [a_n, c_n]) \ge n$ and
  $\hil(x_n, [b_n, c_n]) \ge n$. Applying \Cref{thm:benzecri}, we can
  choose elements $g_n \in \PGL(V)$ and extract a subsequence so that
  the pointed domains $g_n(\Omega, x_n)$ converge to some limiting
  pointed domain $(\Omega_\infty, x_\infty)$, and the points
  $g_na_n, g_nb_n, g_nc_n$ converge to points $a, b, c$ in
  $\dee \Omega_\infty$.

  Since $g_nx_n$ converges to $x_\infty \in \Omega_\infty$, the
  distances $\dist_{g_n\Omega}(g_nx_n, g_n[a_n, c_n])$ and
  $\dist_{g_n\Omega}(g_nx_n, g_n[b_n, c_n])$ must tend to infinity,
  which means the segments $[a,c]$ and $[b,c]$ must converge to
  subsets of $\dee \Omega_\infty$. However, since the limit of
  $g_nx_n$ lies in the interior of $(a, b) \cap \Omega_\infty$, the
  segments $[a,c]$ and $[b,c]$ must also be nontrivial and
  distinct. As $(a,b)$ contains the limit of
  $g_nx_n \in \Omega_\infty$, the points $a, b, c$ are the vertices of
  a half-triangle in $\Omega_\infty$.

  If $\{g_n\}$ is a divergent sequence in $\PGL(V)$, then the
  properness condition in \cref{thm:benzecri} implies that $x_n$ must
  tend towards an endpoint of $\ell$ in $\dee \Omega$. This endpoint
  is forward conically related to one of the limiting endpoints $a, b$
  of $g_n[a_n, b_n]$, and backward conically related to the other. In
  this case, we have proved the lemma.

  On the other hand, if $\{g_n\}$ is \emph{not} divergent in
  $\PGL(V)$, then
  $(\Omega_\infty, [a,b]) = g(\Omega, \overline{\ell})$ for some
  $g \in \PGL(V)$. Thus both endpoints of $\ell$ already lie in a
  half-triangle in $\Omega$. If the conical relation were reflexive,
  this would finish the proof. But since conical relation satisfies
  only a weak form of reflexivity, we must appeal to
  \cref{lem:conical_related_reflexive} followed by
  \cref{lem:boundary_segment_halftri}. This leads to the conclusion
  that the endpoints of $\ell$ are both forward and backward conically
  related to points in a half-triangle $\Delta$ in some properly
  convex domain $\Omega'$, as required.
\end{proof}

This concludes the proof of \Cref{prop:morse_hilbert_equivalences},
hence of \Cref{prop:morse_contracting_slim}.

\subsection{Uniformity}
\label{sec:uniform_morse_contracting}

\Cref{prop:contracting_implies_morse} gives us a stronger version of
the implication \ref{item:contracting} $\implies$
\ref{item:nondiv_morse} in \Cref{prop:morse_hilbert_equivalences}: it
says that any $D$-contracting projective geodesic in a properly convex
domain $\Omega$ is $M$-Morse for a Morse gauge $M$ determined solely
by $D$ and $\Omega$. In the case where $\Omega$ is divisible, we can
strengthen the opposite implication in a similar manner.

\begin{proposition}
  \label{prop:morse_contracting_uniform}
  Let $\Omega$ be a properly convex divisible domain. For every Morse
  gauge $M$, there exists a constant $\delta > 0$ (depending only on $M$
  and $\Omega$) so that any $M$-Morse geodesic in $\Omega$ is
  projectively $\delta$-slim.
\end{proposition}
Observe that, by applying this proposition together with
\cref{prop:slim_contracting}, we obtain the following uniform version
of \ref{item:nondiv_morse} $\implies$ \ref{item:contracting}:
\begin{corollary}
  \label{cor:morse_contracting_uniform}
  Let $\Omega$ be a properly convex divisible domain, $M$ be a Morse
  gauge, and $\delta$ be the constant (determined solely by $M$ and
  $\Omega$) from \cref{prop:morse_contracting_uniform}. Then any
  $M$-Morse geodesic in $\Omega$ is $24 \delta$-contracting.
\end{corollary}

\begin{proof}[Proof of \cref{prop:morse_contracting_uniform}]
  Fix a Morse gauge $M$, and suppose for a contradiction that there is
  an infinite sequence of $M$-Morse geodesics $\{\ell_n\}$ in $X$ so
  that $\ell_n$ fails to be projectively $n$-slim. Then for each $n$
  there is a projective triangle
  $[a_n,b_n]\cup[b_n, c_n] \cup [c_n,a_n]$ in $\Omega$ with
  $[b_n,c_n]\subset \ell_n$ which is not $n$-slim.  By
  \cref{rem:proj_delta_slim_defn}, this implies that there is a point
  $u_n \in [b_n,c_n]$ such that
  \[
  \hil(u_n,[a_n,b_n]\cup [a_n,c_n]) \geq n.
  \]
  As $\Omega$ is divisible, there exists a discrete subgroup
  $\Gamma \subseteq \Aut(\Omega)$ and a compact set $D \subset \Omega$
  such that $\Gamma \cdot D=\Omega$. Then, we can find $\gamma_n$ in
  $\Gamma$ such that $\gamma_n u_n \in D$. Up to passing to a
  subsequence, we can assume that the points
  $\gamma_n a_n, \gamma_n b_n,\gamma_n c_n, \gamma_n u_n, \gamma_n
  \ell_n$ converge to $a,b,c,u, \ell$ in $\overline{\Omega}$
  respectively. By construction $u \in D$. Since $u_n \in \ell_n$,
  this implies that $u \in \ell$ and hence $\ell$ is a bi-infinite
  projective geodesic in $\Omega$. Moreover,
  $[a,b]\cup[a,c]\subset \bdry$, because
  \begin{align*}
    \liminf_{n \to \infty}\hil(u,[a_n,b_n]\cup[a_n,c_n])
    &\geq \left(\liminf_{n\to\infty}\hil(u_n,[a_n,b_n]\cup[a_n,c_n]) \right)-
      \lim_{n \to \infty} \hil(u,u_n)\\
    &= \infty.
  \end{align*}
  Then $\ell=(a,b)$ and the points $a,b,c$ lie in the boundary of a
  half triangle in $\Omega$.

  Now, as $\ell_n$ is a sequence of $M$-Morse geodesics converging
  uniformly to a geodesic $\ell$ on compact sets, it follows from
  \cite[Lemma 2.10]{Cordes17} that $\ell$ is $M$-Morse. But then this
  contradicts \Cref{lem:morse_not_in_halftri}.
\end{proof}

\subsection{Morseness, $C^1$ points, and extreme points}

We record a few more consequences of
\Cref{prop:morse_hilbert_equivalences}. These results will be relevant
later in the paper when we consider the behavior of Morse geodesics
as subsets of the automorphism group $\Aut(\Omega) \subset \PGL(V)$. Recall from \cref{rem:morseness_basept_indep} that it makes sense to refer to a point $y$ in $\bdry$ as a ``Morse point'' if some geodesic ray $[x, y)$ is Morse.

\begin{proposition}
  \label{prop:morse_implies_c1_extreme}
  Suppose that $y \in \dee \Omega_1$ is $M$-Morse and $(y,\Omega_1)$
  is forward conically related to $(x,\Omega_2)$. Then $x$ is an
  extreme point and a $C^1$ point in $\dee \Omega_2$.
\end{proposition}
\begin{proof}
  Follows immediately from \Cref{lem:morse_conically_related} and
  \Cref{cor:morse_point_is_C1_extreme}.
\end{proof}

Below, we provide a partial converse to
\Cref{prop:morse_implies_c1_extreme}. Recall that $\domains(V)$
denotes the space of properly convex domains in $\Pb(V)$.

\begin{definition}
  Let $\Omega$ be a convex projective domain in $\Pb(V)$. We let
  $\orbitcl(\Omega)$ denote the closure of the $\PGL(V)$-orbit of
  $\Omega$ in $\domains(V)$.
\end{definition}
Recall the notion of domains with exposed boundary from \cref{defn:exposed_boundary}.
\begin{proposition}
  \label{prop:c1_extreme_implies_morse}
  Suppose $\Omega_1$ is a properly convex domain such that every
  $\Omega \in \orbitcl(\Omega_1)$ has exposed boundary. Let
  $y \in \dee \Omega_1$ be such that: if $(y,\Omega_1)$ is forward
  conically related to $(x, \Omega_2)$, then $x$ is a $C^1$ extreme
  point in $\bdry_2$. Then $y$ is $M$-Morse for some Morse gauge $M$.
\end{proposition}
\begin{proof}
  Fix a point $y \in \dee \Omega_1$ satisfying the assumption
  in the statement above, and suppose that $y$ is not a Morse point in $\dee
  \Omega_1$. We will show that if this holds, there is a domain
  $\Omega \in \orbitcl(\Omega_1)$ which does not have exposed
  boundary.

  The implication \ref{item:segclose_forward} $\implies$
  \ref{item:nondiv_morse} in \Cref{prop:morse_hilbert_equivalences}
  means that $y$ is forward conically related to a point
  $y_2 \in \dee \Omega_2$, lying in the closure of a nontrivial
  segment $s$ in $\dee \Omega_2$. By definition $\Omega_2$ lies in
  $\orbitcl(\Omega_1)$. Since $(y,\Omega_1)$ is forward conically related to $(y_2,\Omega_2)$, 
  our hypothesis on $y$ implies that $y_2$ is a $C^1$ extreme point in $\partial \Omega_2$. 
  So $y_2$ lies in the boundary of the nontrivial segment $s \subset \partial \Omega_2$. As $y_2$ is a $C^1$ point in $\dee \Omega_2$, there is a unique
  supporting hyperplane $H$ of $\Omega_2$ at $y_2$. Now observe that if $H'$ is any hyperplane supporting $\Omega_2$ at a point in the relative interior
  of $s$, then $H'$ must also contain $y_2$, and therefore $H = H'$. Thus, as every point in the relative interior of $s$ is contained in \emph{some} supporting hyperplane, it follows that $H$ contains all of $s$. Therefore $y_2$ cannot be an exposed point and $\Omega_2$ cannot have exposed boundary (see \cref{fig:exposed_face}).
\end{proof}

When $\Omega$ is a \emph{divisible} domain in $\Pb(V)$, the
$\PGL(V)$-orbit of $\Omega$ in $\domains(V)$ is closed, i.e. $\orbitcl(\Omega)=\PGL(V) \cdot \Omega$; see \cref{cor:benzecri_for_divisible}. 
So in this case every domain in $\orbitcl(\Omega)$ has exposed boundary if and only if $\Omega$ has
exposed boundary, and we can combine
\Cref{prop:morse_implies_c1_extreme} and
\cref{prop:c1_extreme_implies_morse} to obtain the following:
\begin{corollary}
  \label{cor:divisible_c1_extreme_morse}
  Let $\Omega$ be a convex divisible domain with exposed boundary, and
  let $y \in \dee \Omega$. Then the following are equivalent:
  \begin{enumerate}
  \item For some (any) $x \in \Omega$, the projective geodesic
    $[x, y)$ is $M$-Morse for some Morse gauge $M$.
  \item If $y$ is forward conically related to $z \in \dee \Omega$,
    then $z$ is a $C^1$ extreme point in $\partial \Omega$.
  \end{enumerate}
\end{corollary}

\section{Estimating singular values using convex projective geometry}
\label{sec:sv_estimates}

In this section, we will estimate singular values using projective
geometry. Specifically, if $\{g_n\}$ is a sequence in $\PGLdR$ that
``almost" preserves a properly convex domain, then we obtain
asymptotic estimates for various singular values of $\{g_n\}$. We will
use these estimates in the next section to study the singular values
of sequences that track Morse geodesic rays.

\subsection{Singular value gap estimates when a domain is preserved}

We first record some known estimates on singular values of
automorphisms of $\Omega$.  The first estimate relates Hilbert
distances to the $\mu_{1,d}$ singular value gap.
\begin{proposition}[{\cite[Proposition 10.1]{DGK2017convex}}]
  \label{prop:dgk_sv_gap_omega}
  Let $\Omega$ be a properly convex domain in $\Pb(\Rb^d)$. For any
  basepoint $x_0 \in \Omega$, there exists a constant $D$ so that for
  any $\gamma \in \Aut(\Omega)$, we have
  \[
    \abs{\mu_{1,d}(\gamma) - \frac{1}{2}\hil(x_0, \gamma x_0)} \le D.
  \]
  Moreover, the constant $D$ can be chosen to vary continuously as
  $(x_0, \Omega)$ varies in the space of pointed properly convex
  domains.
\end{proposition}

To obtain estimates for other singular value gaps, we can consider the
faces in the boundary of a properly convex domain $\Omega$. Let $F$ be
a $k$-dimensional face of $\Omega$, fix a basepoint $x_0 \in \Omega$,
and let $\{\gamma_n\}$ be a sequence in $\Aut(\Omega)$ so that
$\gamma_n x_0$ accumulates on $F$. \Cref{lem:hilbert_face_metric}
tells us that, if $B(x_0, r)$ is any open ball about $x_0$ (with
respect to $\hil$), then $\gamma_n B(x_0, r)$ also accumulates on the
$k$-dimensional face $F$. This can be used to see that the sequence
has a singular value gap at some index $j$ with $j \le k+1$. Precisely,
we have the following.
\begin{proposition} [{See e.g. \cite[Proposition 5.6]{IZ2021}}]
\label{prop:face_dim_implies_some_sv_gp}
    Suppose $\{\gamma_n\}$ is a sequence in $\Aut(\Omega)$, $x_0 \in
    \Omega$, and $\gamma_n x_0 \to x \in \bdry$. If
    $\dim(F_{\Omega}(x))=k$, then $\mu_{1,k+2}(\gamma_n) \to \infty$.
\end{proposition}
\begin{proof}
  The proof of \cite[Proposition 5.6]{IZ2021} immediately implies this
  (although the result is stated differently in that paper). In the
  notation of \cite{IZ2021}, suppose $\gamma_n \to T$ in
  $\Pb(\End(\Rb^d))$. Then $T$ is a projective linear map with
  $\dim(\image(T))=q$ where
  $q:=\max\{ i : \liminf_{n \to \infty} \mu_{1,i}(\gamma_n) <\infty\}$
  and $\image(T) \subset \Span F_{\Omega}(x)$. Thus $q \leq k+1$ where
  $k:=\dim F_{\Omega}(x)$. Hence
  $\mu_{1,k+2}(\gamma_n) \geq \mu_{1,q+1}(\gamma_n) \to \infty$.
\end{proof}    
   
In the above proposition, $\{\gamma_n\}$ does not need to track the
projective geodesic $[x_0,x)$. But if the sequence $\{\gamma_n\}$ does
actually track the projective geodesic ray $[x_0,x)$, then we get a
stronger statement. In this case, it is possible to show that the
balls $\gamma_n B(x_0, r)$ limit onto a relatively \emph{open} subset
of $F_{\Omega}(x)$, which in turn implies that the sequence $\gamma_n$
does \emph{not} have singular value gaps at an index less than
$k$. Using this idea, one proves the following:
\begin{proposition}[{See e.g. \cite[Proposition 5.7]{IZ2021}}]
  \label{prop:face_sv_gap}
  Let $\Omega$ be a properly convex domain, let  
  $c:[0, \infty) \to \Omega$ be a projective geodesic ray, and let
  $\{\gamma_n\}$ track $c$. The
  following are equivalent:
  \begin{enumerate}
  \item The endpoint $c(+\infty) \in \partial \Omega$ lies in a
    $k$-dimensional face in $\partial \Omega$.
  \item There exists some constant $D > 0$ such that
    $\mu_{k+1, k+2}(\gamma_n)$ tends to infinity as $n \to \infty$,
    and for any $1 \leq \ell \leq k$, we have
    $\mu_{\ell, \ell+1}(\gamma_n) < D$.
  \end{enumerate}
\end{proposition}

\subsection{Singular value estimates when a domain is almost
  preserved}

The remaining estimates in this section are somewhat more
technical. This is partly because we no longer restrict our attention
to automorphisms of a fixed convex projective domain $\Omega$. Rather,
we consider projective transformations that ``almost preserve'' a
domain. This idea is closely tied to the notion of conically related
points from the previous section.

It will be useful to introduce the following definitions.
\begin{definition} Suppose $V$ is a real vector space. Recall that
  $\domains(V)$ denotes the space of properly convex domains in
  $\Pb(V)$.

  Let $\ell \subset V$ be a projective line segment with endpoints
  $x_\pm$, and let $H$ be a projective subspace in $\Pb(V)$ with
  codimension 2. We let
    \[
      \domains(V; \ell, H)
  \]
  denote the set of domains $\Omega \subset \Pb(V)$ such that $\ell$
  is properly embedded in $\Omega$, and the projective hyperplanes
  $\Pb(x_+ \oplus \lift{H})$ and $\Pb(x_- \oplus \lift{H})$ are both
  supporting hyperplanes of $\Omega$. This set is equipped with the
  subspace topology from $\domains(V)$.
\end{definition}

The lemma below is one of the main technical estimates in this
section. 
\begin{lemma}
  \label{lem:sigma_12d-1d_estimates}
  Fix a projective line segment $\ell = (x_+, x_-)$ and a
  codimension-two projective subspace $H \subset \Pb(\Rb^d)$. Let
  $\Kc_1, \Kc_2$ be two compact subsets of $\domains(\Rb^d; \ell,
  H)$. There exists a constant $C$ (depending only on $\Kc_1, \Kc_2$)
  so that if $g \in \GLdR$ preserves the decomposition
  $x_+ \oplus \lift{H} \oplus x_-$, with
  $||g|_{x_+}|| > ||g|_{x_-}||$, and
  $g\Kc_1 \cap \Kc_2 \ne \emptyset$, then:
  \begin{enumerate}
  \item\label{item:mu1_estimate} $\abs{\log \norm{g|_{x_+}}-\mu_1(g)} < C$,
  \item\label{item:mud_estimate} $\abs{\log \norm{g|_{x_-}} - \mu_d(g)} < C$, 
  \item\label{item:mu2_estimate} $\abs{\log \norm{g|_{\lift{H}}} -\mu_2(g)}< C$,
  \item\label{item:mud1_estimate}
    $\abs{\log \conorm{g|_{\lift{H}}}-\mu_{d-1}(g)}<C$.
  \end{enumerate}
\end{lemma}
\begin{proof}
  Let $W = \lift{H}$. Because of \Cref{lem:transverse_decomp_compact}
  and \Cref{lem:additive_root_bound}, we may assume that the
  decomposition $x_+ \oplus W \oplus x_-$ is orthogonal. In this
  situation, whenever $g$ satisfies the hypotheses of the lemma, we
  can find indices $1\leq i < j \leq d$ so
  $||g|_{x_+}|| = \sigma_i(g)$ and $||g|_{x_-}|| = \sigma_j(g)$. We first claim that:
  \begin{claim}
      It suffices to prove only part \eqref{item:mu1_estimate}. 
  \end{claim}
  \begin{proof}[Proof of Claim] Suppose we have part \eqref{item:mu1_estimate}. Part
  \eqref{item:mud_estimate} follows immediately by applying part
  \eqref{item:mu1_estimate} to $g^{-1}$ (and interchanging the roles
  of $\Kc_1, \Kc_2$). So we only need to see that parts
  \eqref{item:mu1_estimate} and \eqref{item:mud_estimate} together
  imply parts \eqref{item:mu2_estimate} and
  \eqref{item:mud1_estimate}. Parts \eqref{item:mu1_estimate} and
  \eqref{item:mud_estimate} imply that $0 \leq \mu_1(g)-\mu_i(g) < C$
  and $0 \leq \mu_j(g) -\mu_d(g) < C$, giving us
  $\abs{\mu_{1,d}(g)-\mu_{i,j}(g)} \leq C'$ where $C':=2C$. Then
  \cref{lem:singular_value_ratios} implies that
  \begin{align}
	\label{eqn:compare_singular_values_with_sigma_i_and_j}
    \max \left\{ \max_{1 \leq k \leq i} \mu_{1,k}(g), \max_{j \leq k \leq d} \mu_{k,d}(g) \right\} \leq C'.
  \end{align}

  Let $i'$ and $j'$ be the minimum and the maximum, respectively, of
  the set $(\{1,\dots,d\}-\{i,j\})$. Since $x_+ \oplus W \oplus x_-$
  is an orthogonal decomposition,
  \[ \norm{g|_W}=\sigma_1(g|_W)=\sigma_{i'}(g) \text{ and }
    \conorm{g|_W}=\sigma_{d-2}(g|_W)=\sigma_{j'}(g).
  \] 
  We consider several cases depending on the value of $i'$. If
  $i'=2$, then part \eqref{item:mu2_estimate} is immediate. On the
  other hand, if $i'=1$, then the definition of $i'$ implies that
  $i\geq 2$. Then
  \eqref{eqn:compare_singular_values_with_sigma_i_and_j} implies
  that $$\abs{\mu_{i'}(g)-\mu_2(g)} = \mu_{1,2}(g) \leq C'$$ which
  again implies part \eqref{item:mu2_estimate}. So we are left with
  the case that $i'>2$.  Note that this occurs precisely when $i=1$
  and $j=2$. But in that case,
  \eqref{eqn:compare_singular_values_with_sigma_i_and_j} implies that
  $$\abs{\mu_{i'}(g)-\mu_{2}(g)} =\mu_{j,i'}(g) \leq C'$$ which
  again implies part \eqref{item:mu2_estimate}. Thus we have shown
  that part \eqref{item:mu1_estimate} implies part
  \eqref{item:mu2_estimate}.
  
  Finally, since $\conorm{g|_W}=\norm{(g|_W)^{-1}}$, we can apply part
  \eqref{item:mu2_estimate} to $(g|_W)^{-1}$ to prove part
  \eqref{item:mud1_estimate}. This finishes the proof of the claim
  that it suffices to prove only part \eqref{item:mu1_estimate}.
  \end{proof}

  We now proceed with the proof of \textbf{part
    \eqref{item:mu1_estimate}}.  Suppose, on the contrary, that part
  \eqref{item:mu1_estimate} fails. Then there is a sequence $\{g_n\}$ in $\GLdR$ satisfying the hypotheses of the lemma, but with
  \begin{align}
    \label{eqn:defn_of_gn}\frac{\sigma_1(g_n)}{\sigma_i(g_n)} \to \infty
  \end{align}
  as $n \to \infty$. Here, $i$ is an index such that 
  $||g_n|_{x_+}|| = \sigma_i(g_n)$ for every $n$ (after passing to
  a subsequence, we can ensure that the same fixed index $i$ works for
  each $g_n$). Up to passing to a further subsequence, we may also
  assume that there exists $j \in \{ i + 1,\dots, d\}$ such that
  $\norm{g_n|_{x_-}}=\sigma_j(g_n)$. Note that in particular,
  \eqref{eqn:defn_of_gn} implies that $i > 1$.

  We can fix an orthonormal basis for $W$, and extend it to an
  orthonormal basis for $\Rb^d$ by adding unit vectors spanning
  $x_+, x_-$. With respect to this basis, $g_n$ is block-diagonal, of
  the form
  \[
    \begin{pmatrix} \sigma_i(g_n) & & \\ & g_n|_{W}& \\ & &
      \sigma_j(g_n)\end{pmatrix}.
  \]
  The restriction $g_n|_W$ has a Cartan decomposition $k_na_nl_n$,
  where $a_n$ is a diagonal matrix with respect to our chosen basis on
  $W$, and $k_n, l_n$ lie in the group $\mathrm{O}(W)$ of orthogonal
  transformations of $W$.

  Observe that, if we pre-compose or post-compose $g_n$ with any
  orthogonal matrix of $\Rb^d$ fixing $\ell=(x_+,x_-)$ pointwise and
  preserving $W$, the values of $\mu_i(g_n)$, $\norm{g_n|_{x_\pm}}$,
  $\norm{g_n|_W}$, and $\conorm{g_n|_W}$ do not change. So, after
  replacing $g_n$ with the sequence
  \[
    \begin{pmatrix} 1 & & \\ & k_n^{-1}& \\ & &
      1\end{pmatrix}
    \begin{pmatrix} \sigma_i(g_n) & & \\ & g_n|_{W}& \\ & &
      \sigma_j(g_n)\end{pmatrix}
    \begin{pmatrix} 1 & & \\ & l_n^{-1}& \\ & &
      1\end{pmatrix},
  \]
  and replacing the sets $\Kc_1, \Kc_2$ with the sets
  \[
    \begin{pmatrix} 1 & & \\ & \mathrm{O}(W)& \\ & & 1\end{pmatrix}\Kc_1,
    \qquad
    \begin{pmatrix} 1 & & \\ & \mathrm{O}(W)& \\ & &
      1\end{pmatrix}\Kc_2,
  \]
  we can assume that each $g_n$ is a diagonal matrix with respect to a
  fixed orthonormal basis $e_1, \ldots, e_d$, compatible with the
  orthogonal decomposition $\Rb^d = x_+ \oplus W \oplus x_-$. In
  particular, we order our basis so that $x_+ = [e_i]$ and
  $x_- = [e_j]$ and $g_ne_k = \sigma_k(g_n)e_k$ for each
  $1 \le k \le d$.

  Now fix a point $v \in \Rb^d$ so that $[v] \in \ell$. We may write
  $v = ae_i + be_j$ for $a, b$ both nonzero. Fix any $t \neq 0$. Then,
  using \eqref{eqn:defn_of_gn} (and the fact that $i < j$), we have
  \begin{align*}
    \label{eq:gn_attracting_converge}
    \frac{1}{\sigma_1(g_n)}g_n(te_1 + v) = te_1 +
    a\frac{\sigma_i(g_n)}{\sigma_1(g_n)}e_i +
    b\frac{\sigma_j(g_n)}{\sigma_1(g_n)}e_j \to te_1.
  \end{align*}
  Thus $g_n[v + te_1] \to [e_1]$ for any $t \ne 0$.

  Now, choose domains $\Omega_n \in \Kc_1$ so that
  $g_n\Omega_n \in \Kc_2$. By compactness of $\Kc_2$, we can pass to a
  subsequence and assume that $g_n\Omega_n$ converges to a domain
  $\Omega_\infty$. Since $\Kc_1$ is a compact subset of
  $\domains(\Rb^d; \ell, H)$, there is some $\eps > 0$ so that for
  each $\Omega \in \Kc_1$, the Hilbert distance in $\Omega$ between
  $[v + te_1]$ and $[v]$ is uniformly bounded for
  $t \in (-\eps, \eps)$. Our assumption \eqref{eqn:defn_of_gn} means
  that $\{g_n\}$ is divergent when viewed as a sequence of projective
  transformations, so \Cref{thm:benzecri} implies that $[g_nv]$ only
  accumulates on $\partial \Omega_\infty$. Since $[v]$ lies in the
  $g_n$-invariant subspace $\ell$ and
  $\norm{g_n|_{x_+}} > \norm{g_n|_{x_-}}$, the only possibility is
  that $[g_nv]$ converges to $x_+$.

  Since $\lim_{n\to \infty}g_n[v + te_1] =[e_1]$ for any $t\neq 0$, it follows from
  \Cref{lem:hilbert_face_metric} that
  $[e_1] \in F_{\Omega_{\infty}}(x_+).$
  Moreover, by the same lemma, for any $t \in (-\eps, \eps)-\{0\}$, we have
  \begin{align*}
    d_{F_{\Omega_{\infty}}(x_+)}(x_+, [e_1])
    &\leq \liminf_{n \to \infty} d_{g_n\Omega_n}(g_n[v], g_n[v +te_1])\\
    &= d_{\Omega_n}([v],[v + te_1]).
  \end{align*}
  As $\Omega_n$ lies in a compact set $\Kc_1$, the Hilbert distances 
  $d_{\Omega_n}([v],[v + te_1])$ tend to $0$ uniformly in $n$ as
  $t \to 0$. This means that in fact
  $d_{F_{\Omega_{\infty}}(x_+)}(x_+, [e_1]) = 0$, i.e. $x_+ =
  [e_1]$. But this is a contradiction since we have also arranged
  $x_+ = [e_i]$ for $i \ne 1$, and $\{e_1, \ldots, e_d\}$ is a basis
  for $\Rb^d$.
\end{proof}

\subsection{Application to automorphisms of properly convex domains}

We now apply the previous lemma to establish estimates on singular
values of projective transformations which \emph{actually} (instead of
``approximately'') preserve a convex domain. First we introduce some more notation.
\begin{definition}
  Let $\Omega$ be a properly convex domain in $\Pb(\Rb^d)$.
  \begin{itemize}
  \item We let $\Gc(\Omega)$ denote the space of all projective
    bi-infinite geodesics in $\Omega$, with unit-speed (in $\hil$) 
    parameterization. Let $c(\pm\infty) \in \bdry$ denote the ideal endpoints of any $c \in \Gc(\Omega)$.
  \item We let $\gtriple(\Omega)$ denote the set of triples
    $(c, H_+, H_-)$, such that $c \in \Gc(\Omega)$ and $H_{\pm}$ are
    supporting hyperplanes of $\Omega$ at $c(\pm \infty)$.
  \item For any compact subset $K \subset \Omega$, we let
    $\Gc_K(\Omega)$ denote the set of geodesics $c \in \Gc(\Omega)$
    such that $c(0) \in K$. Similarly, we use $\gtriple[K](\Omega)$ to
    denote the set
    \[
      \gtriple[K](\Omega) := \{(c, H_+, H_-) \in \gtriple(\Omega) : c
      \in \Gc_K(\Omega)\}.
    \]
  \end{itemize}
\end{definition}

If $\Omega$ does not have $C^1$ boundary, then the projection map
$\gtriple(\Omega) \to \Gc(\Omega)$ is not a homeomorphism. However,
this map is always proper, due to the compactness of the set of
supporting hyperplanes at any point in $\dee \Omega$. The map
$\Gc(\Omega) \to \Omega$ given by $c \mapsto c(0)$ is also proper, as the
space of projective geodesics passing through a given basepoint in
$\Omega$ is also compact.

If we fix an element $(c, H_+, H_-) \in \gtriple(\Omega)$, we know
that $H_+$ cannot contain $c(-\infty)$, since otherwise $H_+$ would
also contain $c(0)$ and would not be a supporting hyperplane of
$\Omega$. Similarly $H_-$ cannot contain $c(+\infty)$. So, we have a
direct sum decomposition
\[
  \Rb^d = c(+\infty) \oplus (\lift{H_+} \cap \lift{H_-}) \oplus c(-\infty).
\]

For triples lying in some $\gtriple[K](\Omega)$, this decomposition is
actually uniformly transverse in the following sense:
\begin{lemma}
  \label{lem:triples_uniformly_transverse}
  For any compact set $K \subset \Omega$, there exists some
  $\eps_0 > 0$ such that for any
  $(c, H_+, H_-) \in \gtriple[K](\Omega)$, we have
  \[
    \min\left\{ \angle(c(+ \infty), H_+ \cap H_-), \angle(c(- \infty),
      H_+ \cap H_-), \angle(c(+\infty),c(-\infty)) \right\} \ge
    \eps_0.
  \]
\end{lemma}
\begin{proof}
  The map $\gtriple(\Omega) \to \Rb$ given by
  \[
    (c, H_+, H_-) \mapsto \min\left\{ \angle(c(+ \infty), H_+ \cap
      H_-), \angle(c(- \infty), H_+ \cap H_-),
      \angle(c(+\infty),c(-\infty)) \right\}
  \]
  is continuous and positive on $\gtriple(\Omega)$. The set 
  $\gtriple[K](\Omega)$ is compact since it is precisely the preimage
  of $K$ under the proper map $\gtriple(\Omega) \to \Omega$. So the
  result is immediate.
\end{proof}

Using this observation, we can apply \Cref{lem:sigma_12d-1d_estimates}
to obtain the following estimate on singular values for automorphisms
of a convex projective domain. In this lemma, and throughout the
paper, if $g$ is an element of $\GLdR$, and $W \subseteq \Rb^d$ is a
subspace (not necessarily $g$-invariant), then the restriction $g|_W$
is interpreted as a map $W \to \Rb^d$; since both $W$ and $\Rb^d$ are
normed spaces, both $\norm{g|_W}$ and $\conorm{g|_W}$ make sense.

\begin{proposition}\label{prop:automorphism_sv_estimates}
Let $\Omega$ be a properly convex domain and $K \subset \Omega$ be
  compact. Then there exists $D > 0$ (depending only on $K, \Omega$)
  satisfying the following: if
  $(c, H_+, H_-) \in \gtriple[K](\Omega)$, and
  $\gamma \in \Aut(\Omega)$ satisfies
  $\gamma^{-1}c(t) \cap K \neq \varnothing$ for some $t > 0$, then:
  \begin{enumerate}
  \item\label{item:first_automorphism_sv_estimate}
    $\abs{\log\norm{\gamma^{-1}|_{c(-\infty)}} -
      \mu_1(\gamma^{-1})} < D$,
  \item
    $\abs{\log \norm{\gamma^{-1}|_{c(+\infty)}} -
      \mu_d(\gamma^{-1})} < D$,
  \item
    $\abs{\log \norm{\gamma^{-1}|_{\lift{H_0}}}
      -\mu_{2}(\gamma^{-1})}< D$,
  \item
    $\abs{\log
      \conorm{\gamma^{-1}|_{\lift{H_0}}}-\mu_{d-1}(\gamma^{-1})}<D$, where $H_0 = H_+ \cap H_-$.
  \end{enumerate}
\end{proposition}
\begin{remark}
  \label{rem:lift_notation_abuse}
  We have slightly abused notation in the statement of this
  proposition, since elements in $\Aut(\Omega)$ are \emph{projective}
  transformations and so the quantities $\mu_i(\gamma)$, etc. are not
  well-defined. So, strictly speaking, the inequalities above apply to
  lifts $\lift{\gamma} \in \GLdR$ of $\gamma$, but the validity of the
  inequalities is independent of the choice of lift.
\end{remark}
\begin{proof}
  This proof is mainly an application of
  \cref{lem:sigma_12d-1d_estimates}. We first fix, once and for all, a
  decomposition $\Rb^d=x_+ \oplus \tilde{H} \oplus x_-$ where
  $x_\pm \in \Pb(\Rb^d)$ and $H$ is a codimension-2 projective
  subspace. Let $\ell$ be a projective line segment in $\Pb(\Rb^d)$
  joining $x_+$ and $x_-$, i.e. $\ell$ is one of the two connected
  components of $\Pspan{x_+, x_-} \minus \{x_+, x_-\}$.
  
  Now we need to modify
  $\gamma^{-1}$ so that it preserves the decomposition $\Rb^d=x_+ \oplus \tilde{H} \oplus x_-$. 
    Applying \Cref{lem:transverse_decomp_compact} and
  \Cref{lem:triples_uniformly_transverse} above, we see that there
  exists a compact set $Q \subset \GLdR$ (depending only on $K$) so
  that for any $(c, H_+, H_-) \in \gtriple[K](\Omega)$, with
  $H_0 = H_+ \cap H_-$, we can find some $k = k(c, H_+, H_-) \in Q$
  taking the decomposition
  $\Rb^d=c(+\infty) \oplus \lift{H_0} \oplus c(-\infty)$ to
  $x_- \oplus \lift{H} \oplus x_+$. Moreover, we can also assume that this $k$ takes the image of $c$ to the projective
  line segment $\ell$. Indeed, $\ell$ is one of the two connected components of $\Pspan{x_+, x_-} \minus \{x_+, x_-\}$. Thus, if necessary, we can compose all of the elements in $Q$ with a fixed involution interchanging the
  connected components of $\Pspan{x_+, x_-} \minus \{x_+, x_-\}$ and ensure that $k$ takes $c(\Rb)$ to $\ell$.

  Possibly after replacing $Q$ with the closure of the set
  \[
    Q' := \{k(c, H_+, H_-) : (c, H_+, H_-) \in
      \gtriple[K](\Omega)\},
  \]
  we may assume that for every $q \in Q$, the projective segment
  $q^{-1}\ell$ is properly embedded in $\Omega$, and $q^{-1}H$ is
  disjoint from $\Omega$. This means that the set $qK \cap \ell$ has
  bounded diameter with respect to the Hilbert metric $d_{\ell}$ on $\ell$, and
  that the set
  \[
    \Kc := \{q\Omega : q \in Q\}
  \]
  is a compact subset of $\domains(\Rb^d; \ell, H)$. Further, since
  $Q$ is compact, the diameter (with respect to the Hilbert metric $d_{\ell}$) of the set $\left( \bigcup_{q \in Q} (\ell \cap qK) \right)$
  is also bounded. Let $L$ be an upper bound for the diameter of this
  set.

  Now fix some $(c, H_+, H_-) \in \gtriple[K](\Omega)$, and assume
  that $\gamma^{-1}c(t) \in K$ for $\gamma \in \Aut(\Omega)$ and
  $t > 0$. As $\Aut(\Omega)$ acts properly on $\Omega$, if $t \le L$
  then $\gamma^{-1}$ belongs to a fixed compact subset of
  $\Aut(\Omega)$ depending only on $L$, and we may choose $D$
  sufficiently large so that each of the inequalities in the statement
  of the proposition holds for every $\gamma$ in this set. So, we may
  assume from now on that $t > L$. Fix a lift of $\gamma$ in $\GLdR$;
  abusing notation we also denote this lift by $\gamma$ (see
  \Cref{rem:lift_notation_abuse}).

  We let $c'$ be the translated and reparameterized geodesic
  $s \mapsto \gamma^{-1}c(s + t)$, so that $c' \in \Gc_K(\Omega)$, and
  \[
    (c', \gamma^{-1}H_+, \gamma^{-1}H_-) \in \gtriple[K](\Omega).
  \]
  By our construction of $Q$, we can choose $k, k' \in Q$ so that $k$
  takes the decomposition
  $c(\infty) \oplus \lift{H_0} \oplus c(-\infty)$ to
  $x_- \oplus \lift{H} \oplus x_+$, and similarly for $k'$, $c'$ and
  $\gamma^{-1}H_\pm$. Then, the group element $g \in \GLdR$ defined by
  $g = k'\gamma^{-1}k^{-1}$ preserves the decomposition
  $x_+ \oplus \lift{H} \oplus x_-$ and the projective line $\ell$,
  which verifies one of the hypotheses of
  \Cref{lem:sigma_12d-1d_estimates}. Moreover, recalling that
  $\Kc = \{k\Omega : k \in Q\}$ is a compact subset of
  $\domains(\Rb^d; \ell, H)$, we see that $g\Kc$ contains
  $k'\gamma^{-1}k^{-1}k\Omega) = k'\Omega \in \Kc$, hence
  $g\Kc \cap \Kc \ne \emptyset$. This verifies another hypothesis of
  \Cref{lem:sigma_12d-1d_estimates}, when we take
  $\Kc_1 = \Kc_2 = \Kc$.

  Finally, we need to verify that $||g_{x_+}|| \ge ||g_{x_-}||$, by
  considering the action of $g$ on the projective line segment
  $\ell$. Let $x_0 = k c(0)$. Observe that the 4-tuple
  $[c(-\infty), c(0), c(t), c(+\infty)]$ is arranged on the image of
  $c$ in this order. Applying $k'\gamma^{-1}$ to $c$, we then observe that the points
  \[
    [k'c(-\infty), k'\gamma^{-1}c(0), k'\gamma^{-1}c(t), k'c(+\infty)] =
    [x_+, gx_0, k'\gamma^{-1}c(t), x_-]
  \]
  lie on the projective segment $\ell$ in this order. Moreover, since
  $\gamma^{-1}c(t) \in K$, and $c(0) \in K$, we know from the
  definition of $L$ that  $d_{\ell}(k'\gamma^{-1}c(t), x_0) \leq L$ (with respect to the Hilbert metric $d_{\ell}$ on
  $\ell$). Since
  $d_\ell(k'\gamma^{-1}c(0), k'\gamma^{-1}c(t)) = t > L$, it follows
  that the 4-tuple of points
  \[
    [x_+, k'\gamma^{-1}c(0), kc(0), x_-] = [x_+, gx_0, x_0,
    x_-]
  \]
  are also arranged in this order on $\ell$. Since $g$ fixes the
  endpoints of $\ell$, the eigenvalue of $g$ on $x_+$ must be larger
  than the eigenvalue of $g$ on $x_-$, or equivalently,
  $\norm{g|_{x_+}} > \norm{g|_{x_-}}$.

  We have now verified that we can apply
  \Cref{lem:sigma_12d-1d_estimates} to $g$. Then 
  $\log(\norm{g|_{x_+}})$ is within bounded additive error $C$ of
  $\mu_1(g)$, where the constant $C$ depends only on $K$ and $\Omega$. However, since
  $gk = k'\gamma^{-1}$ for $k, k'$ in the fixed compact set $Q$, and
  $kc(-\infty) = x_+$, we can apply \Cref{lem:additive_root_bound} to
  get the first desired estimate for $\gamma^{-1}$. The other
  estimates follow similarly.
\end{proof}

\subsection{A ``straightness" lemma}

The estimate given by \Cref{prop:dgk_sv_gap_omega} implies that, if
$\gamma_n$ is a sequence tracking a projective geodesic in a convex
projective domain $\Omega$, then the gap $\mu_{1,d}(\gamma_n)$
increases roughly linearly in $n$. The same linear estimate need not
hold for other singular value gaps. In fact, \cite{BPS}  proves that uniform linear growth in $n$ imposes a strong restriction. Suppose $\Gamma$ divides $\Omega$ and there is an index $j$ such that $\mu_{j, j+1}(\gamma_n)$ grows
uniformly linearly in $n$ for all tracking sequences $\{\gamma_n\}$. Then $\Gamma$ must be a hyperbolic
group \cite{BPS}. Thus, in the non-hyperbolic setting, there is no way to obtain
such a sharp estimate. However, for a sequence $\{\gamma_n\}$ tracking a Morse geodesic, we can prove a ``coarse monotonicity''
property for $\mu_{1,2}(\gamma_n)$ and
$\mu_{d-1, d}(\gamma_n)$. Our main tool is the following ``straightness" lemma.

\begin{lemma}
  \label{lem:geodesic_additivity}
  Suppose $\Omega$ is a properly convex domain and  $K\subset \Omega$ is a compact
 set. Then there exists a constant $D > 0$ satisfying the
  following: if $c \in \Gc_{K}(\Omega)$ and $\seq{\gamma_n}$  is a sequence in $\Aut(\Omega)$ 
  such that $\gamma_n^{-1}c(n) \in K$ for all $n \in \Nb$, then for any
  $n, m \in \Nb$, we have 
  \begin{align*}
    &\mu_{i,i+1}\left( \gamma_n \right) + \mu_{i,i+1}\left( \gamma_n^{-1}\gamma_{n + m} \right) \le
    \mu_{i,i+1} \left( \gamma_{n+m} \right) + D,
  \end{align*}
  where $i \in \{ 1,d-1\}$.
\end{lemma}
\begin{remark}
  Results of a similar flavor were also obtained by
  Canary-Zhang-Zimmer \cite{CZZ2022} in their work on transverse
  subgroups; see Section 6 in \cite{CZZ2022}, especially Lemma 6.4. A
  crucial difference in our context is that we impose no assumption on
  the regularity properties of the sequence $\{\gamma_n\}$ in
  $\Aut(\Omega)$. In particular, the sequence $\{\gamma_n\}$ does
  \emph{not} need to lie in a uniformly $1$-regular subgroup of
  $\Aut(\Omega)$, which is a condition required by the results in
  \cite{CZZ2022}.
\end{remark}

\begin{proof}
  Throughout the proof, we implicitly identify each $\gamma_n$ in the
  sequence with a chosen lift in $\GLdR$. As in the proof of
  \Cref{prop:automorphism_sv_estimates}, we start by fixing a direct
  sum decomposition $\Rb^d = x_+ \oplus W \oplus x_-$, a projective
  line $\ell$ joining $x_\pm$, and a compact subset $Q \subset \GLdR$
  so that for any $(c, H_+, H_-) \in \gtriple[K](\Omega)$, we can find
  some $k \in Q$ taking
  $c(+\infty) \oplus (\lift{H_+} \cap \lift{H_-}) \oplus c(-\infty)$
  to $x_- \oplus W \oplus x_+$ and the image of $c$ to $\ell$. We also
  fix a constant $L > 0$ as in the proof of the same proposition, so
  that the diameter (in the Hilbert metric $d_\ell$ on $\ell$) of the set 
  \[
  \bigcup_{q \in Q}(\ell \cap qK)
  \]
  is bounded by $L$.

  Next, observe that, if $D > 0$ is chosen large enough (depending on
  $L$), then the desired inequality holds whenever $n < L$. This
  follows from \Cref{lem:additive_root_bound} and the fact that
  $\Aut(\Omega)$ acts properly on $\Omega$: if $n \le L$, then since
  $\hil(c(n), K) \le n$ and $\gamma_n^{-1}c(n) \in K$, the
  automorphism $\gamma_n$ lies in compact subset of $\Aut(\Omega)$
  depending only on $L$, and both quantities $\mu_{i, i+1}(\gamma_n)$
  and
  $|\mu_{i, i+1}(\gamma_n^{-1}\gamma_{n + m}) - \mu_{i,
    i+1}(\gamma_{n+m})|$ are uniformly bounded by
  \Cref{lem:additive_root_bound}.

  Similarly, since $\hil(\gamma_n^{-1}c(n + m), K) \le m$ and
  $\gamma_{n + m}^{-1}\gamma_n \cdot \gamma_n^{-1}c(n + m) \in K$ for
  any $m$, we may also choose $D$ so that the desired inequality holds
  whenever $m \le L$. So, for the rest of the proof, we may assume
  that both $n > L$ and $m > L$.
  
  For each $j \in \Nb$, since $\gamma_j^{-1}c$ passes through $K$, we
  can choose $k_j \in Q$ taking the decomposition
  \[
    \gamma_j^{-1}c(+\infty) \oplus \gamma_j^{-1}(\lift{H}_+ \cap
    \lift{H}_-) \oplus \gamma_j^{-1}c(-\infty)
  \]
  to $x_- \oplus W \oplus x_+$. Here, we assume that
  $\gamma_0 = \identity$. Then, defining 
  $g_j := k_j\gamma_j^{-1}k_0^{-1}$, we observe that $g_j$ preserves the
  decomposition $x_- \oplus W \oplus x_+$. Moreover, by
  \Cref{prop:automorphism_sv_estimates} and
  \Cref{lem:additive_root_bound}, there is a uniform constant $C$ so
  that for given $n \ge 1$, we have
  \begin{align}
    \label{eqn:mu_d-1d_bound_from_general_lemma}
    \abs{\mu_{d-1,d}(\gamma_n^{-1}) - \left( \log \conorm{g_n|_W}
    -\log \norm{g_n|_{c(+\infty)}} \right) } \leq 2C, \text{ and }\\
    \label{eqn:mu_12_bound_from_general_lemma}
    \abs{\mu_{1,2}(\gamma_n^{-1}) - \left( \norm{g_n|_{c(-\infty)}} -
    \log \norm{g_n|_W} \right) } \leq 2C.
  \end{align}
  Next, for given $n, m \in \Nb$, we consider the group
  element
  \[
    T_{n,m} := g_{n + m}g_n^{-1}.
  \]
  By \eqref{eqn:mu_d-1d_bound_from_general_lemma} we know
  \begin{align*}
    \mu_{1,2}(\gamma_n) - \mu_{1,2}(\gamma_{n+m})
    &=\mu_{d-1,d}(\gamma_n^{-1}) -\mu_{d-1,d}(\gamma_{n+m}^{-1}) \\ 
    & \leq \log \dfrac{\conorm{g_n|_W}}{\conorm{g_{n+m}|_W}} + \log
      \dfrac{\norm{g_{n+m}|_{x_-}}}{\norm{g_n|_{x_-}}}+4C.
  \end{align*}
  Since $g_{n+m}=T_{n,m}g_n$, the inequality
  $\conorm{gh} \ge \conorm{g}\conorm{h}$ implies that the first term
  above is at most $\log \frac{1}{\conorm{T_{n,m}|_W}}$. And, since
  $x_-$ is a one-dimensional eigenspace of both $g_n$ and $g_{n + m}$, the
  second term is equal to $\log \norm{T_{n,m}|_{x_-}}$. Thus
  \begin{align}
   \label{eqn:sub_additivity_claim_reduction_to_Tnm}
    \mu_{1,2}(\gamma_n) - \mu_{1,2}(\gamma_{n+m})  \leq \log
    \dfrac{\norm{T_{n,m}|_{x_-}}}{\conorm{T_{n,m}|_W}}  +4C. 
  \end{align}
  We wish to apply \Cref{lem:sigma_12d-1d_estimates} to the element
  $T_{n,m}$, so we let $\Kc = \{k\Omega : k \in Q\}$. Then, since
  \begin{align}
    \label{eq:Tnm_gamma}
    T_{n,m} = g_{n + m}g_n^{-1} = k_{n +
    m}\gamma_{n+m}^{-1}\gamma_nk_n^{-1},
  \end{align}
  we have
  \[
    T_{n,m} k_n\Omega = k_{n+m}\Omega \in \Kc
  \]
  and therefore $T_{n,m} \Kc \cap \Kc \ne \varnothing$. We also need
  to verify the other hypothesis of \Cref{lem:sigma_12d-1d_estimates},
  and show that $\norm{T_{n,m}|_{x_+}} \ge \norm{T_{n,m}|_{x_-}}$. For
  this, we again argue as in the proof of
  \Cref{prop:automorphism_sv_estimates}, and consider the 4-tuple of
  points
  \[
    [c(-\infty), c(n), c(n + m), c(+\infty)]
  \]
  arranged in this order on the image of $c$. Then the 4-tuple
  \begin{align}
  \label{eqn:4_tuple_on_ell_2}
    k_n\gamma_n^{-1}\cdot [c(-\infty), c(n), c(n + m), c(+\infty)]
  \end{align}
  is arranged in the corresponding order on the projective segment
  $\ell$. We let $y_0 := k_n\gamma_n^{-1}c(n)$ and
  $y_m := k_n\gamma_n^{-1}c(n + m)$. Then the 4-tuple in \eqref{eqn:4_tuple_on_ell_2} is the same
  as
  \[
    [x_+, y_0, y_m, x_-].
  \]
  Since $T_{n,m}$ fixes the endpoints $x_\pm$ of $\ell$ and preserves
  $\ell$, the points
  \[
    [x_+, T_{n,m}y_0, T_{n,m}y_m, x_-]
  \]
  are also arranged in this order on $\ell$. Further, since
  $y_0 \in QK \cap \ell$, and
  $T_{n,m}y_m = k_{n + m}\gamma_{n+m}^{-1}c(n+m) \in QK \cap \ell$, we
  have $\mathrm{d}_\ell(T_{n,m}y_m, y_0) \le L$. But 
  $$\mathrm{d}_\ell(y_0, y_m) = d_{c}(c(n),c(n+m))=m > L.$$ So,
  $T_{n,m}y_m$ must lie in the open projective segment
  $(x_+,y_m) \subset \ell$. Thus it follows that the points
  \[
    [x_+, T_{n,m}y_m, y_m, x_-]
  \]
  are arranged on $\ell$ in that order which implies that 
  $\norm{T_{n,m}|_{x_+}} > \norm{T_{n,m}|_{x_-}}$.

  We may therefore apply estimate \eqref{item:mud_estimate} and estimate
  \eqref{item:mud1_estimate} from \Cref{lem:sigma_12d-1d_estimates} to
  $T_{n,m}$. This tells us that there is a uniform constant $C' > 0$ so
  that
  \[
    \log \dfrac{\norm{T_{n,m}|_{x_-}}}{\conorm{T_{n,m}|_W}} \leq
    \mu_d(T_{n,m})-\mu_{d-1}(T_{n,m}) + 2C' =
    -\mu_{1,2}(T_{n,m}^{-1})+2C'.
  \]
  Putting this together with
  \eqref{eqn:sub_additivity_claim_reduction_to_Tnm} and
  \eqref{eq:Tnm_gamma}, we obtain
  \[
    \mu_{1,2}(\gamma_n) - \mu_{1,2}(\gamma_{n+m}) \le
    -\mu_{1,2}(k_n\gamma_n^{-1}\gamma_{n+m}k_{n+m}^{-1}) + 4C + 2C'.
  \]
  Then an application of \Cref{lem:additive_root_bound} proves
  that the desired inequality holds when $i = 1$. 
  
  The case where
  $i = d-1$ is similar; we apply
  \eqref{eqn:mu_12_bound_from_general_lemma} in place of
  \eqref{eqn:mu_d-1d_bound_from_general_lemma} to see that
  \[
    \mu_{d-1,d}(\gamma_n)-\mu_{d-1,d}(\gamma_{n+m}) \leq \log
    \frac{\norm{T_{n,m}|_{x_+}}}{\norm{T_{n,m}|_W}} + 4C.
  \]
  Then we use the other estimates from
  \Cref{lem:sigma_12d-1d_estimates} to see that
  \[
    \log \dfrac{\norm{T_{n,m}|_{x_+}}}{\norm{T_{n,m}|_W}} \leq
    -\mu_{d-1,d}(T_{n,m}^{-1})+2C',
  \]
  and apply \Cref{lem:additive_root_bound} again to complete the
  proof.
\end{proof}

\section{Singular values of Morse geodesics in convex projective geometry}
\label{sec:sv_morse}

In this section, we combine results from
\Cref{sec:morse_contracting,sec:sv_estimates} to study the behavior of
singular value gaps of sequences that track Morse projective geodesic
rays. The main aim of the section is to prove
\Cref{thm:morse_implies_bdry_reg_and_unif_reg} and
\Cref{thm:1_morse_ggt_morse}.

\subsection{Morse geodesics are strongly uniformly regular}

We first address \Cref{thm:morse_implies_bdry_reg_and_unif_reg}. We
start with the following lemma, which we will strengthen later.
\begin{lemma}
  \label{lem:morse_geodesic_uniform_gap}
  Let $M$ be a Morse gauge, let $C> 0$, and let $x_0 \in
  \Omega$. There exists $k=k(M,C,x_0) > 0$ such that, for any
  $\gamma \in \Aut(\Omega)$, if $\hil(x_0, \gamma x_0) > k$ and
  $[x_0,\gamma x_0]$ is $M$-Morse, then $\mu_{1,2}(\gamma) > C$.
\end{lemma}
\begin{proof}
  Fix $M, C, x_0$. Suppose for a contradiction that there exists a
  sequence $\{\gamma_n\}$ in $\Aut(\Omega)$ such that
  $\hil(x_0,\gamma_n x_0) > n$ and $[x_0,\gamma_n x_0]$ $M$-Morse, but
  $\mu_{1,2}(\gamma_n) \leq C$. After passing to a subsequence, we can
  assume that $\gamma_n x_0$ converges to $x \in \bdry$. Then
  $[x_0,\gamma_n x_0] \to [x_0,x)$ uniformly on compact subsets of
  $\Omega$. As each $[x_0,\gamma_n x_0]$ is $M$-Morse, so is
  $[x_0,x)$.

  By \cref{cor:morse_point_is_C1_extreme}, $x$ is a $C^1$ extreme
  point in $\bdry$, so $\dim F_{\Omega}(x)=0$. Then
  \cref{prop:face_dim_implies_some_sv_gp} implies that
  $\mu_{1,2}(\gamma_n) \to \infty$. This contradicts the assumption
  that $\mu_{1,2}(\gamma_n) \leq C$.
\end{proof}

For the next result, we slightly refine the notion of tracking
sequences from \cref{defn:tracking}. If $c:[0,L] \to \Omega$ is a
projective geodesic segment of length $L>0$, then we will say that a
finite sequence $\{ \gamma_n\}$ in $\Aut(\Omega)$ \emph{$R$-tracks
  $c$} if $\hil(\gamma_n x_0,c(n))<R$ for all $n \in \Nb \cap [0,L]$.

\begin{remark}
  \label{rem:tracking_implies_mu_1d_gap}
  If $\{\gamma_n\}$ $R$-tracks $c$ (a geodesic ray or segment), then there exists a constant
  $D'$ depending on $x_0$ and $R$ such that
  \[
    \frac{k}{2}-D' \leq \mu_{1,d}(\gamma_i^{-1}\gamma_{i+k}) \leq \frac{k}{2} + D'.
  \]
  This follows from \cref{prop:dgk_sv_gap_omega} and the definition of a tracking sequence.
\end{remark}
\begin{proposition}
  \label{prop:morse_geodesic_gap_growth}
  Fix a Morse gauge $M$, a positive real number $R$, and
  $x_0 \in \Omega$. There exist constants $A,B > 0$ (depending on $M$, $x_0$, 
  and $R$), such that: if $\gamma \in \Aut(\Omega)$ for which
  $[x_0,\gamma x_0]$ is $M$-Morse, $\hil(x_0,\gamma x_0)>B$, and there
  is a finite sequence $\{\eta_n\}$ in $\Aut(\Omega)$ that $R$-tracks
  $[x_0,\gamma x_0]$, then
  \[
    \frac{\mu_{1,2}(\gamma)}{\mu_{1,d}(\gamma)} > A \quad\text{ and }
    \quad \frac{\mu_{d-1,d}(\gamma)}{\mu_{1,d}(\gamma)} > A.
  \]
\end{proposition}
\begin{proof}
  It suffices to only prove the first inequality. The second inequality follows from the first after replacing $\gamma$ with $\gamma^{-1}$, since
  $\mu_{1,2}(\gamma) = \mu_{d-1,d}(\gamma^{-1})$ and $[x_0,\gamma x_0]$ is $M$-Morse if and only if $[x_0, \gamma^{-1}x_0]$ is $M$-Morse.

  Fix $\gamma \in \Aut(\Omega)$ such that $[x_0, \gamma x_0]$ is
  $M$-Morse.  Let $c:[0,\hil(x_0,\gamma x_0)] \to \Omega$ be the unit-speed projective geodesic joining $x_0$ and $\gamma x_0$. Set $L := \lfloor \hil(x_0, \gamma x_0) \rfloor$ and let $\{\eta_n\}_{n=1}^L$ be a sequence that
  $R$-tracks $c|_{[0,L]}$.

   We will first establish a lower bound on $\frac{\mu_{1,2}}{\mu_{1,d}}(\eta_n)$ for $n$ sufficiently large; see \eqref{eqn:bound_for_eta_n} below.
  To do this, observe that for
  any $n, m$, the geodesic segment $[\eta_nx_0, \eta_mx_0]$ is
  $M'$-Morse for a Morse gauge $M'$ depending only on $M$ and $R$. Now we apply \Cref{lem:geodesic_additivity}, taking the compact set
  $K$ in the lemma to be $\overline{B_R(x_0)}$. Let $D$ be the
  constant in \Cref{lem:geodesic_additivity}. Then for all
  $n,n+m \in \{1,\dots, L\}$,
 \begin{align}
 \label{eq:appl_of_mu_12_additive_ineq}
 \mu_{1,2}(\eta_n)+\mu_{1,2}(\eta_n^{-1}\eta_{n+m}) \leq \mu_{1,2}(\eta_{n+m})+D.
 \end{align}
  By \cref{lem:morse_geodesic_uniform_gap}, there exists a constant $k > 0$ so that for every $n=1,\dots,L$, 
  \[
    \mu_{1,2}(\eta_n^{-1}\eta_{n + k}) > 3D.
  \]
   Since our result applies to $\gamma$ for which $\hil(x_0,\gamma x_0)$ is sufficiently large, we can assume that $L \geq k$. Fix any $n \in \{k,k+1,\dots,L\}$. Let $j \in \{1,\dots,   \lfloor L/k \rfloor\}$ be such that $kj \leq n <kj+k$. Then,
  \begin{align*}
  \mu_{1,2}(\eta_n) \geq \mu_{1,2}(\eta_{kj}^{-1} \eta_n) + \mu_{1,2}(\eta_{kj})-D \geq \mu_{1,2}(\eta_{kj})-D. 
  \end{align*}
  But the inequality \eqref{eq:appl_of_mu_12_additive_ineq} further implies that 
  \[ \mu_{1,2}(\eta_{kj}) \ge -D + \mu_{1,2}(\eta_{kj-k}) +
    \mu_{1,2}(\eta_{kj-k}^{-1}\eta_{kj}).
  \]
  We can then conclude (by inducting on $j$, and assuming
  $\eta_0 = \mathrm{id}$) that for all $j$, we have
  \[
    \mu_{1,2}(\eta_{kj}) \ge -jD + \sum_{i=0}^{j-1}
    \mu_{1,2}(\eta_{ki}^{-1}\eta_{ki+k}).
  \]
  By our choice of $k$ this implies
  \[
     \mu_{1,2}(\eta_{kj}) \ge -jD + j(3D) = 2Dj.
  \]
  Thus
  \[
  \mu_{1,2}(\eta_n) \geq \mu_{1,2}(\eta_{kj})-D > 2jD-D \geq jD. 
  \]
  
  On the other hand, \Cref{rem:tracking_implies_mu_1d_gap} implies that
  $\mu_{1,d}(\eta_n) \leq D' + \frac{n}{2}$.

  Set $A:=\frac{D}{4k(1+2D')}$.  Then, for any $n \in \{k,\dots,L\}$, 
  \begin{align}
  \label{eqn:bound_for_eta_n}
  \frac{\mu_{1,2}}{\mu_{1,d}}(\eta_n) \geq \frac{jD}{n+2D'} \geq \frac{D}{k(1+2D')}\frac{kj}{n} \geq 4A \cdot \frac{kj}{kj+k} =4A \cdot \frac{j}{j+1}.
  \end{align}

  We will now upgrade this to the claimed lower bound on $\frac{\mu_{1,2}}{\mu_{1,d}}(\gamma).$ Indeed, by proper discontinuity of the $\Gamma$-action, $\gamma$ differs from $\eta_L$ by a matrix of bounded norm. Hence the claimed lower bound on  $\frac{\mu_{1,2}}{\mu_{1,d}}(\gamma)$ is a consequence of \eqref{eqn:bound_for_eta_n} and \cref{lem:additive_root_bound}. We provide the details below.

  Note that $\hil(\gamma x_0,\eta_L x_0) \leq R+1.$ As $\Aut(\Omega)$ acts properly on $\Omega$,
  there is a compact set $Q=Q(\Omega,\Gamma, x_0,R)$ in $\PGL_d(\Rb)$ such that $\gamma=q \cdot \eta_L$ for some $q \in Q$. Then, by \cref{lem:additive_root_bound}, 
  $$\mu_{1,2}(\gamma) \geq \mu_{1,2}(\eta_L)-C_Q \text{ and } \mu_{1,d}(\gamma) \leq \mu_{1,d}(\eta_L)+C_Q,$$
  for some constant $C_Q$ that depends only on $Q$. Combining this with \eqref{eqn:bound_for_eta_n}, 
  \begin{align*}
      \frac{\mu_{1,2}}{\mu_{1,d}}(\gamma) \geq \frac{\mu_{1,2}(\eta_L)-C_Q}{\mu_{1,d}(\eta_L)+C_Q} \geq \frac{A \cdot \frac{4j}{j+1}-\frac{C_Q}{\mu_{1,d}(\eta_L)}}{1+\frac{C_Q}{\mu_{1,d}(\eta_L)}}.
  \end{align*}
  Since \cref{rem:tracking_implies_mu_1d_gap} implies that $\mu_{1,d}(\eta_L) \geq \frac{L}{2}-D'$, we have that $\frac{C_Q}{\mu_{1,d}(\eta_L)}<\frac{1}{3}\min\{1,A\}$ when $L$ is sufficiently large (say $L \geq L_0$). Let $B>0$ be such that $\hil(x,\gamma x_0) >B$ implies that $L=\lfloor \hil(x_0,\gamma x_0)\rfloor \geq L_0$. Then, $\hil(x_0,\gamma x_0)>B$ implies that
  \begin{equation*}
      \frac{\mu_{1,2}}{\mu_{1,d}}(\gamma) \geq \frac{A \cdot \frac{4j}{j+1}-A \cdot \frac{1}{3}}{1+\frac{1}{3}}\geq \frac{1+\frac{2}{3}}{1+\frac{1}{3}}A>A. \qedhere\end{equation*} 
      \end{proof}

Now we can prove the proposition below, which is a restatement of
\Cref{thm:morse_implies_bdry_reg_and_unif_reg} from the introduction.
\begin{proposition}
  Let $c$ be a projective geodesic in a properly convex domain
  $\Omega$, and let $\{\gamma_n\}$ be a sequence which $R$-tracks $c$
  with respect to a basepoint $x_0 \in \Omega$. If $c$ is $M$-Morse,
  then there are constants $C, N > 0$ (depending only on $M, x_0, R$)
  so that, for all $n \geq 1$ and $m > N$, we have
  \[
    \frac{\mu_{1,2}(\gamma_n^{-1}\gamma_{n +
        m})}{\mu_{1,d}(\gamma_n^{-1}\gamma_{n+m})} > C \text{ and }
    \frac{\mu_{d-1,d}(\gamma_n^{-1}\gamma_{n +
        m})}{\mu_{1,d}(\gamma_n^{-1}\gamma_{n+m})} > C.
  \]
\end{proposition}
\begin{proof}
  Fix an $M$-Morse geodesic $c$ and a tracking sequence $\{\gamma_n\}$
  as in the statement. Since $\{\gamma_n\}$ $R$-tracks $c$, there is
  some Morse gauge $M'$ (depending only on $M$ and $R$) so that for
  any $n, m$, the projective geodesic segment
  $[\gamma_nx_0, \gamma_{n+m}x_0]$ is $M'$-Morse, hence so is the
  projective geodesic $[x_0, \gamma_n^{-1}\gamma_{n+m}x_0]$. So then
  \Cref{prop:morse_geodesic_gap_growth} implies that there are
  positive constants $C, N$ depending only on $M'$ so that if $m > N$,
  then
  $\mu_{1,2}(\gamma_n^{-1}\gamma_{n+m})/\mu_{1,d}(\gamma_n^{-1}\gamma_{n+m})
  > C$, as required.
\end{proof}

\subsection{The partial converse}

The examples below show that the full converse to
\Cref{thm:morse_implies_bdry_reg_and_unif_reg} does not always hold.

\begin{example}
  \label{ex:nondivisible_nonmorse}
  Identify the hyperbolic plane $\Hb^2$ with its projective model in
  $\Pb(\Rb^3)$, so that $\PO(2,1) \subset \PSL(3, \Rb)$ acts by
  isometries. Let $\ell$ be a geodesic in $\Hb^2$. The two tangent
  lines to $\Hb^2$ at the endpoints of $\ell$ meet in unique dual
  point $\ell^*$ to $\ell$; this point is the orthogonal complement to
  $\ell$, with respect to the Minkowski bilinear form defining this
  model of $\Hb^2$.

  Let $\Omega$ be the convex hull of $\Hb^2$ and $\ell^*$, let
  $x_0 \in \ell$, and let $h$ be a loxodromic element in $\PO(2,1)$
  preserving $\ell$, with translation length 1. Then the sequence
  $\{h^nx_0\}$ lies along $\ell$, i.e. $\{h^n\}$ tracks a projective
  geodesic sub-ray of $\ell$. As a subset of $\Hb^2$, the projective
  geodesic $\ell$ is Morse, since $\Hb^2$ is hyperbolic; in particular
  by \Cref{thm:morse_implies_bdry_reg_and_unif_reg} this means that
  the sequence $\{h^n\}$ is strongly uniformly $k$-regular for
  $k = 1, 2$. However, while $\ell$ is still a geodesic in the larger
  domain $\Omega$, it cannot be a Morse geodesic in this domain, as
  both of its endpoints lie in the closure of nontrivial segments in
  $\dee \Omega$ (see \Cref{cor:morse_not_in_segment}).
\end{example}

There are two important points to observe in the previous example:
first, $\Omega$ does not have exposed boundary, and second,
$\{h^n: n \in \Zb\}$ does not divide $\Omega$. In the next example, we
observe that problems can still occur even if we assume that the
domain $\Omega$ is divisible.

\begin{example}
  \label{ex:triangle_nonmorse}
  Consider the projective 2-simplex $\Delta:=\{[x:y:z]~|~x,y,z>0\}$ in $\Pb(\Rb^3)$ and fix
  $x_0:= [1:1:1]$. Let $\Gamma \subset \PSL(3, \Rb)$ be the
  projectivization of the group of diagonal matrices whose entries are
  integer powers of $2$.
  Then $\Gamma$ is an abelian subgroup dividing
  $\Delta$. So if $h \in \Gamma$ is the diagonal matrix $h = {\rm diag}(2,1,1/2)$, then 
  the mapping $n \mapsto h^nx_0$ is a quasi-isometric embedding. The
  sequence $\{h^n\}$ is also strongly uniformly $k$-regular for
  $k = 1, 2$. However, the set of points $\{h^nx_0\}$ cannot
  be in a uniform neighborhood of a Morse geodesic, since $\Delta$ is
  quasi-isometric to the 2-dimensional Euclidean space, which contains no
  Morse geodesics.

  Note that, although the example above fails to be irreducible, one
  can find irreducible divisible domains (indeed, irreducible rank-one
  domains) which contain an embedded copy of this example; we work
  closely with such an example in \cref{sec:morse_counterexample} of
  this paper. So the precise converse to
  \Cref{thm:morse_implies_bdry_reg_and_unif_reg} can still fail even
  in the case where the ambient domain $\Omega$ is divisible and rank
  one.
\end{example}

Despite the existence of the examples above, it is still possible to
prove \cref{thm:1_morse_ggt_morse} -- a partial converse to
\Cref{thm:morse_implies_bdry_reg_and_unif_reg}. We recall the
statement of this partial converse.

\KLPMorseGGTMorse*

\begin{remark}\
  \begin{enumerate}
  \item The sequence given in \Cref{ex:nondivisible_nonmorse} tracks a
    projective geodesic, but the domain $\Omega$ in this example both
    fails to have exposed boundary and also fails to be divisible. We
    do not know if the ``exposed boundary'' assumption is necessary in
    \Cref{thm:1_morse_ggt_morse}; there are no known examples of
    divisible domains without exposed boundary.
  \item \Cref{thm:1_morse_ggt_morse} tells us that the quasi-geodesic
    considered in \Cref{ex:triangle_nonmorse} cannot track any
    projective geodesic, which can also be verified directly.
  \end{enumerate}
\end{remark}

The main idea in the proof of
\Cref{thm:morse_implies_bdry_reg_and_unif_reg} is to use the
characterization of Morse geodesics in divisible domains with exposed
boundary given at the end of \Cref{sec:morse_contracting}. This allows
us to prove a weaker version of the theorem, which does not provide
uniform control over the Morse gauge; then we use a compactness
argument to prove the full (uniform) result.

The non-uniform version of \Cref{thm:1_morse_ggt_morse} is given by
the proposition below.
\begin{proposition}
  \label{prop:1_morse_ggt_morse_nonuniform}
  Let $\Omega$ be a convex divisible domain with exposed boundary, let
  $c$ be a projective geodesic ray in $\Omega$, and let $\{\gamma_n\}$
  be a sequence which tracks $c$. If $\{\gamma_n\}$ is both strongly
  uniformly $1$-regular and strongly uniformly $(d-1)$-regular, then
  $c$ is $M$-Morse.
\end{proposition}
\begin{proof}
  We will prove the contrapositive. We let $c:[0, \infty) \to \Omega$
  be a projective geodesic which is \emph{not} Morse, and let
  $\{\gamma_n\}$ be a sequence tracking $c$. Extend $c$
  (uniquely) to a bi-infinite projective geodesic
  $c:(-\infty, \infty) \to \Omega$ and let $y = c(-\infty)$. 
  
  By \Cref{cor:divisible_c1_extreme_morse}, we know that
  $z = c(+\infty)$ is either forward conically related to a
  non-extreme point in $\dee \Omega$, or else $c(+\infty)$ is forward
  conically related to a non-$C^1$ point in $\dee \Omega$. Since
  $\gamma_n$ tracks $c$, the properness part of the B\'enzecri
  cocompactness theorem tells us that we can use $\gamma_n$ to realize
  the conical relation: there is a subsequence of $\gamma_n$ so that
  $\gamma_n^{-1}(z, y)$ converges to a properly embedded projective
  segment $(z_\infty, y_\infty) \subset \Omega$, so that $z_\infty$ is
  either in the interior of a segment or a non-$C^1$ point. In this
  proof, we will consider the case where $z_\infty$ lies in the
  interior of a nontrivial segment; the case where $z_\infty$ is a
  non-$C^1$ point is nearly identical.

  Now we begin the proof. Let $L$ be a projective line spanned by a
  nontrivial segment in $\bdry$ containing $z_\infty$, and let $P$ be
  the projective 2-plane spanned by $(y_\infty, z_\infty)$ and
  $L$. Fix a basis $\{v_1, v_2, v_3\}$ for $\lift{P}$, so that
  $\Pspan{v_1, v_2} = L$ and $[v_3] = y_\infty$. Then, for each
  $m \in \Nb$, let $\lift{h}_m$ be linear map on $\lift{P}$ defined
  (with respect to the chosen basis) by
  \[
    h_m := \begin{pmatrix}
      e^{-2m}\\
      & e^{-2m}\\
      & & 1
    \end{pmatrix},
  \]
  and let $h_m$ be the corresponding projective transformation on $P$.
  \begin{claim}
    For infinitely many $m \in \Nb$, there exists $g_m \in \PGLdR$ and
    $n = n(m) \in \Nb$ so that each pair
    $(g_m\Omega, g_m\gamma_{n(m)}^{-1}c(n(m) + m))$ lies in a fixed compact subset of the space of pointed domains and
    $g_m|_{P} = h_m$.
  \end{claim}
  \begin{proof}[Proof of Claim]
    Observe that as $m \to \infty$, the sequence of domains
    $h_m(\Omega \cap P)$ converges (after extraction) to some fixed
    properly convex domain in $P$. So, by \cite[Lemma 2.8]{B2003}, we
    may extend each $h_m$ to a linear map $g_m \in \GL(d, \Rb)$
    agreeing with $h_m$ on $\lift{P}$, so that, as $m \to \infty$, a
    subsequence of $g_m\Omega$ converges to a properly convex domain
    $\Omega_\infty$ in $\Pb(\Rb^d)$, containing
    $\lim_{m \to \infty} h_m(\Omega \cap P)$ as a 2-section (see
    \Cref{defn:k_section}).
    
    Now, as $n \to \infty$, we know that (after extracting a
    subsequence) the sequence $\gamma_n^{-1}c(n)$ converges to some
    point in the geodesic $(y_\infty, z_\infty)$. We may fix a
    unit-speed parameterization
    $c_\infty:(-\infty, \infty) \to \Omega$ of this geodesic so that
    $c_\infty(\infty) = z_\infty$ and
    $\gamma_n^{-1}c(n) \to c_\infty(0)$. Then, for any fixed $m$, 
    $\gamma_n^{-1}c(n + m) \to c_\infty(m)$ as $n \to \infty$.

    Fix an auxiliary metric $\dproj$ on $\Pb(\Rb^d)$. Since $\GLdR$
    acts by homeomorphisms on $\Pb(\Rb^d)$, for each fixed $m$ we may
    choose some $\delta$ so that if $\dproj(u, v) < \delta$, then
    $\dproj(g_mu, g_mv) < 1/m$. In particular, for each $m$ we can
    find $n(m)$ so that
    \begin{align}
    \label{eqn:defn_nm_seqn}
      \dproj(g_m\gamma_{n(m)}^{-1}c(n(m) + m), g_mc_\infty(m)) < 1/m.
    \end{align}
    However, by construction, we know that
    $g_m c_\infty(m) = h_m c_\infty(m) = c_\infty(0)$, as $h_m$ acts
    by a translation of Hilbert distance $m$ along
    $(y_\infty,z_\infty)$ in the direction of $y_\infty$. Moreover,
    $c_\infty(0)$ lies in the limit of the 2-sections
    $h_m(P \cap \Omega)$. Thus $c_\infty(0)$ lies in the limiting
    domain $\Omega_\infty$. Then \eqref{eqn:defn_nm_seqn} implies that
    for $m$ large enough, $g_m\gamma_{n(m)}^{-1}c(n(m) + m)$ lies in a fixed compact subset of the domain
    $\Omega_\infty=\lim_{m\to\infty}g_m \Omega$.
  \end{proof}
  
  The last part of the previous claim tells us that the projective
  transformations $g_m$ ``approximate'' the automorphisms 
  $\gamma_{n(m)}^{-1}\gamma_{n(m) + m}$. To be precise, we have:
  \begin{claim}
    There is a fixed compact subset $Q$ of $\PGLdR$ so that, if $g_m$
    and $n(m)$ are as in the previous claim, then
    $g_m\gamma_{n(m)}^{-1}\gamma_{n(m) + m} \in Q$.
  \end{claim}
  \begin{proof}[Proof of Claim]
    Since $\{\gamma_n\}$ tracks $c$,
    $x_m:=\gamma_{n(m) + m}^{-1}c(n(m) + m)$ lies in a fixed compact
    subset of $\Omega$ for all $m$. By the previous claim,
    $g_m \gamma_{n(m)}^{-1}\gamma_{n(m) + m} (\Omega,x_m)=g_m(\Omega,
    \gamma_{n(m)}^{-1}c(n(m)+m))$ lies in a compact subset of the
    space of pointed domains. The claim is then immediate from the
    properness part of the Benz\'ecri compactness \Cref{thm:benzecri}.
  \end{proof}

  Finally, we can show:
  \begin{claim}
    The sequence $\{\gamma_n\}$ is not strongly uniformly
    $(d-1)$-regular.
  \end{claim}
  \begin{proof}[Proof of Claim]
    Since $\{\gamma_n\}$ tracks $c$, \cref{prop:dgk_sv_gap_omega}
    implies that the quantity
    $\mu_{1,d}(\gamma_{n(m)}^{-1}\gamma_{n(m) + m})$ tends to infinity
    as $m \to \infty$. We will show that
    $\mu_{d-1, d}(\gamma_{n(m)}^{-1}\gamma_{n(m) + m})$ is bounded,
    independent of $m$. Owing to the previous claim and
    \Cref{lem:additive_root_bound}, it suffices to show that
    $\mu_{d-1, d}(g_m)$ is bounded.

    To prove this, fix supporting hyperplanes $H_+, H_-$ of $\Omega$
    at $c(\pm \infty)$, and let $H_0 = H_+ \cap H_-$. Using
    \Cref{lem:triples_uniformly_transverse} and
    \Cref{lem:transverse_decomp_compact} (as in the proof of
    \Cref{prop:automorphism_sv_estimates}), we can find a fixed
    compact set $Q \subset \PGL(d, \Rb)$ and elements
    $q_n, q_n' \in Q$ so that any lift of $q_ng_mq_n'$ preserves the
    decomposition
    \[
      c(+\infty) \oplus \lift{H_0} \oplus c(-\infty).
    \]
    Let $\lift{g}_m$ be a lift of $g_n$ agreeing with $\lift{h}_m$ on
    $\lift{P}$, and let $\lift{q}_n$, $\lift{q}_n'$ be lifts of
    $q_n, q_n'$ lying in a fixed compact subset of $\GLdR$. Then,
    \Cref{lem:sigma_12d-1d_estimates} and
    \Cref{lem:additive_root_bound} imply that $\mu_d(\lift{g}_m)$ is
    within uniformly bounded additive error of $-2m$.  In addition,
    since the $e^{-2m}$-eigenspace of $\lift{g}_m$ is at least
    2-dimensional, it follows from the ``minimax'' formula
    \eqref{eq:sv_minimax} for singular values that
    $\sigma_{d-1}(\lift{g}_m) \le e^{-2m}$ and therefore
    $\mu_{d-1,d}(\lift{g}_m) = \mu_{d-1,d}(g_m)$ is uniformly bounded.
  \end{proof}

  This finishes the proof of \Cref{prop:1_morse_ggt_morse_nonuniform}
  in the first case, where $z_\infty$ is not an extreme point. In the
  other case (where $z_\infty$ is not a $C^1$ point) we argue
  similarly, but we instead pick our projective 2-plane $P$ so that
  $z_\infty$ is not a $C^1$ point in $\Omega \cap P$. Then we pick a
  basis $\{v_1, v_2, v_3\}$ so that $v_1$ spans $z_\infty$, and take
  $\lift{h}_m$ to be the sequence of matrices $\lift{h}_m = {\rm diag} (1, e^{2m}, e^{2m}).$ Arguing as in the other case, we see that for a
  sequence of indices $n(m)$, the gap
  $\mu_{1,2}(\gamma_{n(m)}^{-1}\gamma_{n(m) + m})$ is uniformly
  bounded, which implies that $\{\gamma_n\}$ is not strongly uniformly
  $1$-regular.
\end{proof}

\subsubsection{Proof of \Cref{thm:1_morse_ggt_morse}}
  We proceed by contradiction and suppose that there is a sequence of
  projective geodesics $\{c_m\}$ and tracking sequences
  $\{\gamma_{n,m}\}_{n \in \Nb}$, so that
  \[
    \hil(\gamma_{n,m}x_0, c_m(n)) \le R,
  \]
  and each $\{\gamma_{n,m}\}_{n \in \Nb}$ is both strongly uniformly
  $1$-regular and strongly uniformly $(d-1)$-regular (with uniform
  constants), but $c_m$ eventually fails to be $M$-Morse for any given
  Morse gauge $M$. Applying \Cref{prop:contracting_implies_morse} and
  \Cref{prop:slim_contracting}, it then follows that $c_m$ eventually
  fails to be projectively $\delta$-slim, for any given $\delta >
  0$. After extracting a subsequence, we can then assume that each
  $c_m$ fails to be projectively $m$-slim.

  We now argue as in the proof of
  \Cref{lem:halftri_projective_delta_slim}: for each $m$, let
  $x_m, y_m, z_m \in \Omega$ be points such that $x_m, y_m$ lie on the
  image of $c_m$, but $[x_m, y_m]$ is not contained in the
  $m$-neighborhood $N_m([x_m, z_m] \cup [y_m, z_m])$. Then let $w_m$
  be a point in $[x_m, y_m]$ such that
  $\hil(w_m,[x_m, z_m] \cup [y_m, z_m])\geq m$. Choose some $n_m$ so
  that $\gamma_{n_m, m}$ satisfies
  $\hil(\gamma_{n_m,m}^{-1}w_m, x_0) \le R$. After extracting a
  further subsequence, the geodesic rays $\gamma_{n_m, m}^{-1}c_m$
  converge to a bi-infinite projective geodesic $c_\infty$ whose
  endpoints lie in the boundary of a half-triangle. Then by
  \Cref{lem:morse_not_in_halftri}, no sub-ray of $c_\infty$ is Morse.

  For each $m \in \Nb$, define the geodesic sub-ray
  $c_m':[0,\infty) \to \Omega$ of $c_m$ by
  $c'_m(t):=\gamma_{n_m,m}^{-1}c_m(n_m+t)$. Note that the $n_m$-tail
  of the sequence $\{\gamma_{n_m,m}^{-1}\gamma_{n,m}\}_{n \in \Nb}$
  $R$-tracks $c'_m$ with respect to $x_0$. Moreover, as $m\to\infty$,
  $c'_m$ converges to $c_\infty$ uniformly on compact subsets of
  $\Omega$. Then, we can run a diagonalization argument along the
  sequences $\{\gamma_{n_m,m}^{-1}\gamma_{n,m}\}_{n \in \Nb}$ to
  produce a sequence $\{f_n\}$ in $\Gamma$ that tracks
  $c_\infty$. Moreover, $\{f_n\}$ is also strongly uniformly 1-regular
  since the sequences $\{\gamma_{n,m}\}_{n \in \Nb}$ are all strongly
  uniformly 1-regular with uniform regularity constants. Thus, by
  \Cref{prop:1_morse_ggt_morse_nonuniform}, the corresponding sub-ray
  of $c_\infty$ is Morse -- a contradiction. \qed

\section{Regularity at boundary points and singular value gaps}
\label{sec:boundary_regularity_sv}

Our goal in this section is to prove
\Cref{thm:regularity_equiv_regularity}, which connects the linear
algebraic behavior of a tracking sequence in a properly convex domain
$\Omega$ with the regularity of the endpoint of this geodesic in
$\dee \Omega$.

\subsection{Pointwise regularity in convex hypersurfaces}

As we have alluded to previously, the boundary of a properly convex
domain is often nowhere $C^1$, but differentiable in a dense set. We
therefore wish to have a notion of ``$C^\alpha$-regularity'' which
makes sense at a single point in a convex hypersurface. Morally, $x$ is a $C^\alpha$ point if the convex hypersurface $\bdry$ is majorized by the graph of $y \mapsto \norm{y}^\alpha$ near $x$. 

\begin{definition}
  \label{defn:c_alpha_point}
  Let $\Omega$ be a properly convex domain, $x \in \dee \Omega$, and $\alpha > 1$. Fix an Euclidean distance $d$ on an affine chart that contains $\overline{\Omega}$. 
   We say that $x$ is a \emph{$C^\alpha$ point} if there is a neighborhood $U$ of $x$ and a
  constant $C > 0$ so that: for any supporting hyperplane $H$ of
  $\Omega$ at $x$ and any $y \in U \cap \dee \Omega$,
  \begin{equation}
    \label{eq:c_alpha_inequality}
    d(y, H) \le C d(y, x)^\alpha.
  \end{equation}
\end{definition}
\begin{remark} This notion of a $C^\alpha$ point is independent of the choice of the distance
  $d$. Indeed, changing the affine chart or the distance is a bi-Lipschitz map
  in a neighborhood of $x$ and does not impact the definition. We observe further that if the inequality
  \eqref{eq:c_alpha_inequality} holds for some $\alpha > 1$, $\dee \Omega$ has a 
  \emph{unique} supporting hyperplane at $x$, i.e. $x$ is a $C^1$ point.
\end{remark}

One can alternatively define $C^\alpha$ points in $\dee \Omega$ in 
the following equivalent way. Suppose that in some affine chart, the hypersurface $\dee \Omega$ is the graph of a
convex function $f:\Rb^{\dim(\dee \Omega)} \to \Rb$ such that
$x = (0, f(0))$ and there exists a linear map $D_f(0):\Rb^{\dim(\bdry)} \to \Rb$ such that $\ker D_f(0)$ is a supporting hyperplane at $x$. 
We say that $x$ is a $C^\alpha$ point if and only if 
the following limit exists:
\[
  \lim_{y \to 0}\frac{f(y) - f(0) - D_f(0)(y)}{||y||^\alpha}.
\]

Dual to the notion of a $C^\alpha$ point is a \emph{$\beta$-convex
  point}. Just as the $C^\alpha$ property strengthens the condition
that there is a unique supporting hyperplane of $\Omega$ at $x$,
$\beta$-convexity strengthens the condition that $x \in \dee \Omega$
is an extreme point of $\overline{\Omega}$. Morally, $x$ is a $\beta$-convex point if the convex hypersurface $\bdry$ majorizes the graph of $y \mapsto \norm{y}^\beta$ near $x$. 

\begin{definition}
  Let $\Omega$ be a properly convex domain, let $x \in \dee \Omega$,
  and let $\beta < \infty$. We say that $x$ is a \emph{$\beta$-convex
    point} if there is a neighborhood $U$ of $x$ and a constant
  $C > 0$ so that for any $y \in U \cap \dee \Omega$, we have
  \[
    d(y, H) \ge C d(y, x)^\beta.
  \]
\end{definition}

As for $C^\alpha$ regularity, we have an alternative characterization
of $\beta$-convex points. If $U$ is a neighborhood of $0$ $\Rb^n$, and
$f:U \to \Rb$ is a convex function, we say that $f$ is
\emph{$\beta$-convex at} $0 \in U$ if there is a linear map
$A_f:\Rb^n \to \Rb$ so that $f(y) - f(0) > A_f(y)$ for all
$y \in U$, and the limit
\[
  \lim_{y \to 0}\frac{||y||^\beta}{f(y) - f(0) - A_f(y)}
\]
exists. Then a point $x$ in the boundary of a properly convex domain
$\Omega$ is $\beta$-convex if, in coordinates on some (any) affine
chart containing $x$, $\dee \Omega$ is locally the graph of a function
$f:\Rb^{\dim(\dee \Omega)} \to \Rb$ such that $x = (0, f(0))$ and $f$
is $\beta$-convex at $0$.

Note that the linear map $A_f$ defining $\beta$-convexity of the
function $f$ may not be uniquely determined---so in particular a
non-$C^1$ point in $\dee \Omega$ can be a $\beta$-convex
point. However, a $\beta$-convex point in $\dee \Omega$ is always an
extreme point in $\overline{\Omega}$.
\begin{example}
  Consider the graph of the function $f:\Rb \to \Rb$ such that
  $f(x)=x^2$ for $x \geq 0$ and $f(x)=-x$ otherwise. Set $A_f$ to be
  the constant function $0$. Then $f(x)$ is $\beta$-convex at $0$ with
  $\beta=2+\varepsilon$ for any $\varepsilon>0$.

  Now consider a properly convex domain $\Omega \subset \Pb(\Rb^3)$,
  whose boundary in a neighborhood of a point $x \in \bdry$ is
  projectively equivalent to the graph of $f$. Then $x$ is a
  $\beta$-convex point of $\Omega$ that is not $C^1$.
\end{example}

We recall \cref{defn:alpha_beta_point} from the introduction.
\defnbdryregularity*

\subsection{Boundary regularity and uniform regularity}

We will devote the rest of this section to the proof of
\cref{thm:regularity_equiv_regularity} whose statement we recall below. 

\boundaryRegularity*

 The proof is largely an
application of the estimates we proved in \Cref{sec:sv_estimates},
together with a computation in appropriate coordinates (\cref{lem:regularity_inequalities}).

\subsubsection{Choosing coordinates}

For the rest of this section, we will fix the following general
setup. Let $\Omega$ be a properly convex domain, let
$c:[0,\infty) \to \Omega$ be a projective geodesic ray, and let
$\{\gamma_n\}$ be a sequence in $\Aut(\Omega)$ tracking $c$. We will
also denote by $c:(-\infty,\infty) \to \Omega$ the unique bi-infinite
projective geodesic that extends the geodesic ray $c([0,\infty))$. Fix
supporting hyperplanes $H_\pm$ of $\Omega$ at $c(\pm \infty)$ and set
$H_0:=H_+ \cap H_-$.

We fix a coordinate system on the $d$-dimensional affine chart
$A:=\Pb(\Rb^d) \setminus H_-$, chosen so that $c(\infty)$ is the
origin, $H_+$ is the codimension one ``horizontal'' coordinate plane,
and $(c(\infty),c(-\infty))$ is the ``vertical'' ray based at the
origin. 

More formally, let $W_0,W_+,W_- \subset \Rb^d$ be the linear subspaces
such that $\Pb(W_*)=H_*$ for $*\in \{\pm,0\}$. Fix representatives
$v_{\pm} \in \Rb^d$ for $c(\pm \infty)$ in $\Pb(\Rb^d)$, chosen so
that the image of $c$ is the projectivization of
$\{tv_+ + sv_- : s, t > 0\}$. Consider the identification
$\Psi: W_- \to A$ defined by
\[
  \Psi(v)=[v + v_+].
\]
Note that $\Psi$ is a diffeomorphism such that $\Psi(0)=c(\infty)$,
$\Psi(\Rb_{>0} v_-)=c(\Rb)$, $\Psi(W_0)=H_+ \cap A$. So the
decomposition of $W_-=W_0 \oplus [v_-]$ into ``horizontal'' $W_0$ and
``vertical'' $[v_-]$ corresponds to making $A \cap H_+$ ``horizontal''
and $A \cap \Pspan{c(\infty),c(-\infty)}$ ``vertical''. Note that the
map $\Psi^{-1}$ identifies open neighborhoods $U$ of $c(+\infty)$ in
$H_+$ with open subsets of $W_0$ containing the origin.

The set $\Psi^{-1}(\bdry \cap A)$ is a convex hypersurface in $W_-$
passing through the origin in $W_-$, with tangent hyperplane
$W_0$. So, we can make the following definition.
\begin{definition}
  Let $f:W_0 \to \Rb$ be the function such that the image of the
  mapping $x \mapsto \Psi(x,f(x))$ is $\bdry \cap A$.
\end{definition}

\begin{remark}
  \label{rem:exposed_unique_minimum}
  As $\partial \Omega \cap A$ is a convex hypersurface, $f$ is a
  convex function.  The assumption that $c(\infty)$ is a $C^1$ point
  ensures that $f$ is differentiable at $0$. The
  assumption that $c(\infty)$ is an exposed extreme point ensures that $f$ is uniquely minimized at $0$.
\end{remark}

\begin{figure}[h]
  \centering
\begingroup%
  \makeatletter%
  \providecommand\color[2][]{%
    \errmessage{(Inkscape) Color is used for the text in Inkscape, but the package 'color.sty' is not loaded}%
    \renewcommand\color[2][]{}%
  }%
  \providecommand\transparent[1]{%
    \errmessage{(Inkscape) Transparency is used (non-zero) for the text in Inkscape, but the package 'transparent.sty' is not loaded}%
    \renewcommand\transparent[1]{}%
  }%
  \providecommand\rotatebox[2]{#2}%
  \newcommand*\fsize{\dimexpr\f@size pt\relax}%
  \newcommand*\lineheight[1]{\fontsize{\fsize}{#1\fsize}\selectfont}%
  \ifx\svgwidth\undefined%
    \setlength{\unitlength}{224.03468937bp}%
    \ifx\svgscale\undefined%
      \relax%
    \else%
      \setlength{\unitlength}{\unitlength * \real{\svgscale}}%
    \fi%
  \else%
    \setlength{\unitlength}{\svgwidth}%
  \fi%
  \global\let\svgwidth\undefined%
  \global\let\svgscale\undefined%
  \makeatother%
  \begin{picture}(1,0.7865761)%
    \lineheight{1}%
    \setlength\tabcolsep{0pt}%
    \put(0,0){\includegraphics[width=\unitlength,page=1]{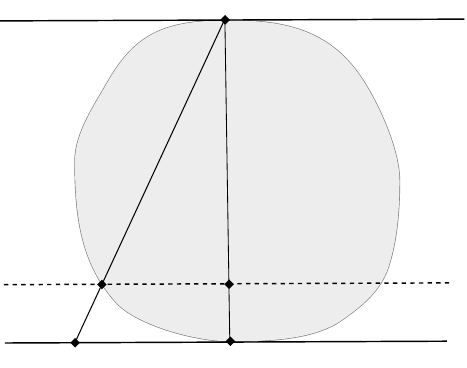}}%
    \put(0.43537885,0.01046206){\color[rgb]{0,0,0}\makebox(0,0)[lt]{\lineheight{1.25}\smash{\begin{tabular}[t]{l}$c(+\infty)$\end{tabular}}}}%
    \put(0.19498829,0.22769457){\makebox(0,0)[lt]{\lineheight{1.25}\smash{\begin{tabular}[t]{l}$y_z$\end{tabular}}}}%
    \put(0.50680093,0.20143695){\makebox(0,0)[lt]{\lineheight{1.25}\smash{\begin{tabular}[t]{l}$c(t_z)$\end{tabular}}}}%
    \put(0.44320754,0.7607684){\color[rgb]{0,0,0}\makebox(0,0)[lt]{\lineheight{1.25}\smash{\begin{tabular}[t]{l}$c(-\infty)$\end{tabular}}}}%
    \put(0.9102346,0.2067452){\makebox(0,0)[lt]{\lineheight{1.25}\smash{\begin{tabular}[t]{l}$H_{y_z}$\end{tabular}}}}%
    \put(0.89438726,0.08476611){\makebox(0,0)[lt]{\lineheight{1.25}\smash{\begin{tabular}[t]{l}$H_+$\end{tabular}}}}%
    \put(0.90064873,0.69870582){\makebox(0,0)[lt]{\lineheight{1.25}\smash{\begin{tabular}[t]{l}$H_-$\end{tabular}}}}%
    \put(0.13920061,0.00620131){\color[rgb]{0,0,0}\makebox(0,0)[lt]{\lineheight{1.25}\smash{\begin{tabular}[t]{l}$z$\end{tabular}}}}%
  \end{picture}%
\endgroup%

  \caption{Illustration of the function $h(x) = t_{\Psi(x)}$ in
    \Cref{defn:height_function}. In this affine chart, the
    intersection $H_0 = H_+ \cap H_-$ is the point at infinity
    corresponding to the ``horizontal direction.''}
    \label{fig:defn_h}
\end{figure}

Next, we define a function $h$ whose level sets determine annular neighborhoods of $c(+\infty)$ in the hyperplane $H_+$. 

\begin{definition}[see \cref{fig:defn_h}]
  \label{defn:height_function}
  For each point $z \in H_+-\{c(\infty)\}$ which is sufficiently close
  to $c(\infty)$, let $y_z$ be the unique point in $\bdry$ such that
  \[
    \Pspan{y_z, c(-\infty)} = \Pspan{z, c(-\infty)}.
  \]
  Let $H_{y_z}$ be the projective hyperplane spanned by $y_z$ and
  $H_0=H_+ \cap H_-$. Then $H_{y_z} \cap c(\Rb)$ is a singleton set
  $\{c(t_z)\}$ for some $t_z \in \Rb$.
  
  Let $U$ be a neighborhood of the origin in $W_0$. We define a
  function $h:U \minus \{0\} \to \Rb$ as follows: for any
  $x \in U \minus \{0\}$, define $h(x) = t_{\Psi(x)}$.
\end{definition}

\begin{remark}\
\begin{enumerate}
\item The intersection $H_{y_z} \cap c(\Rb)$ is always a singleton set
  for $z \in H_+-\{c(\infty)\}$. Indeed, since
  $y_z \in \bdry -\{c(-\infty)\}$, this can only possibly fail if
  $H_{y_z}$ is a supporting hyperplane of $\Omega$ at $y_x$. But if
  this is the case, then the projective segment $[y_z, c(+\infty)]$
  lies in $\bdry$. Since $c(+\infty)$ is extreme and exposed, this
  implies that $y_z=c(+\infty)=z$.
\item We always have $h(U-\{0\})=(a,\infty)$ for some $a \in \Rb$. So,
  by reparameterizing $c$, we can assume that the image of $h$ is
  $(0, \infty)$. Further, as $x \in U-\{0\}$ tends towards $0$, the
  function $h$ tends to $\infty$.
\item The definition of the function $h$ does not require
  $c(\infty)$ to be a $C^1$ point -- it makes sense whenever
  $c(\infty)$ is an extreme and exposed point in $\bdry$. 
\end{enumerate}
\end{remark}

Now we define annular neighborhoods of $c(\infty)$ using $h$.
\begin{definition}
  \label{defn:S_n_family}Suppose $U$ is a sufficiently small
  neighborhood of $0$ in $W_0$ and $h$ is as in
  \cref{defn:height_function} above. We define a family
  $\{S_n\}_{n\in \Nb}$ of subsets of $U$ by $S_n:=h^{-1}([n-1,
  n])$. Note that $\cup_{n\in \Nb}{S_n}=U-\{0\}$.
\end{definition}

\subsubsection{The key lemma} The lemma below gives the key estimates we need for the proof of
\Cref{thm:regularity_equiv_regularity}.
\begin{lemma}
  \label{lem:regularity_inequalities} Suppose $U$ is a sufficiently small neighborhood of $0$. Then, there is a constant $B > 0$ satisfying the following: for any
  $n \in \Nb$ and any $x \in S_n$, we have
  \begin{align}
    \label{eq:vertical_ineq}
    -\mu_{1,d}(\gamma_n) - B \le \log f(x) \le -\mu_{1,d}(\gamma_n) + B,\\
    \label{eq:horizontal_ineq}
    -\mu_{1,d-1}(\gamma_n) - B \le \log||x|| \le -\mu_{1,2}(\gamma_n) + B.
  \end{align}
  In addition, for any $n \in \Nb$, there are points 
  $x_2(n), x_{d-1}(n)$ in $S_n$ satisfying
  \begin{align}
    \label{eq:horizontal_ineq_2}
    \log||x_{d-1}(n)|| \le -\mu_{1,d-1}(\gamma_n) + B,\\
    \label{eq:horizontal_ineq_3}
    \log||x_2(n)|| \ge -\mu_{1,2}(\gamma_n) - B.
  \end{align}
\end{lemma}
\begin{proof}
  Note that there are two disjoint properly convex cones in $\Rb^d$
  that project to $\Omega \subset \Pb(\Rb^d)$, each of which is the
  negative of the other. We fix one of them, denoted by $\wt{\Omega}$,
  and call it the cone above $\Omega$.  For each
  $\gamma_n \in \Aut(\Omega)$, we fix a lift $\wt{\gamma_n}$ in
  $\GLdR$ that preserves $\wt{\Omega}$. By definition
  $\mu_{i,j}(\gamma_n)=\mu_{i,j}(\wt{\gamma_n})$, so our estimates
  will not be affected by switching between $\gamma_n$ and its
  lifts. So, by slight abuse of notation, we will henceforth denote
  the lifts by $\gamma_n$.

  We can use \Cref{lem:triples_uniformly_transverse} and
  \Cref{lem:transverse_decomp_compact} to find a fixed compact subset
  $Q \subset \GLdR$ and a sequence $\{k_n\}$ in $Q$ so that for every
  $n$, the group element $g_n := k_n\gamma_n^{-1}$ preserves the
  decomposition $c(\infty) \oplus \wt{H_0} \oplus
  c(-\infty)$. Then, we can apply
  \Cref{prop:automorphism_sv_estimates} and
  \Cref{lem:additive_root_bound} to see that there is a positive real
  number $D > 0$ so that for every $n$, we have
  \begin{align}
  \label{eqn:singular_value_comparison_application_1}
    |\log(||k_n\gamma_n^{-1}|_{c(\infty)}||) - \mu_d(\gamma_n^{-1})| <
    D,\\
    \label{eqn:singular_value_comparison_application_2}
    |\log(||k_n\gamma_n^{-1}|_{c(-\infty)}||) - \mu_1(\gamma_n^{-1})|
    < D,\\
    \label{eqn:singular_value_comparison_application_3}
    |\log(||k_n\gamma_n^{-1}|_{\wt{H_0}}||) - \mu_2(\gamma_n^{-1})| <
    D,\\
    \label{eqn:singular_value_comparison_application_4}
    |\log(\conorm{k_n\gamma_n^{-1}|_{\wt{H_0}}}) - \mu_{d-1}(\gamma_n^{-1})| <
    D.
  \end{align}
  Let $\lambda_\pm(g_n)$ be the eigenvalues of $g_n$ on
  $c(\pm \infty)$, i.e. $\lambda_{\pm}(g_n)=||k_n\gamma_n^{-1}|_{c(\pm\infty)}||$. Since each group element $g_n$ preserves $H_-$,
  $g_n$ acts by an affine map in our chosen affine chart
  $A = \Pb(\Rb^d) - H_-$. Via the identification $\Psi: W_- \to A$,
  the action of $g_n$ on $A$ (i.e. the map
  $\Psi^{-1}\circ g_n \circ \Psi$) is identified with the linear map
  $\phi(g_n):W_- \to W_-$ given by
  \begin{equation}
    \label{eq:gn_linear_formula}
    \phi(g_n)v = \frac{g_nv}{\lambda_+(g_n)}.
  \end{equation}
  
  Now we analyze the linear map $\phi(g_n)$. With respect
  to the decomposition $W_-=W_0 \oplus [v_-]$, we can write
  $\phi(g_n)$ as
  \begin{equation}
    \label{eq:gn_coordinate_formula}
    \phi(g_n)(x, y) = \left(\frac{g_nx}{\lambda_+(g_n)},
      \frac{\lambda_-(g_n)}{\lambda_+(g_n)}y\right)
  \end{equation}
  where $x \in W_0$ and $y \in [v_-]$. 
  
  Now, for each $n$, consider the
  intersection $g_n\Omega \cap A = k_n\Omega \cap A$. In coordinates given by $\Psi$, $\partial \Omega \cap A$ is the graph of the
  function $f:W_0 \to \Rb$.  Then, in the $\Psi$-coordinates,
  $g_n \partial \Omega \cap A$ is the graph of the convex function
  $f_n:W_0 \to \Rb$ given by
  \[
    f_n(v) =
    \frac{\lambda_-(g_n)}{\lambda_+(g_n)}
    f\left(\frac{g_n^{-1}v}{\lambda_+(g_n^{-1})}\right).
  \]
  This holds because
  $\phi(g_n)(x,f(x))=\left(\frac{g_nx}{\lambda_+(g_n)},\frac{\lambda_-(g_n)}{\lambda_+(g_n)}f(x)\right)$
  and $\lambda_+(g_n)=\frac{1}{\lambda_+(g_n^{-1})}$.  Further, as the
  action of $g_n$ preserves $H_+$, the graph of $f_n$ is a convex
  hypersurface through the origin with a supporting hyperplane
  $\{(w, 0) : w \in W_0\}$ at the origin.

  \begin{claim}
    \label{claim:vertically_bounded}
    There exists a constant $0< C < \infty$ such that: for any
    $n \in \Nb$ and any $x \in S_n$,
    \begin{align*}
      1/C \le \frac{\lambda_-(g_n)}{\lambda_+(g_n)}f(x) < C.
    \end{align*}
  \end{claim}
  \begin{proof}[Proof of Claim] To prove this claim, fix $x \in S_n$ for some $n \in \Nb$. Letting
  $h$ be the function from \Cref{defn:height_function}, the point
  $c(h(x))$ has coordinates $(0,f(x))$ in the coordinates given by
  $\Psi$. Thus, applying the coordinate formula
  \eqref{eq:gn_coordinate_formula} for $\phi(g_n)$, we see that
  $g_nc(h(x))$ has coordinates
  \[
    \left(0, \frac{\lambda_-(g_n)}{\lambda_+(g_n)}f(x)\right).
  \]

  Since $\{\gamma_n\}$ tracks $c$, there is compact set
  $K \subset \Omega$ such that $\gamma_n^{-1}c(n) \in K$ for any
  $n \in \Nb$.  Letting $K' \subset \Omega$ be the closed
  1-neighborhood of $K$ in the Hilbert metric $\hil$, we see that
  $\gamma_n^{-1}c(h(x)) \in K'$ since $h(x) \in [n-1,n]$.

  Then, as $g_n = k_n\gamma_n^{-1}$ for $k_n$ in a compact subset
  $Q \subset \GLdR$, we must have
  \[
    g_nc(h(x)) \in \bigcup_{q \in Q} qK'.
  \]
  
  Each $k_n$ takes $\Omega$ to some domain in a compact family of
  domains that are all supported by the hyperplanes $H_-, H_+$. So, we
  may assume that $Q$ is chosen such that the union
  $\bigcup_{q \in Q} qK'$ lies in the set
  $\Pb(\Rb^d) \minus (H_- \cup H_-)$. This means that $g_nc(h(x))$
  lies in a fixed compact subset of
  $A \minus H_+ = \Pb(\Rb^d) \minus (H_+ \cup H_-)$ which does not
  depend on $n$. As $H_+$ is identified with the horizontal coordinate
  plane in our chosen coordinates on the affine chart $A$, this means
  that the vertical coordinate of $g_nc(h(x))$ is bounded above and
  below, establishing the claim.
    \end{proof}

  Now we explain how the above \cref{claim:vertically_bounded} immediately implies the inequality
  \eqref{eq:vertical_ineq}.  Note that
  $-\mu_{1,d}(\gamma_n)=\mu_{1,d}(\gamma_n^{-1})$. Then, applying
  inequalities \eqref{eqn:singular_value_comparison_application_1} and
  \eqref{eqn:singular_value_comparison_application_2} above, we see
  that there is a constant $B$ (independent of $n$) so that
  \[
    -\mu_{1,d}(\gamma_n) - B \le \log f(x) \le -\mu_{1,d}(\gamma_n) + B,
  \]
  which is the inequality \eqref{eq:vertical_ineq} we wanted to show.
  
  The argument for the proof of the second inequality
  \eqref{eq:horizontal_ineq} is similar. We first claim the following:
  \begin{claim}
    There exists a constant $0<C'<\infty$ such that: for any
    $n \in \Nb$ and any $x \in S_n$,
    \begin{equation}
      \label{eq:horizontal_linear_bound}
      1/C' < ||\phi(g_n)x|| < C'.
    \end{equation} 
  \end{claim}
  \begin{proof}[Proof of Claim] Since each $S_n$ is a compact subset of $W_0$ not containing the
  origin, we can prove this claim by showing that, for any sequence
  $x_n \in S_n$, no subsequence of $\phi(g_n)x_n$ tends towards zero
  or infinity.

  Consider any such sequence $x_n \in S_n$. Let $\Omega_n:=g_n\Omega$. We know that the point with
  coordinates $(x_n, f(x_n))$ lies on the hypersurface $\bdry \cap A$,
  so the points
  \[
    (\phi(g_n)x_n, f_n(x_n))=\phi(g_n)(x_n,f(x_n))
  \]
  lies on the convex hypersurface $\partial \Omega_n \cap A$. Here, we are using the notation of \eqref{eq:gn_coordinate_formula} so that $\phi(g_n)x_n=\frac{g_nx_n}{\lambda_+(g_n)}$ (as $x_n \in W_0$).

  As each domain $\Omega_n$ lies in a fixed compact subset of the
  space of properly convex domains, we may extract a subsequence so
  that the domains $\Omega_n$ converge to a properly convex domain
  $\Omega_\infty$, which is supported by the hyperplanes $H_{\pm}$ at
  $c(\pm \infty)$. Thus, the hypersurfaces $\partial \Omega_n \cap A$
  converge to the convex hypersurface $\partial \Omega_\infty \cap A$.
  The convex functions $f_n:W_0 \to \Rb$ then converge pointwise to a
  convex function $f_\infty:W_0 \to \Rb$, whose graph (in
  $\Psi$-coordinates on $A$) is the hypersurface
  $\partial \Omega_\infty \cap A$.

  After extracting a further subsequence, we can assume that the
  points $$(\phi(g_n)x_n, f_n(x_n))$$ converge to a point in
  $\partial \Omega_\infty$ (a priori, this limit may not lie in the
  affine chart $A$). However, the previous claim gives us
  $\frac{1}{C} \leq f_n(x_n) \leq C$, i.e. there exist uniform upper
  and lower bounds on the vertical coordinates $f_n(x_n)$ of these
  points. In particular, this implies that the limit of the sequence
  $\{(\phi(g_n)x_n,f_n(x_n))\}$ lies on the graph of $f_\infty$, and
  that the limit of the sequence $\{\phi(g_n)x_n\}$ lies in the subset
  $f_\infty^{-1}([1/C, C])$. Since $f_\infty$ is convex,
  $\{\phi(g_n)x_n\}$ lies in a uniformly bounded subset of
  $W_0 \setminus \{0\}$. This proves the claim.
  \end{proof}

  We can now use the claim to show inequality
  \eqref{eq:horizontal_ineq}. For any $x \in W_0$ and any
  $g \in \GL(W_0)$, by definition we have $||x|| \cdot \conorm{g} \le ||g x|| \le ||x|| \cdot ||g||.$ Since $\phi(g_n)$ is a linear map preserving $W_0$, 
  this yields
  \begin{equation}
  \label{eq:W0_norm_conorm}
    ||x|| \cdot \conorm{\phi(g_n)|_{W_0}} \le ||\phi(g_n) x|| \le ||x||
    \cdot ||\phi(g_n)|_{W_0}||, 
  \end{equation}
  for every $n$. Then by applying our formula \eqref{eq:gn_linear_formula} for
  $\phi(g_n)$ we obtain
  \begin{equation}
    \label{eq:norm_bounds}
    ||x||\frac{\conorm{g_n|_{W_0}}}{\lambda_+(g_n)} \le ||\phi(g_n)x||
    \le ||x|| \frac{||g_n|_{W_0}||}{\lambda_+(g_n)}.
  \end{equation}
  Now, the inequalities
  \eqref{eqn:singular_value_comparison_application_3} and
  \eqref{eqn:singular_value_comparison_application_4} tell us that
  $\log||g_n|_{W_0}||$ and $\log\conorm{g_n|_{W_0}}$ are within
  uniformly bounded additive error of $-\mu_{d-1}(\gamma_n)$ and
  $-\mu_2(\gamma_n)$ respectively, and
  \eqref{eqn:singular_value_comparison_application_1} tells us that
  $\log \lambda_+(g_n)$ is within uniformly bounded additive error of
  $-\mu_1(\gamma_n)$. Putting this together with
  \eqref{eq:horizontal_linear_bound} and \eqref{eq:norm_bounds}, we
  see that there is another uniform constant $B > 0$ so that
  \[
    -\mu_{1,d-1}(\gamma_n) - B \le \log||x|| \le -\mu_{1,2}(\gamma_n)
    + B.
  \]
  This establishes inequality \eqref{eq:horizontal_ineq}.

  Now we prove the last two inequalities of \cref{lem:regularity_inequalities}. Observe that each $S_n$ contains
  $h^{-1}(n)$, which is a level set of the convex function $f$. Since
  $f$ is uniquely minimized at the origin (see
  \Cref{rem:exposed_unique_minimum}), this means that each $S_n$
  contains the boundary of a convex open ball in $W_0$, containing the
  origin. Then restrict the continuous function $\phi(g_n)$ to each $S_n$, consider \eqref{eq:W0_norm_conorm}, and recall the definition of $\conorm{\cdot}$ and $\norm{\cdot}$. It is clear that for each $n$,  we can find
  a pair of points $x_2=x_2(n)$ and $x_{d-1}=x_{d-1}(n)$ in $S_n$ so that when $x = x_2$ (resp. 
  $x=x_{d-1}$), the left-hand (resp. right-hand) inequality in \eqref{eq:W0_norm_conorm} is actually an equality. In particular,
  this implies that the corresponding
  inequalities in \eqref{eq:norm_bounds} are equalities when $x=x_2$ or $x_{d-1}$. Then, we
  again use the fact that $\log||g_n|_{W_0}||$ and
  $\log\conorm{g_n|_{W_0}}$ are within bounded error of
  $-\mu_{d-1}(\gamma_n)$ and $-\mu_2(\gamma_n)$ to establish
  \eqref{eq:horizontal_ineq_2} and \eqref{eq:horizontal_ineq_3}.
\end{proof}
We will now use \cref{lem:regularity_inequalities} to finish the proof of \cref{thm:regularity_equiv_regularity}.

\subsection{Proof of \Cref{thm:regularity_equiv_regularity}}

Let $\beta=\beta_0$ and $\alpha=\alpha_0$ where $\alpha_0,\beta_0$ are
as in the statement of \Cref{thm:regularity_equiv_regularity}. We will first prove that 
$\beta<\infty \implies \beta(x, \Omega) \le \beta$ and then show that $\beta(x, \Omega) <\infty \implies \beta \le \beta(x, \Omega)$. This
proves that $\beta = \beta(x, \Omega)$ when either side is finite or
infinite.

Assume first that $\beta < \infty$. For each $u \in U$, choose some
$n$ so that $u \in S_n$. We let
$\beta_n := \frac{\mu_{1,d}(\gamma_n)}{\mu_{1,2}(\gamma_n)}$. We apply
\Cref{lem:regularity_inequalities}: putting the left-hand side of
\eqref{eq:vertical_ineq} together with the right-hand side of
\eqref{eq:horizontal_ineq}, we have
\begin{align*}
  \log f(u) &\ge -\mu_{1,d}(\gamma_n) - B = -\mu_{1,2}(\gamma_n)\beta_n - B \ge \beta_n(\log||u|| - B) - B.
\end{align*}
Hence there is a uniform constant $D > 0$ such that
$f(u) \ge D^{-\beta_n}||u||^{\beta_n}$. Now, fix some 
$\beta<\beta' <\infty$. Since
$\limsup_{n \to \infty} \beta_n = \beta < \beta'$, we have
$||u||^{\beta_n} \ge ||u||^{\beta'}$ for $u$ sufficiently close to
zero. Thus for some $C > 0$ we have $f(u) \ge C||u||^{\beta'}$ in a
small neighborhood of the origin. Hence $\dee \Omega$ is
$\beta'$-convex at $x$. Since $\beta' > \beta$ was arbitrary,
$\beta(x, \Omega) \le \beta$ by definition of $\beta(x,\Omega)$.

Conversely, suppose that $\beta(x, \Omega) < \infty$, and fix
$\beta' > \beta(x, \Omega)$. Now, for each $n \in \Nb$, choose
$u_n \in S_n$ so that the inequality \eqref{eq:horizontal_ineq_3}
holds. We know that $\dee \Omega$ is $\beta'$-convex at $x$, so there
is some $C > 0$ so that $f(u_n) \ge C||u_n||^{\beta'}$. Then we
combine the right-hand inequality in \eqref{eq:vertical_ineq} with
\eqref{eq:horizontal_ineq_3} to obtain
\[
  -\mu_{1,d}(\gamma_n) + B \ge -\beta'(\mu_{1,2}(\gamma_n) + B) +
  \log C.
\]
Replacing $\mu_{1,d}(\gamma_n)$ by $\beta_n \cdot \mu_{1,2}(\gamma_n)$ and
rearranging, we obtain
\[
  \beta' \ge
  \beta_n\left(\frac{\mu_{1,2}(\gamma_n)}{\mu_{1,2}(\gamma_n) +
      B}\right) + \frac{B + \log C}{\mu_{1,2}(\gamma_n) + B}.
\]
Since $x$ is an extreme point in $\Omega$, \Cref{prop:face_sv_gap}
implies that $\mu_{1,2}(\gamma_n)\to\infty$. So the above
implies that $\beta' \ge \limsup_{n \to \infty} \beta_n =
\beta$. Since $\beta' > \beta(x, \Omega)$ was arbitrary, the definition of $\beta(x,\Omega)$ implies that 
$\beta(x, \Omega) \ge \beta$. This concludes the proof that
$\beta = \beta(x, \Omega)$.

The proof of $\alpha = \alpha(x, \Omega)$ is completely symmetric,
using the opposite inequalities in \eqref{eq:vertical_ineq},
\eqref{eq:horizontal_ineq}, and \eqref{eq:horizontal_ineq_2}, and
applying the fact that $x$ is a $C^1$ point to see that
$\mu_{d-1, d}(\gamma_n)$ tends to infinity.
\qed

\section{Boundary regularity does not imply Morse}
\label{sec:morse_counterexample}

In this section, we construct a specific example realizing
\Cref{thm:not_strongly_uniform}: we will find a projective geodesic
ray $c$ in a divisible domain $\Omega$ so that $c$ is tracked
by a sequence $\{\gamma_n\}$ that is uniformly regular, but not \emph{strongly}
uniformly regular. Thus $c$ is not Morse in either the group-theoretic sense or the sense of Kapovich-Leeb-Porti.  But, by
\Cref{thm:regularity_equiv_regularity}, its endpoint $c(\infty)$ in
$\bdry$ is still $C^\alpha$-regular and $\beta$-convex.

\subsection{Convex divisible domains in dimension 3}
\label{sec:description_of_benoist_domains}

The starting point for our construction is a convex divisible domain
$\Omega$ in $\Pb(\Rb^4)$ which is irreducible (meaning it is not
projectively equivalent to the cone over a 2-dimensional domain in
$\Pb(\Rb^3)$), but not strictly convex. Domains of this type were
studied and classified by Benoist \cite{B2006}. Benoist proved that
when $\Gamma$ is a torsion-free discrete group dividing such a domain,
the quotient manifold $M = \Omega / \Gamma$ can be cut along a
nonempty collection of incompressible tori so that each connected
component is homeomorphic to a (non-compact) finite-volume hyperbolic
3-manifold. This means that $\Gamma \simeq \pi_1M$ is a relatively
hyperbolic group, relative to the collection $\Pc$ of fundamental
groups of cutting tori. Moreover, it turns out that the cutting tori
in $\Omega/\Gamma$ lift to properly embedded 2-simplices in $\Omega$
whose stabilizers act by a group of simultaneously diagonalizable
matrices in $\PGL(4, \Rb)$. Thus, each connected component of the
geometric decomposition of $\Omega/\Gamma$ has the structure of a
convex projective manifold.

Benoist also provided explicit constructions for examples of these
domains, using the theory of projective actions of Coxeter groups.
Additional examples were later constructed by Ballas-Danciger-Lee
\cite{BDL2018} and Blayac-Viaggi \cite{BV2023}.

\subsection{Construction}

For the rest of the section, we let $\Omega$ be one of the convex
divisible domains in $\Pb(\Rb^4)$ as above, and let
$\Gamma \subseteq \Aut(\Omega)$ divide $\Omega$. Note that $\Omega$ is
a \emph{rank one} domain \cite{Islam2019}, so the dividing group
$\Gamma$ contains infinitely many \emph{rank one automorphisms} (see
\Cref{sec:rank_one}); these are precisely the automorphisms which do
not preserve any projective geodesic lying in a properly embedded
triangle in $\Omega$. Fix such a rank one automorphism
$\gamma \in \Gamma$, and let $\alpha$ be the closed projective
geodesic in $\Omega / \Gamma$ representing $\gamma$. In addition, fix
a cutting torus $T$ in the geometric decomposition of $M$.

Let us first give an informal sketch of the idea behind
\cref{thm:not_strongly_uniform}. The cyclic subgroup
$\langle \gamma \rangle \subset \Gamma$ gives a Morse geodesic in the
group $\Gamma$ tracking a lift of $\alpha$, along which both
$\mu_{1,2}$ and $\mu_{1,4}$ tend (uniformly) to infinity. On the other
hand, we can find a geodesic $\beta$ in $T$ (not necessarily closed),
corresponding to a sequence of group elements
$a_n \in \pi_1T \subset \Gamma$ along which $\mu_{1,2}$ stays bounded
but $\mu_{1,4}$ goes to infinity. We will produce a projective
geodesic ray $c$ in $M$ that successively follows $\alpha$ and $\beta$
for increasingly longer times. As $c$ spends arbitrarily long times
close to the torus $T$, it fails to be Morse. However, $c$ picks up
enough singular value gaps by looping around $\alpha$ to ensure that
the ratio $\mu_{1,2}/\mu_{1,4}$ stays bounded away from zero in the
limit.

We now turn to the details. First, we note the following.
\begin{lemma}
  \label{lem:singular_values_of_gamma}
  The group element $\gamma \in \Gamma$ representing the geodesic
  $\alpha$ satisfies the following (equivalent) conditions:
  \begin{enumerate}[label=(\roman*)]
  \item\label{item:gamma_morse} The mapping $\Zb \to \Gamma$ given by
    $j \mapsto \gamma^j$ is a Morse quasi-geodesic.
  \item\label{item:gamma_biproximal} $\gamma$ is biproximal, i.e. it
    has unique eigenvalues with maximum and minimum modulus.
  \item\label{item:gamma_gap} There is a positive constant $B_0$ such
    that $\mu_{1,2}(\gamma^j) \ge B_0 \cdot |j|$ and
    $\mu_{3,4}(\gamma^j) \ge B_0 \cdot |j|$ for any $j \in \Zb$.
  \end{enumerate}
\end{lemma}
\begin{proof}
  \cref{prop:rank_one_equiv_morse} implies that both
  \ref{item:gamma_morse} and \ref{item:gamma_biproximal} are
  equivalent to the fact that $\gamma$ is a rank one automorphism. The
  equivalence of \ref{item:gamma_biproximal} and \ref{item:gamma_gap}
  follows from the relationship between the Jordan projection
  $\ell:\GL(d, \Rb) \to \Rb_{\ge 0}^d$ and Cartan projection
  $\mu:\GL(d, \Rb) \to \Rb_{\ge 0}^d$: if
  $\ell_1(g) \ge \ell_2(g) \ge \ldots \ge \ell_d(g)$ are the
  logarithms of the moduli of the eigenvalues of $g \in \GL(d, \Rb)$,
  then $\ell_i(g) = \lim_{n \to \infty} \mu_i(g^n)/n$ (see e.g.
  \cite[Section 2.4]{ggkw2017anosov}).
\end{proof}

Next, let $A_0 \simeq \Zb^2$ be the subgroup of $\Gamma$ identified
with $\pi_1T \subset \pi_1M \simeq \Gamma$. 
\begin{lemma}
  There is a finite-index subgroup $A \subseteq A_0$ so that the
  subgroup $\Gamma' \subset \Gamma$ generated by $\{\gamma\} \cup A$
  is naturally isomorphic to the abstract free product
  $\langle \gamma \rangle * A$.

  Moreover, this subgroup is \emph{strongly quasi-convex} in the sense
  of \cite{Tran}: there exists a function
  $M:\Rb_{\ge 0}^2 \to \Rb_{\ge 0}$ so that any
  $(K_1, K_2)$-quasi-geodesic in $\Gamma$ with endpoints in $\Gamma'$
  lies in the $M(K_1, K_2)$-neighborhood of $\Gamma'$.
\end{lemma}
\begin{proof}
  The first part of the lemma follows from a combination theorem for
  relatively quasi-convex subgroups of relatively hyperbolic groups
  (\cite[Theorem 1.1]{MartinezPedroza}). To apply the combination
  theorem, we need to check that the group $\langle \gamma \rangle$ is
  relatively quasi-convex in $\Gamma$, which follows from
  \Cref{lem:singular_values_of_gamma} \ref{item:gamma_morse}. This
  combination theorem also implies that every parabolic subgroup in
  $\Gamma'$ is a finite-index subgroup of some conjugate of
  $A_0$. Consequently, the second part of the lemma follows from the
  characterization of strongly quasi-convex subgroups in relatively
  hyperbolic groups given by \cite[Theorem 1.9]{Tran}.
\end{proof}

The next step is to construct the geodesic $\beta$ in the torus $T$ we
alluded to previously. We know that the finite-index subgroup
$A \subseteq \pi_1T$ is generated by a pair of commuting
diagonalizable matrices.  So, we can choose a basis for $\Rb^4$ and
find linearly independent vectors $(x_i) \in \Rb^4$, $(y_i) \in \Rb^4$
so that (with respect to this basis) $A$ can be written as the group
\begin{equation}
  \label{eq:a_definition}
  \left\{ {\rm diag}(e^{ux_1 + vy_1}, e^{ux_2 + vy_2}, e^{ux_3
      + vy_3}, e^{ux_4 + vy_4}) \in {\rm PGL}(4,\Rb): u, v \in \Zb\right\}.
\end{equation}

Fix a finite generating set $S_A$ for $A$, and let $|\cdot|_{S_A}$
denote the corresponding word metric on $A$.

\begin{lemma}
  \label{lem:abelian_mu12_bound}
  There exists a constant $C > 0$ and a sequence $\{a_n\}$ in $A$ so
  that $|a_n|_{S_A} = n$ but $\mu_{1,2}(a_n) < C$.
\end{lemma}
\begin{proof}
  The basic idea of the proof is that the group $A$ spans a
  two-dimensional subgroup in the group of positive diagonal matrices
  inside $\PGL(4, \Rb)$, which is isomorphic to $\Rb^3$; since the
  condition $\mu_{1,2}(g) = 0$ is roughly a codimension-1 linear
  constraint on an $\Rb^2$ subgroup containing $A$, there must be
  a one-dimensional ``direction'' in $A$ where $\mu_{1,2}$ coarsely
  vanishes. To make this more precise, we view $A$ as a lattice in a
  subgroup $\hat{A} \subset \PGL(4, \Rb)$ isomorphic to $\Rb^2$; the
  group $\hat{A}$ is defined exactly as in \eqref{eq:a_definition},
  except that the parameters $u, v$ are allowed to vary in $\Rb$
  instead of $\Zb$. We can additionally lift $\hat{A}$ to an
  (isomorphic) subgroup of $\SL(4, \Rb)$, so that every element of
  $\hat{A}$ has positive eigenvalues. After choosing an appropriate
  inner product on $\Rb^4$, we may assume that the eigenspaces of $A$
  (hence of $\hat{A}$) are mutually orthogonal. Then, since the
  eigenvalues of any $a \in \hat{A}$ are positive, the eigenvalues of
  $a$ are precisely the singular values of $a$.

  Let $e_1, \dots, e_4$ be the eigenvectors of $\hat{A}$, and let
  $\lambda_i(a)$ denote the eigenvalue of $a$ on $e_i$ for
  $i = 1,2,3,4$. For each $i$, the mapping $w_i$ given by
  \[
    w_i(a) = \log\lambda_i(a)
  \]
  is an element of the dual $\hat{A}^* \simeq (\Rb^2)^*$. Since $A$ is
  discrete with rank 2, these four dual vectors must span $\hat{A}^*$,
  so their convex hull is a polygon $P$ whose interior contains the
  origin. Pick an edge of this polygon, with endpoints $w_i, w_j$. We
  can pick a vector $v \in \hat{A}^{**} = \hat{A}$ which is positive
  on the chosen edge, but vanishes on the line through the origin in
  $\hat{A}^*$ parallel to the edge. It follows that $v$ achieves its
  maximum on $P$ on both $w_i$ and $w_j$, hence
  $\mu_1(v) = \mu_2(v) = w_i(v) = w_j(v)$ and thus $\mu_{1,2}(v) =
  0$. The same is true for any positive real multiple of $v$. Then,
  since $A$ is a lattice in $\hat{A}$, we can find length-$n$ points
  $a_n \in A$ which are uniformly close to the line
  $\{rv : r \in \Rb_{>0}\}$, giving us the desired sequence.
\end{proof}

The sequence $\{a_n\}$ corresponds to our geodesic $\beta$ in the
torus $T$. Next, we define a sequence in $\Gamma$ that we will use to determine the projective geodesic in
\Cref{thm:not_strongly_uniform}:
\begin{definition}
  \label{defn:wk}
  Define a sequence of words $\{w_k\}_{k \in \Nb}$ as follows:
  \begin{align}
    w_k:=\begin{cases}
     \identity, &\text{ if } k=0\\
      a_1 \gamma \dots a_m \gamma^m, &\text{ if } k=2m \\
      a_1 \gamma \dots a_m \gamma^m a_{m+1}, &\text{ if } k=2m+1
    \end{cases}
  \end{align}
\end{definition}

\subsection{Proof of \cref{thm:not_strongly_uniform}}
Fix a finite generating set $S_A$ for $A$, and extend $S_A \cup \{\gamma\}$ to a finite generating set
$S$ for $\Gamma$. Fix a basepoint $x_0 \in \Omega$ and let
$F: \Gamma \to \Omega$ be the orbit map defined by $F(g)=g x_0$, so
that $F$ is a quasi-isometry with respect to the word metric $d_S$ on
$\Gamma$ induced by $S$. We first prove that if $\{w_k\}$ is the
sequence in \Cref{defn:wk}, then we can extend $\{w_k\}$ to a sequence
that tracks a projective geodesic ray; this ray will be the ray
appearing in \Cref{thm:not_strongly_uniform}.
\begin{lemma}
  \label{lem:wk_along_ray}
  There exists $R>0$ and a projective geodesic ray $[x_0,\xi)$ such
  that for any $k \ge 0$, $\hil(w_k x_0,[x_0,\xi))\leq R$.
\end{lemma}
\begin{proof}
  Fix any $1 \leq k < l$. We first claim that there exists $R>0$,
  independent of $k,l$, such that $\hil(w_k x_0,[x_0,w_l
  x_0])<R$. Before proving this claim, let us explain how this claim
  immediately implies the lemma. Choose a subsequence of $\{w_l x_0\}$
  such that it converges to a point $\xi \in \bdry$. As
  $[x_0,w_lx_0] \to [x_0,\xi)$ uniformly on compact subsets of
  $\Omega$,
  $\hil(w_kx_0,[x_0,\xi)) \leq \limsup_{l \to \infty}
  \hil(w_kx_0,[x_0,w_l x_0])$. Supposing that the claim holds, it is
  immediate that $\hil(w_k x_0,[x_0,\xi)) \leq R$ for all $k$.

  Now we prove the claim. Fix a quasi-inverse $F^{-1}$ for the
  quasi-isometry $F$, and consider the quasi-geodesic
  $F^{-1}([x_0, w_l]) \subset \Gamma$; we will show that for some
  uniform $R$ we have
  \[
    d_S(w_k, F^{-1}([x_0, w_lx_0])) < R.
  \]
  We may assume that the quasi-geodesic $F^{-1}([x_0, w_lx_0])$ joins
  $\identity$ to $w_l$. So, by strong quasi-convexity of $\Gamma'$,
  this quasi-geodesic is within uniformly bounded Hausdorff distance
  of some quasi-geodesic $q$ in the Cayley graph
  $\mathrm{Cay}(\Gamma', S_A \cup \{\gamma\})$. We may assume that $q$
  is continuous (see e.g. \cite[Lemma III.H.1.11]{BH2013}). However,
  observe that if $k < l$ then $w_k$ separates
  $\mathrm{Cay}(\Gamma', S_A \cup \{\gamma\})$ into two components,
  one containing $\identity$ and the other containing $w_l$. In
  particular, $q$ passes through $w_k$, which completes the proof of
  the claim.
\end{proof}

Fix a sequence $\{\gamma_n\}$ tracking the ray $[x_0, \xi)$ from the
previous lemma; we may assume that the sequence $\{w_k\}$ is a
subsequence of $\{\gamma_n\}$. From now on, we make this assumption about $\{\gamma_n\}$. We will show that $\{\gamma_n\}$ is
both uniformly $1$-regular and uniformly $3$-regular. The first step
is the following:
\begin{lemma}
  \label{lem:wk_singular_value_estimates}
  There exist constants $\hat{C},\hat{D}>0$ such that, for
  $i \in \{1, 3\}$, we have
  \[
    \liminf_{n \to \infty}\frac{\mu_{i,i+1}(w_k)}{k^2} > \hat{C}
    ~\text{ and }~ \limsup_{k\to \infty}\frac{\mu_{1,4}(w_k)}{k^2}<
    \hat{D}.
  \]
\end{lemma}
\begin{proof}
  For concreteness, take $i = 1$; the proof when $i = 3$ is
  essentially the same. We first claim that:
  \begin{claim}There is a positive constant $C_0$ such that: for any
    $k
    \in\{2m,2m+1\}$, \begin{align} \label{eqn:telescope_mu_12}\mu_{1,2}(w_k)-\mu_{1,2}(w_{k-2})
      \geq \mu_{1,2}(\gamma^m)-2C_0.\end{align}
\end{claim}
\begin{proof}[Proof of Claim] 
  Since $\{w_k\}$ is a subsequence of a tracking sequence, we may
  prove the claim by applying \cref{lem:geodesic_additivity}. First,
  suppose that $k=2m$. Then, by \cref{lem:geodesic_additivity}, there
  exists a constant $C_0$---independent of $k$---such that:
 \begin{align*}
   \mu_{1,2}(w_k) &\geq \mu_{1,2}(w_{k-1}) + \mu_{1,2}(\gamma^m) - C_0 \\
                  &\geq \mu_{1,2}(w_{k-2})+\mu_{1,2}(a_m)+\mu_{1,2}(\gamma^m)-2C_0.
 \end{align*}
 Since $\mu_{1,2}(a_m) \geq 0$,
 $\mu_{1,2}(w_k)-\mu_{1,2}(w_{k-2}) \geq \mu_{1,2}(\gamma^m)-2C_0$.
 This proves the claim for $k=2m$. The case $k=2m+1$ is similar.
\end{proof}

Using \eqref{eqn:telescope_mu_12} above, we have
$$\mu_{1,2}(w_k) \geq \sum_{j=1}^m (\mu_{1,2}(\gamma^j)-2C_0),$$ for
any $k\in\{2m,2m+1\}$. By \cref{lem:singular_values_of_gamma}, there
is a positive constant $B_0$ such that
$\mu_{1,2}(\gamma^j) \geq B_0 \cdot j$ for any $j \geq 1$. Then for
any $k \in \{2m,2m+1\}$,
$$\mu_{1,2}(w_k) \geq \sum_{j=1}^m (B_0 \cdot j-2C_0)=\frac{B_0}{2}
m(m+1)-2C_0m.$$ Since $2m \leq k \leq 2m+1$,
$\frac{m^2}{k^2} \to \frac{1}{4}$ while $\frac{m}{k^2} \to 0$ as
$k \to \infty$. Thus, there exists a constant $\hat{C}>0$ such
that $$\liminf_{k \to \infty}\frac{\mu_{1,2}(w_k)}{k^2} > \hat{C}.$$
This finishes the proof of the first part.

To prove the estimate for $\mu_{1,4}$, observe that the triangle
inequality implies that, if $k = 2m$, then
\[
  \hil(x_0, w_kx_0) \le \sum_{j=1}^m(\hil(x_0, a_jx_0) + \hil(x_0,
  \gamma^jx_0)).
\]
Since both groups $\langle \gamma \rangle$ and $A$ are
quasi-isometrically embedded in $\Gamma$, and the orbit map for
$\Gamma$ is a quasi-isometry, both terms appearing in the sum above
are uniformly linear in $j$. So there is a uniform constant $D > 0$ so
that the sum is at most $Dk^2$. Then the desired bound follows from
\Cref{prop:dgk_sv_gap_omega}.
\end{proof}

Using the above lemma, we can show:
\begin{lemma}
  Any sequence $\{\gamma_n\}$ which tracks the geodesic ray $[x_0, \xi)$
  is both uniformly $1$-regular and uniformly $3$-regular.
\end{lemma}
\begin{proof}
  We know that each point $\{w_kx_0\}$ lies within uniformly bounded
  distance of $[x_0, \xi)$, and that
  $\hil(w_kx_0, w_{k+1}x_0) = O(k)$. So, as $\{w_k\}$ is an unbounded subsequence of $\{\gamma_n\}$,
  it follows that for each $n$ there is some $k = k(n) \in \Nb$ so
  that
  \[
    \hil(\gamma_nx_0, w_kx_0) = O(k).
  \]
  Then by \Cref{prop:dgk_sv_gap_omega} we also have
  $\mu_{1,4}(\gamma_n^{-1}w_k) = O(k)$. So, by 
  \Cref{lem:additive_root_bound}, 
  \begin{align*}
    |\mu_{1,2}(\gamma_n) - \mu_{1,2}(w_k)| = O(k), ~~\text{ and }~~ |\mu_{1,4}(\gamma_n) - \mu_{1,4}(w_k)| = O(k).
  \end{align*}
  So, it follows from the previous lemma that, $\liminf_{n \to \infty}
    \frac{\mu_{1,2}(\gamma_n)}{\mu_{1,4}(\gamma_n)} \ge
    \frac{\hat{C}}{\hat{D}} > 0,$
  i.e. $\{\gamma_n\}$ is uniformly $1$-regular. The proof for
  $3$-regularity is similar.
\end{proof}

\begin{corollary}
    Let $[x_0,\xi)$ be the geodesic ray in \cref{lem:wk_along_ray} that $w_k x_0$ embeds along. Then $\xi$ is $C^\alpha$-regular and $\beta$-convex for some $\alpha>1$ and $\beta<\infty$. 
\end{corollary}
\begin{proof}
  Since $\Omega$ has exposed boundary (see \cite{B2006}), \cref{thm:regularity_equiv_regularity} applies and the previous
  lemma implies the result.
\end{proof}

\begin{lemma}
  The sequence $\{\gamma_n\}$ is not strongly uniformly $1$-regular.
\end{lemma}
\begin{proof}
  Observe that $\{w_k\}$ is a subsequence of
  $\{\gamma_n\}$ and since strong uniform regularity passes to
  subsequences, it suffices to prove the claim for $\{w_k\}$. Suppose $k=2m+1$. Then $w_{k-1}^{-1}w_k=a_{m+1}$. Recall that
  $\mu_{1,2}(a_{m+1})$ is uniformly bounded while
  $\mu_{1,4}(a_{m+1}) \to \infty$ linearly in $m$. Then
  $\frac{\mu_{1,2}(w_{k-1}^{-1}w_k)}{\mu_{1,d}(w_{k-1}^{-1}w_k)} \to
  0$. So $\{w_k\}$ is not strongly uniformly 1-regular.
\end{proof}

\bibliographystyle{alpha}
\bibliography{geom}

\end{document}